\RequirePackage{fix-cm}
\documentclass[smallextended, envcountsect]{svjour3}
\smartqed  
\usepackage{graphicx}
\usepackage{epstopdf}
\usepackage{amsmath, amssymb, amsfonts}
\usepackage{algorithmic}
\usepackage{cases}
\usepackage{bm}
\usepackage{tikz}
\usetikzlibrary{calc}
\usetikzlibrary{arrows}
\usepackage{multirow}
\usepackage{tabularx}
\usepackage{booktabs}
\usepackage{placeins}
\usepackage{cite}
\usepackage{verbatim}
\usepackage{color}
\usepackage{subfigure}
\usepackage{epsfig}
\usepackage{mathrsfs}
\usepackage{stmaryrd}
\usepackage{amsopn}

\numberwithin{equation}{section}
\usepackage[colorlinks=true, citecolor=blue, linkcolor=blue, urlcolor=blue]{hyperref}
\numberwithin{equation}{section}
\allowdisplaybreaks[4]

\begin{document}

\title{High-order parametric local discontinuous Galerkin methods for anisotropic curve-shortening flows}

\subtitle{}   

\titlerunning{High-order LDG methods for anisotropic curve-shortening flows}
\author{Xiuhui Guo \and Wei Jiang \and Chunmei Su}
\authorrunning{X. Guo, W. Jiang and C. Su}

\institute{
  Xiuhui Guo \at
    QiuZhen College, Tsinghua University, Beijing 100084, China \\
    \email{gxh23@mails.tsinghua.edu.cn}
  \and
  Wei Jiang \at
    School of Mathematics and Statistics, Hubei Key Laboratory of Computational Science, Wuhan University, Wuhan 430072, China \\
    \email{jiangwei1007@whu.edu.cn}
  \and
  Chunmei Su \at
    Yau Mathematical Sciences Center, Tsinghua University, Beijing 100084, China; \\
    and Beijing Institute of Mathematical Sciences and Applications, Beijing 101408, China \\
    \email{sucm@tsinghua.edu.cn}
}

\date{Received: date / Accepted: date}

\maketitle

\begin{abstract}
We propose a family of high-order local discontinuous Galerkin (LDG) methods, built on a parametric representation and coupled with a semi-implicit backward Euler time discretization, for isotropic and anisotropic curve-shortening flows. The spatial LDG formulation introduces auxiliary variables and carefully designed numerical fluxes which inherit the underlying variational structure. We prove unconditional energy dissipation for the semidiscrete scheme and establish energy dissipation and well-posedness for the fully discrete scheme under certain assumptions. For $P^k$ approximations, the LDG method achieves high-order spatial convergence; extensive numerical experiments confirm optimal $(k+1)$-order accuracy when the surface energy is isotropic or weakly anisotropic. Compared to classical parametric finite element methods (PFEM), the proposed LDG schemes do not need to rely on good mesh distributions or auxiliary symmetrized surface energy matrices for strongly anisotropic surface energy cases, and remain numerically stable on severely degraded meshes that typically cause PFEMs
failure. This intrinsic stability enables effective capture of complex geometric evolution and sharp corner singularities produced by strong anisotropy. The approach thus provides a flexible and reliable framework for the numerical simulation of a broader class of geometric flows.
\keywords{Local discontinuous Galerkin method \and Geometric flows \and Anisotropic surface energy \and Energy dissipation \and High-order accuracy}
\subclass{74H15 \and 74S05 \and 74M15 \and 65M60}
\end{abstract}

\section{Introduction}
\label{sec:intro}

Geometric flows (e.g., mean-curvature flow, surface diffusion) are widely used as mathematical models to describe geometric shape evolution. They were formally introduced by Mullins in materials science
to model grain boundary motion and thermal grooving~\cite{Mullins1956,Mullins1957}. Nowadays, geometric flows have attracted growing attention because of their broadening applications in
image processing, materials science, data science and solid-state physics---for example, in describing microstructuce evolution in solid materials~\cite{Cahn1991,Fonseca2014}, solid-state dewetting of thin films~\cite{Bao2017330,Jiang2020}
and image smoothing~\cite{Clarenz2000}.

In this paper, we consider the evolution of planar curves driven by the anisotropic curve-shortening flow (CSF), i.e., the curve case of anisotropic mean curvature flow.
Let $\Gamma=\Gamma(t)$ be a closed planar curve parameterized by $\boldsymbol{X}(s,t)=\left(x(s,t),y(s,t)\right)^T$, where $t$ denotes time and $s$ is the arc-length parameter. The anisotropic CSF evolves $\Gamma$ according to the geometric law:
\begin{equation}
\label{1.1}
\partial_{t} \boldsymbol{X}=-\mu\boldsymbol{n},
\end{equation}
where $\boldsymbol n=(-\sin\theta,\cos\theta)^T$ is the outward unit normal vector and
$\theta\in[-\pi,\pi]$, with $-\pi$ and $\pi$ identified, is measured counterclockwise from the positive vertical axis to $\boldsymbol n$; see also \cite[Eq.~(1.4)]{Li2021}, and chemical potential (weighted mean curvature) $\mu:=\mu(s,t)$ is defined as \cite{Bao2017,Jiang2016,Jiang2019}:
\begin{equation}
\label{1.2}
\mu=(\gamma(\theta)+\gamma''(\theta))\kappa.
\end{equation}
 Here $\kappa=-\partial_{ss}\boldsymbol{X}\cdot\boldsymbol{n}$ represents the curvature of $\Gamma(t)$ and $\gamma(\theta)\in  C^2([-\pi,\pi])$ is the positive, $2\pi$-periodic surface energy density satisfying $\gamma(-\pi)=\gamma(\pi)$ and $\gamma'(-\pi)=\gamma'(\pi)$. The initial condition for \eqref{1.1} is given as
\begin{equation}
\label{1.3}
\boldsymbol{X}(s,0)=\boldsymbol{X}_0(s)=(x_0(s),y_0(s))^T,\quad\quad 0\le s\le L_0,
\end{equation}
where $L_0$ is the length of the initial curve $\Gamma_0:=\Gamma(0)$.

Parameterization by arc-length yields the usual expression:
\begin{equation}
\label{1.4}
\boldsymbol{n}=-\partial_s\boldsymbol{X}^\perp,
\end{equation}
where $(a,b)^\perp=(b,-a)$.
Denote the total interfacial energy by
\begin{equation}
\label{1.5}
W_c(t)=\int_{
\Gamma(t)}\gamma(\theta)ds=\int_0^{
L(t)}\gamma(\theta)ds,\quad t\ge 0,
\end{equation}
where $L(t)=\int_{\Gamma(t)}1 ds$ denotes the length of $\Gamma(t)$. Following the variational framework of \cite{Bao2017,Bao2017330,Li2021}, one readily shows the energy dissipation law for \eqref{1.1}-\eqref{1.2}:
\begin{equation}
\label{1.6}
    \frac{d}{dt}W_c(t)=-\int_0^{L(t)}\mu^2ds\le 0,\quad t\ge 0.
\end{equation}
Following the approach in \cite{Li2021}, the weighted mean curvature \eqref{1.2} can be written in the divergence form as
\begin{equation}
\label{1.8}
\mu\boldsymbol{n}=-\partial_{s}(G(\theta)\partial_s\boldsymbol{X}),
\end{equation}
where the surface energy matrix $G(\theta)$ is defined as
\begin{equation}
\label{1.9}
G(\theta)=\begin{pmatrix}
\gamma(\theta)& -\gamma'(\theta)\\
\gamma'(\theta) & \gamma(\theta)
\end{pmatrix}.
\end{equation}
Consequently, the anisotropic CSF \eqref{1.1}-\eqref{1.2} can be reformulated as \cite{Barrett200729, Li2021}
\begin{equation}
\label{1.10}
\left\{
\begin{aligned}
&\boldsymbol{n}\cdot\partial_{t} \boldsymbol{X}=-\mu,\\
&\mu\boldsymbol{n}=-\partial_{s}(G(\theta)\partial_s\boldsymbol{X}).
\end{aligned}
\right.
\end{equation}
The isotropic case is given by $\gamma(\theta)\equiv 1$, in which
$\mu$ coincides with $\kappa$, and \eqref{1.10} reduces to the BGN formulation \cite{Barrett200729}.
Similar to the BGN formulation, \eqref{1.10} prescribes only the normal velocity of the evolving curve. The tangential component remains undetermined at the continuous level, since adding an arbitrary tangential velocity changes only the parametrization and leaves the geometric curve unchanged. Thus, the continuous law determines the evolution of
$\Gamma(t)$, but does not determine a unique parametrization
$\boldsymbol X(\rho,t)$. 
We characterize the surface energy anisotropy according to the surface stiffness $\widetilde{\gamma}(\theta)= \gamma(\theta) + \gamma''(\theta)$~\cite{Bao2017330,Bao2017}. Particularly, it is weakly anisotropic when $\widetilde{\gamma}(\theta)>0$ for all $\theta$; it is strongly anisotropic when there exists some  $\theta \in [-\pi, \pi]$ such that $\widetilde{\gamma}(\theta) <0$; and the marginal case when $\min\limits_{\theta} \widetilde{\gamma}(\theta)=0$.
For example, a commonly used $l$-fold surface energy is defined as~\cite{Bao2017330}
\begin{equation}
\label{eq:surface_energy}
\gamma(\theta) = 1 + \beta\cos(l\theta).
\end{equation}
We illustrate the following regimes for the above surface energy: the isotropic case when $\beta = 0$; weakly anisotropic when $0 < \beta < \frac{1}{l^2-1}$; strongly anisotropic when $\beta > \frac{1}{l^2-1}$,
with $\beta = \frac{1}{l^2-1}$ marking the marginal transition.

In recent decades, significant progress has been made in the numerical simulation of both isotropic and anisotropic CSFs. For example, parametric finite element methods (PFEM) were designed by Dziuk \cite{Dziuk1991} for solving mean curvature flows and were further developed by Barrett, Garcke and N\"urnberg~\cite{Barrett2007222,Barrett200729}. The resulting family of schemes---well-known as the BGN scheme---is widely used because it could preserve good mesh quality by incorporating an implicit tangential velocity in the discrete scheme  \cite{Bao2017330, Bao2021,Barrett2008227,Barrett200828,Barrett2019}.  
 For anisotropic flows, the original BGN stability analysis establishes unconditional energy stability for particular classes of anisotropies, including suitable finite sums of elliptic norms~\cite{Barrett200828}, but does not cover general anisotropies. In related work on anisotropic surface diffusion, Li and Bao~\cite{Li2021} introduced the surface energy matrix $G(\theta)$ in \eqref{1.9}, whose symmetric part is positive definite, and proposed an energy-stable PFEM (ES-PFEM) under restrictive conditions on the anisotropy. Bao et al.~\cite{Bao202361,Bao202345} subsequently developed symmetrized PFEMs with stabilized surface energy matrices for both weak and strong anisotropies. Energy stability alone, however, does not ensure good mesh quality, and substantial growth in the mesh ratio under strong anisotropy may lead to severe loss of robustness or failure of the scheme.

 For anisotropic CSF, Deckelnick and N\"urnberg~\cite{DeckelnickNuernberg2023Anisotropic} developed a finite element approximation and established an optimal semidiscrete $H^1$-error bound under smoothness and strict convexity assumptions on the energy density. They also proved unconditional stability for fully discrete schemes. Their subsequent work introduced second-order-in-time schemes for CSF and curve diffusion, with optimal error bounds for CSF~\cite{DeckelnickNuernberg2026SecondOrder}. The anisotropic formulations discussed above nevertheless use piecewise linear spatial approximations. Higher-order isoparametric finite element methods have been developed for isotropic flows~\cite{Garcke2025}, but combining high-order spatial accuracy with structure preservation and numerical robustness in strongly anisotropic regimes remains a challenge.


The local discontinuous Galerkin (LDG) framework offers a promising route toward this combined goal. By introducing auxiliary variables, reducing high-order differential equations to first-order systems, and discretizing with discontinuous finite elements and carefully choosing numerical fluxes, LDG methods could deliver high-order accuracy, flexibility, and provable stability properties. Originating from Cockburn and Shu \cite{Cockburn1998} and motivated by Bassi and Rebay \cite{Bassi1997}, LDG schemes have been successfully applied to a variety of nonlinear problems including KdV equation \cite{Yan2002}, Schr\"odinger equation \cite{Xu2005205}, viscous Burgers equation \cite{Hutridurga2023}, Allen-Cahn equation \cite{Wang2022}, Cahn-Hilliard equation \cite{Xia2007, Guo2014}, and geometric flows based on the graph representation \cite{Xu2009,Ji2012}. Recently, Wimmer et al. \cite{Wimmer2025} proposed a structure-preserving LDG scheme for the Cahn-Hilliard equation with efficient time adaptivity. Such advances highlight the versatility and robustness of the LDG method in handling complex high-order nonlinear problems.

Building on the variational formulation similar to PFEMs, we propose a class of parametric LDG methods for solving the isotropic/anisotropic CSF. The main challenges come from how to make use of the variational structure inherited from
the continuous geometric flow and the design of numerical fluxes which can simultaneously guarantee energy dissipation, high-order convergence, and numerical stability.
The main contributions of the paper are summarized as follows:
\begin{itemize}
 \item \textbf{Unified LDG formulation.}
    We develop an LDG scheme for anisotropic
    curve-shortening flows. For strongly anisotropic energies,
    the scheme uses the surface energy matrix $G(\theta)$
    without the additional matrix stabilization introduced
    in the symmetrized PFEMs~\cite{Bao202361,Bao202345}.

    \item \textbf{Semi-discrete energy dissipation.}
    The numerical fluxes yield an exact energy identity
    that guarantees unconditional energy dissipation
    at the semi-discrete level.

    \item \textbf{Fully discrete analysis.}
    For the semi-implicit backward Euler scheme, we establish energy dissipation under a condition and unique solvability under the stated     positivity and spanning assumptions.

     \item \textbf{High-order spatial accuracy.}
    Numerical experiments show spatial convergence of
    order $k+1$ with $P^k$ elements in the isotropic
    and weakly anisotropic cases.

    \item \textbf{Robust numerical performance.}
    In the strongly anisotropic tests, the scheme produces
    stable evolutions on highly distorted meshes, and the
    computed equilibrium shapes agree closely with their
    theoretical counterparts.

\end{itemize}
These features collectively yield a type of high-order, structure-preserving, and numerically robust methods for simulating the anisotropic CSF, to our knowledge, which
fills a gap in the literature where the high-order accuracy and the stability for anisotropic surface energy have not previously been achieved simultaneously.

The remainder of the paper is organized as follows. Section \ref{sec:ldg_method} presents the LDG spatial discretization and proves unconditional energy dissipation for the semi-discrete scheme. Section~\ref{sec:fully_discrete}  introduces a semi-implicit backward Euler time discretization, establishes well-posedness and conditional energy dissipation for the fully discrete scheme, and extends the method to anisotropic area-preserving curve-shortening flow. Section \ref{sec:numerical} provides numerical experiments that validate high-order convergence, energy dissipation, and robustness in the challenging anisotropic tests. Finally, some conclusions are drawn in Section \ref{sec:conclusion}.

\section{The LDG method for curve-shortening flow}
\label{sec:ldg_method}

This section details the application of the LDG method to the anisotropic CSF. For convenience, we introduce a time-independent parameter $\rho \in I$, where $I = [0,1]$ denotes the periodic unit interval. The curve $\Gamma(t)$ is then parameterized by
\begin{equation*}
\Gamma(t) := \boldsymbol{X}(\rho, t) = \big(x(\rho,t), y(\rho,t)\big)^T \colon \: I \times [0,T] \to \mathbb{R}^2.
\end{equation*}
The arc length parameter $s$ is given by
$s(\rho)=\int^{\rho}_0\left|\partial_\rho\boldsymbol{X}\right|d\rho$, which implies $\partial_\rho s=\left|\partial_\rho\boldsymbol{X}\right|$. Using this, the governing equations, originally presented in \eqref{1.10}, can be rewritten as:
\begin{equation}\label{2.1}
\left\{
\begin{aligned}
&\boldsymbol{n}^\star\cdot\partial_{t} \boldsymbol{X}=-Q\mu,\\
&\mu\boldsymbol{n}^\star=-\partial_{\rho}\left(G(\theta)\frac{\partial_\rho\boldsymbol{X}}{Q}\right),
\end{aligned}
\right.
\end{equation}
where $\boldsymbol{n}^\star=(-\partial_\rho y,\partial_\rho x)^T$ and $Q=|\partial_\rho\boldsymbol{X}|$. Subsequently, we will construct the LDG method for the system \eqref{2.1}.

\subsection{Basic notations}
\label{subsec:notations}

Let $N$ be a positive integer. We define a grid by the points $0 = \rho_{1/2} < \rho_{3/2} < \dots < \rho_{N+1/2} = 1$. The domain $I$ is partitioned into $N$ subintervals (elements) $I_j = [\rho_{j-1/2}, \rho_{j+1/2}]$ for $j = 1, \ldots, N$. For simplicity, we assume a uniform mesh, though this assumption is not essential for the method. We denote the mesh size by $h = 1/N$.

We then define the following finite-dimensional function spaces:
\begin{flalign}
V_{h}=\left\{v \in L^{2}(I):\left.v\right|_{I_j} \in \mathcal{P}^{k}\left(I_j\right), \: j=1,\ldots,N\right\},\quad \boldsymbol{W}_h=[V_h]^2,\nonumber
\end{flalign}
where  $\mathcal{P}^{k}({I_j})$ denotes the space of polynomials of degree at most $k$ on the element $I_j$. To define the numerical fluxes at element interfaces, we introduce the notation $v^{\pm}|_{\rho_{j+1/2}} = \lim\limits_{\varepsilon \to 0^+} v(\rho_{j+1/2} \pm \varepsilon)$ for the left and right limits, respectively. Particularly, we denote $v^{-}|_{\rho_{1/2}}=v^-|_{\rho_{N+1/2}}$ and $v^+|_{\rho_{N+1/2}}=v^+|_{\rho_{1/2}}$ due to the periodic boundary condition.

\subsection{The LDG method}
\label{subsec:ldg_formulation}

This section presents the formulation of the LDG scheme for the system \eqref{2.1}. By introducing the auxiliary variables $\boldsymbol{q} = (q_1, q_2)^T = \partial_\rho \boldsymbol{X}$ and $\bm{\xi} = G(\theta) \boldsymbol{q}/Q$, the system \eqref{2.1} can be reformulated as:
\begin{subequations}
\label{2.2}
\begin{align}
&\boldsymbol{n}^\star\cdot\partial_{t} \boldsymbol{X}+Q\mu=0,\\
&\mu\boldsymbol{n}^\star+\partial_{\rho}\bm{\xi}=0,\\
&\bm{\xi}-G(\theta)\boldsymbol{q}/Q=0,\label{2.2c}\\
&\boldsymbol{q}-\partial_{\rho} \boldsymbol{X}=0,
\end{align}
\end{subequations}
where
\begin{equation}\label{2.3}
    Q=|\boldsymbol{q}|,\quad \theta=\operatorname{atan2}(q_2,q_1),\quad
   \boldsymbol{n}^\star=-\boldsymbol{q}^{\perp}=(-q_2,q_1)^T.
\end{equation}
Here, we assume $Q \neq 0$, which is a reasonable assumption since $Q$ appears in the denominator of equation \eqref{2.2c}. Henceforth, unless otherwise specified, $\mu, \boldsymbol{X}, \boldsymbol{q}, \bm{\xi}$ will denote their numerical approximations. The LDG scheme for \eqref{2.2} seeks $\mu \in V_h$, and $\boldsymbol{X}, \boldsymbol{q}, \bm{\xi} \in \boldsymbol{W}_h$ such that for each element $I_j$ ($j=1, \ldots, N$) and for all test functions $\eta \in V_h$, $\bm{\phi}, \boldsymbol{r}, \bm{\vartheta} \in \boldsymbol{W}_h$, the following equations are satisfied:
\begin{subequations}
\label{2.5}
\begin{align}
&\int_{I_{j}} \partial_{t} \boldsymbol{X}\cdot
\boldsymbol{n}^\star\eta d \rho+\int_{I_{j}}Q\mu\eta d \rho=0,\label{2.5a}\\
&\int_{I_{j}} \mu\boldsymbol{n}^\star\cdot\bm{\phi} d \rho-\int_{I_{j}} \bm{\xi}\cdot\partial_{\rho}\bm{\phi} d \rho+\widehat{\bm{\xi}}\cdot\bm{\phi}^-|_{\rho_{j+1/2}}-
\widehat{\bm{\xi}}\cdot\bm{\phi}^+|_{\rho_{j-1/2}}=0,\label{2.5b}\\
&\int_{I_{j}}\bm{\xi}\cdot\boldsymbol{r} d\rho
-\int_{I_{j}}\left(G(\theta)\frac{\boldsymbol{q}}{Q}\right)\cdot\boldsymbol{r} d\rho=0,\label{2.5c}\\
&\int_{I_{j}} \boldsymbol{q}\cdot \bm{\vartheta} d \rho+\int_{I_{j}} \boldsymbol{X}\cdot\partial_{\rho} \bm{\vartheta} d \rho-\widehat{\boldsymbol{X}}\cdot \bm{\vartheta}^-|_{\rho_{j+1/2}}+\widehat{\boldsymbol{X}}\cdot \bm{\vartheta}^+|_{\rho_{j-1/2}}=0,\label{2.5d}
\end{align}
\end{subequations}
where $Q$ and $\boldsymbol{n}^\star$ are computed by \eqref{2.3}.

The numerical fluxes (denoted by `hat' terms) in \eqref{2.5} are chosen as follows:
\begin{equation}
\label{2.6}
\widehat{\boldsymbol{X}}=\boldsymbol{X}^+,\quad
\widehat{\bm{\xi}}=\bm{\xi}^-+\alpha(\boldsymbol{X}^+-\boldsymbol{X}^-),
\end{equation}
where $\alpha>0$ is a penalty coefficient.
\begin{remark}
Alternating numerical fluxes include
\[
\widehat{\boldsymbol{X}}=\boldsymbol{X}^+,
\qquad
\widehat{\boldsymbol{\xi}}=\bm{\xi}^-.
\]
These fluxes select flux-orientation-dependent tangential parametrizations. Motivated by \cite{Antonietti2015,Cockburn199835}, we add to $\widehat{\bm\xi}$ a penalty term involving the jump of $\boldsymbol X$, which provides the interface control used in the well-posedness proof.
\end{remark}

For the energy analysis, let $\mathcal E_h$ denote the set of
periodic mesh interfaces, with each interface counted once.
We use the notation
\[
[\boldsymbol v]_e
=\boldsymbol v^+|_e-\boldsymbol v^-|_e,
\qquad
(\boldsymbol u,\boldsymbol v)_h
=
\sum_{j=1}^N\int_{I_j}
\boldsymbol u\cdot\boldsymbol v\,d\rho.
\]
The total discrete energy $W_c^h(t)$ for the solution of the semi-discrete system \eqref{2.5}-\eqref{2.6} is defined as:
\begin{equation}\label{2.7}
W_c^h(t)=(\gamma(\theta), Q)_h+ \frac{\alpha}{2} \sum_{e\in\mathcal E_h}|[\boldsymbol X]_e|^2= \sum\limits_{j=1}^N\int_{I_j}\gamma(\theta)Qd\rho+\frac{\alpha}{2}
\sum_{j=1}^N|\bm{X}^+-\bm{X}^-|^2
\biggr|_{\rho_{j+\frac{1}{2}}}.
\end{equation}

\begin{theorem}[Energy dissipation]\label{semid}
    Let $(\mu, \boldsymbol{X}, \boldsymbol{q}, \bm{\xi}) \in V_h \times \boldsymbol{W}_h \times \boldsymbol{W}_h \times \boldsymbol{W}_h$ be a solution of the semi-discrete system \eqref{2.5}-\eqref{2.6}. Then the total discrete energy $W_c^h(t)$, as defined in \eqref{2.7}, is dissipative during the evolution, i.e.,
\begin{equation}
\label{2.8}
    W_c^h(t)\le W_c^h(t_1)\le W_c^h(0),\quad t\ge t_1\ge 0.
\end{equation}
\end{theorem}
\begin{proof}
Utilizing the identities:
$$
\partial_tQ=\partial_t|\boldsymbol{q}|=\frac{\boldsymbol{q}\cdot\partial_t\boldsymbol{q}}{Q},\quad
\theta=\operatorname{atan2}(q_2,q_1),\quad
\partial_t\theta=-\frac{\boldsymbol{q}^\perp\cdot\partial_t\boldsymbol{q}}{Q^2},
$$
we obtain:
\[
\partial_t\bigl(\gamma(\theta)Q\bigr)
=
\left(G(\theta)\frac{\boldsymbol q}{Q}\right)
\cdot\partial_t\boldsymbol q.
\]
Testing \eqref{2.5c} with
$\boldsymbol r=\partial_t\boldsymbol q$ and using the
definition of $W_c^h(t)$, we obtain
\begin{equation}\label{eq:sd-csf-energy-derivative}
\frac{d}{dt}W_c^h(t)
=
(\boldsymbol\xi,\partial_t\boldsymbol q)_h
+\alpha\sum_{e\in\mathcal E_h}
[\boldsymbol X]_e\cdot[\partial_t\boldsymbol X]_e.
\end{equation}
Differentiate \eqref{2.5d} with a time-independent test
function, then choose
$\boldsymbol\vartheta=\boldsymbol\xi$ and sum over the cells.
Using $\widehat{\boldsymbol X}=\boldsymbol X^+$,
cellwise integration by parts, and periodicity, we find
\begin{align*}
(\partial_t\boldsymbol q,\boldsymbol\xi)_h
&=
-(\partial_t\boldsymbol X,\partial_\rho\boldsymbol\xi)_h
-\sum_{e\in\mathcal E_h}
(\partial_t\boldsymbol X)^+|_e
\cdot[\boldsymbol\xi]_e\\
&=
(\partial_\rho\partial_t\boldsymbol X,\boldsymbol\xi)_h
+\sum_{e\in\mathcal E_h}
\boldsymbol\xi^-|_e
\cdot[\partial_t\boldsymbol X]_e.
\end{align*}
Next, testing \eqref{2.5b} with
$\boldsymbol\phi=\partial_t\boldsymbol X$ and using
$\widehat{\boldsymbol\xi}
=\boldsymbol\xi^-+\alpha[\boldsymbol X]$ yields
\begin{align*}
(\mu\boldsymbol n^\star,\partial_t\boldsymbol X)_h
&=
(\boldsymbol\xi,\partial_\rho\partial_t\boldsymbol X)_h
+\sum_{e\in\mathcal E_h}
\widehat{\boldsymbol\xi}|_e
\cdot[\partial_t\boldsymbol X]_e\\
&=
(\boldsymbol\xi,\partial_t\boldsymbol q)_h
+\alpha\sum_{e\in\mathcal E_h}
[\boldsymbol X]_e\cdot[\partial_t\boldsymbol X]_e\\
&=\frac{d}{dt}W_c^h(t),
\end{align*}
where the last equality follows from
\eqref{eq:sd-csf-energy-derivative}.
Finally, taking $\eta=\mu$ in \eqref{2.5a} gives
\[
\frac{d}{dt}W_c^h(t)
=
(\mu\boldsymbol n^\star,\partial_t\boldsymbol X)_h
=
-(Q\mu,\mu)_h
=
-\sum_{j=1}^N\int_{I_j}Q\mu^2\,d\rho
\le0,
\]
which proves \eqref{2.8}.
\end{proof}

\section{A fully discrete scheme}
\label{sec:fully_discrete}
In this section, we further discretize the semi-discrete scheme \eqref{2.5} in time using a semi-implicit backward Euler method to obtain a fully discrete scheme. We establish its well-posedness and prove conditional energy dissipation.

Let $\tau>0$ denote the time step size, and define the discrete time levels as $t_m=m\tau$ for $m=0,1,\ldots,M$. For $m\ge 0$, we denote by $f^m$ the approximation of $f$ at time $t_m$. The initial value \(\boldsymbol{X}^0\in\boldsymbol{W}_h\) is taken as a finite element interpolation of the given initial curve, and \(\boldsymbol{q}^0\) is constructed from \(\boldsymbol{X}^0\) via the discrete \(\boldsymbol{q}\)-\(\boldsymbol
{X}\) relation \eqref{3.1d} below. The resulting fully discrete scheme reads as follows: find $\mu^{m+1}\in V_h$ and $\boldsymbol{X}^{m+1},\boldsymbol{q}^{m+1},\bm{\xi}^{m+1}\in\boldsymbol{W}_h$ such that, for all test functions $\eta\in V_h$ and $\bm{\phi},\boldsymbol{r},\bm{\vartheta}\in \boldsymbol{W}_h$,
 \begin{subequations}\label{3.1}
\begin{equation}
\int_{I_{j}} \frac{\boldsymbol{X}^{m+1}-\boldsymbol{X}^{m}}{\tau} \cdot
\boldsymbol{n}^{\star m}\eta d \rho+\int_{I_{j}} Q^m\mu^{m+1} \eta d \rho=0,\label{3.1a}
\end{equation}
\begin{equation}
\int_{I_{j}} \mu^{m+1}\boldsymbol{n}^{\star m}\cdot\bm{\phi} d \rho-\int_{I_{j}} \bm{\xi}^{ m+1}\cdot\partial_{\rho}\bm{\phi} d \rho+
\widehat{\bm{\xi}}^{m+1}\cdot\bm{\phi}^-|_{\rho_{j+\frac{1}{2}}}
-\widehat{\bm{\xi}}^{m+1}\cdot\bm{\phi}^+|_{\rho_{j-\frac{1}{2}}}=0,
\label{3.1b}
\end{equation}
\begin{equation}
\int_{I_{j}}\bm{\xi}^{m+1}\cdot\boldsymbol{r} d\rho
-\int_{I_{j}}\left(G(\theta^m)\frac{\boldsymbol{q}^{m+1}}{Q^{m}}\right)\cdot\boldsymbol{r} d\rho=0,\label{3.1c}
\end{equation}
\begin{equation}
\int_{I_{j}} \boldsymbol{q}^{m+1}\cdot \bm{\vartheta} d \rho+\int_{I_{j}} \boldsymbol{X}^{m+1}\cdot\partial_{\rho} \bm{\vartheta} d \rho-\widehat{\boldsymbol{X}}^{m+1}\cdot \bm{\vartheta}^-|_{\rho_{j+\frac{1}{2}}}+\widehat{\boldsymbol{X}}^{m+1}\cdot \bm{\vartheta}^+|_{\rho_{j-\frac{1}{2}}}=0,\label{3.1d}
\end{equation}
\end{subequations}
where
\begin{flalign}
\label{3.2}
\widehat{\boldsymbol{X}}^{m+1}&=(\boldsymbol{X}^{m+1})^+,\quad
\widehat{\bm{\xi}}^{ m+1}=(\bm{\xi}^{ m+1})^-+\alpha((\boldsymbol{X}^{m+1})^+-(\boldsymbol{X}^{m+1})^-).
\end{flalign}
Clearly, it is a semi-implicit scheme.

\subsection{Energy dissipation}

For \(\phi,\psi\in[-\pi,\pi]\) and \(r>0\), set
\begin{align*}
&\mathcal{L}_\gamma(\phi,\psi)=2\gamma(\phi)-\gamma(\phi)\cos(\phi-\psi)
-\gamma'(\phi)\sin(\phi-\psi)-\gamma(\psi),\\
&\mathcal{C}_\gamma(r;\phi,\psi)
=\gamma(\phi)(r^2+1)-r\left[\gamma(\phi)\cos(\phi-\psi)
+\gamma'(\phi)\sin(\phi-\psi)
+\gamma(\psi)\right].
\end{align*}
Then
\begin{equation}\label{cg}
\mathcal{C}_\gamma(r;\phi,\psi)
=r\mathcal{L}_\gamma(\phi,\psi)+\gamma(\phi)(r-1)^2.
\end{equation}

%

\begin{theorem}[Conditional energy dissipation]\label{thm:fd-csf-energy}
Assume that \(Q^m>0\) for all \(m\ge0\). Then the fully discrete scheme \eqref{3.1}-\eqref{3.2} satisfies
\begin{equation}\label{identity}
W^{m+1}-W^m+\tau\sum_{j=1}^N\int_{I_j}
Q^m(\mu^{m+1})^2\,d\rho
+\frac{\alpha}{2}
\sum_{e\in\mathcal E_h}
|[\delta\boldsymbol X^m]_e|^2
+\mathcal K^m =0,\end{equation}
where
\begin{align}
&W^m:=\sum_{j=1}^N\int_{I_j}\gamma(\theta^m)Q^m\,d\rho
+\frac{\alpha}{2}
\sum_{e\in\mathcal E_h}
|[\boldsymbol X^m]_e|^2,\label{wm}\\
&\mathcal K^m:
=\sum_{j=1}^N\int_{I_j}
Q^m\mathcal C_\gamma
\left(
\frac{Q^{m+1}}{Q^m};
\theta^m,\theta^{m+1}
\right)\,d\rho.\label{km}
\end{align}
Thus we have
\begin{equation}\label{eq:fd-csf-dissipation}
W^{m+1}\le W^m\le\cdots\le W^0,\quad \forall m\ge 0,
\end{equation}
if  $\mathcal K^m
\ge 0$.
\end{theorem}

\begin{proof}
Fix $m$ and abbreviate
\[\delta\boldsymbol{X}=\boldsymbol{X}^{m+1}-\boldsymbol{X}^{m},\quad \delta\boldsymbol{q}=\boldsymbol{q}^{m+1}-\boldsymbol{q}^{m}.
\]
Subtract \eqref{3.1d} at time levels $m+1$ and $m$,
test with \(\bm{\vartheta}=\bm{\xi}^{m+1}\), and sum over the cells. Using $\widehat{\boldsymbol X}=\boldsymbol X^+$ gives
\begin{align}
(\delta\boldsymbol q,\boldsymbol\xi^{m+1})_h
&=
-(\delta\boldsymbol X,
  \partial_\rho\boldsymbol\xi^{m+1})_h
-\sum_{e\in\mathcal E_h}
(\delta\boldsymbol X)^+|_e
\cdot[\boldsymbol\xi^{m+1}]_e
\nonumber\\
&=
(\partial_\rho\delta\boldsymbol X,
 \boldsymbol\xi^{m+1})_h
+\sum_{e\in\mathcal E_h}
(\boldsymbol\xi^{m+1})^-|_e
\cdot[\delta\boldsymbol X]_e.
\label{eq:fd-csf-discrete-duality}
\end{align}
The second equality follows from cellwise integration by
parts and periodicity.
Next, test \eqref{3.1b} with $\delta \boldsymbol{X}$ and sum over the cells. By \eqref{eq:fd-csf-discrete-duality} and
\[
\widehat{\boldsymbol\xi}^{m+1}
=(\boldsymbol\xi^{m+1})^-
+\alpha[\boldsymbol X^{m+1}],
\]we obtain
\begin{align}
(\mu^{m+1}\boldsymbol n^{\star m},
 \delta\boldsymbol X)_h
&=
(\boldsymbol\xi^{m+1},
 \partial_\rho\delta\boldsymbol X)_h
+\sum_{e\in\mathcal E_h}
\widehat{\boldsymbol\xi}^{m+1}|_e
\cdot[\delta\boldsymbol X]_e
\nonumber\\
&=
(\boldsymbol\xi^{m+1},\delta\boldsymbol q)_h
+\alpha\sum_{e\in\mathcal E_h}
[\boldsymbol X^{m+1}]_e
\cdot[\delta\boldsymbol X]_e.
\label{eq:fd-csf-flux-identity}
\end{align}
Testing \eqref{3.1c} with $\delta \boldsymbol{q}$ gives
\[
(\boldsymbol\xi^{m+1},\delta\boldsymbol q)_h
=
\left(
G(\theta^m)\frac{\boldsymbol q^{m+1}}{Q^m},
\delta\boldsymbol q
\right)_h.
\]
Testing \eqref{3.1a} with $\tau\mu^{m+1}$ gives
\begin{equation}\label{eq:fd-csf-mu-dissipation}
(\mu^{m+1}\boldsymbol n^{\star m},
 \delta\boldsymbol X)_h
=-\tau (Q^m,(\mu^{m+1})^2)_h.
\end{equation}
Combining these identities with
\eqref{eq:fd-csf-flux-identity}, we find
\begin{equation}\label{eq:fd-csf-basic-identity}
\left(
G(\theta^m)\frac{\boldsymbol q^{m+1}}{Q^m},
\delta\boldsymbol q
\right)_h
+\alpha\sum_{e\in\mathcal E_h}
[\boldsymbol X^{m+1}]_e
\cdot[\delta\boldsymbol X]_e
+\tau(Q^m,(\mu^{m+1})^2)_h=0.
\end{equation}
Since
$\boldsymbol q^\ell=Q^\ell\boldsymbol t_{\theta^\ell}$,
this yields
\begin{align*}
\left(
G(\theta^m)\frac{\boldsymbol q^{m+1}}{Q^m},
\delta\boldsymbol q
\right)_h
&=
\sum_{j=1}^N\int_{I_j}
\bigl(
\gamma(\theta^{m+1})Q^{m+1}
-\gamma(\theta^m)Q^m
\bigr)\,d\rho
+\mathcal K^m,
\end{align*}
where we have used the calculation
\begin{align*}
&\quad \frac{\zeta_1}{\zeta_0}G(\phi)\boldsymbol{t}_\psi
\cdot
\left(\zeta_1\boldsymbol{t}_\psi-\zeta_0\boldsymbol{t}_\phi\right)
-\left(
\gamma(\psi)\zeta_1
-\gamma(\phi)\zeta_0
\right)
\\
&=
\zeta_0\left(
\gamma(\phi)((\zeta_1/\zeta_0)^2+1)
-\frac{\zeta_1}{\zeta_0}\left[
\gamma(\phi)\cos(\phi-\psi)
+\gamma'(\phi)\sin(\phi-\psi)
+\gamma(\psi)
\right]
\right)
\\
&=\zeta_0\mathcal{C}_\gamma(\zeta_1/\zeta_0;\phi,\psi),
\end{align*}
for $\zeta_0>0$, $\boldsymbol{t}_\theta=(\cos\theta,\sin\theta)^T$.
Moreover,
\[\alpha\sum_{e\in\mathcal E_h}
[\boldsymbol X^{m+1}]_e
\cdot[\delta\boldsymbol X]_e=\frac{\alpha}{2}
\sum_{e\in\mathcal E_h}
\left(
|[\boldsymbol X^{m+1}]_e|^2
-|[\boldsymbol X^m]_e|^2
\right)
+\frac{\alpha}{2}
\sum_{e\in\mathcal E_h}
|[\delta\boldsymbol X^m]_e|^2,
\]
by the identity
\[
\boldsymbol a\cdot(\boldsymbol a-\boldsymbol b)
=
\frac12\left(
|\boldsymbol a|^2-|\boldsymbol b|^2
+|\boldsymbol a-\boldsymbol b|^2
\right).
\]
Substitution into \eqref{eq:fd-csf-basic-identity} proves
\eqref{identity}, which directly yields \eqref{eq:fd-csf-dissipation} if $\mathcal{K}^m\ge0$.
\end{proof}

\setcounter{remark}{1}
\begin{remark}
The Li-Bao condition for energy stability \cite{Li2021} requires
\[
\mathcal{L}_\gamma(\phi,\psi)\ge0,\qquad \forall \phi,\psi\in \mathbb{R}.
\]
By \eqref{cg}, this condition implies $\mathcal{C}_\gamma\ge 0$ for every choice of positive lengths and angle fields. Hence, $\mathcal{K}^m\ge0$ and the energy dissipates.
However, the condition $\mathcal{K}^m\ge0$ is less restrictive: negative contributions from $\mathcal L_\gamma$ may be
compensated by the length-change term
$\gamma(\phi)(r-1)^2$ or by positive contributions elsewhere.
It is therefore an a posteriori condition on consecutive
discrete states, rather than a condition on $\gamma$ alone.
In the isotropic case \(\gamma\equiv1\),  we have \[\mathcal{C}_\gamma(r;\phi,\psi)=r(1-\cos(\phi-\psi))+(r-1)^2\ge0,\]
so the integrated condition holds automatically.
\end{remark}

\subsection{Well-posedness}
For each $m\ge 0$, assuming $N\ge 3$, the well-posedness of the fully discrete scheme \eqref{3.1} is established under the following condition:
\begin{theorem}[Well-posedness]
Suppose $Q^m>0$ and
\begin{equation}\label{span}
\operatorname{span}
\left\{
\boldsymbol n^{\star m}(\rho):
\rho\in I_j,\ j=1,\ldots,N
\right\}
=\mathbb R^2.
\end{equation}
Then, the fully discrete scheme \eqref{3.1} admits a unique solution.
\end{theorem}
\begin{proof} It suffices to prove that the associated homogeneous system has only the trivial solution. For
brevity, we suppress the superscript $m+1$ on the unknowns. The homogeneous system is obtained
by replacing \eqref{3.1a} with
 \begin{equation}\label{3.5}
\frac{1}{\tau}
(\boldsymbol X\cdot\boldsymbol n^{\star m},\eta)_h
+(Q^m\mu,\eta)_h=0,
\qquad \forall\eta\in V_h.
\end{equation}
Summing \eqref{3.1d} over the cells and integrating by parts,
we obtain, using periodicity and
$\widehat{\boldsymbol X}=\boldsymbol X^+$,
\begin{equation}\label{eq:fd-csf-dg-relation}
(\boldsymbol q,\boldsymbol\vartheta)_h
=
(\partial_\rho\boldsymbol X,\boldsymbol\vartheta)_h
+\sum_{e\in\mathcal E_h}
[\boldsymbol X]_e\cdot\boldsymbol\vartheta^-|_e,
\qquad \forall\boldsymbol\vartheta\in\boldsymbol W_h.
\end{equation}
Testing \eqref{3.1b} with $\boldsymbol\phi=\boldsymbol X$
and applying \eqref{eq:fd-csf-dg-relation} with
$\boldsymbol\vartheta=\boldsymbol\xi$ gives
\begin{align*}
(\mu\boldsymbol n^{\star m},\boldsymbol X)_h
&=
(\boldsymbol\xi,\partial_\rho\boldsymbol X)_h
+\sum_{e\in\mathcal E_h}
\widehat{\boldsymbol\xi}|_e\cdot[\boldsymbol X]_e=
(\boldsymbol\xi,\boldsymbol q)_h
+\alpha\sum_{e\in\mathcal E_h}|[\boldsymbol X]_e|^2,
\end{align*}
where we have used
$\widehat{\boldsymbol\xi}
=\boldsymbol\xi^-+\alpha[\boldsymbol X]$.

Next, choose $\boldsymbol r=\boldsymbol q$ in \eqref{3.1c}
and $\eta=\tau\mu$ in \eqref{3.5}.
Combining the resulting identities with the preceding
equality and using
$\boldsymbol v\cdot G(\theta)\boldsymbol v
=\gamma(\theta)|\boldsymbol v|^2$, we obtain
\begin{equation}\label{eq:fd-csf-homogeneous-coercivity}
\left(
\frac{\gamma(\theta^m)}{Q^m}\boldsymbol q,
\boldsymbol q
\right)_h
+\tau(Q^m\mu,\mu)_h
+\alpha\sum_{e\in\mathcal E_h}|[\boldsymbol X]_e|^2
=0.
\end{equation}
Since $\gamma(\theta^m)>0$, $Q^m>0$, $\tau>0$, and
$\alpha> 0$, it follows that
\[
\boldsymbol q=\boldsymbol0,\qquad
\mu=0,\qquad
[\boldsymbol X]_e=\boldsymbol0
\quad\text{for every }e\in\mathcal E_h.
\]
Taking $\boldsymbol r=\boldsymbol\xi$ in \eqref{3.1c}
then yields $(\boldsymbol\xi,\boldsymbol\xi)_h=0$,
so $\boldsymbol\xi=\boldsymbol0$.
With $\boldsymbol q=\boldsymbol0$ and
$[\boldsymbol X]_e=\boldsymbol0$,
\eqref{eq:fd-csf-dg-relation} reduces to
\[
(\partial_\rho\boldsymbol X,\boldsymbol\vartheta)_h=0,
\qquad \forall\boldsymbol\vartheta\in\boldsymbol W_h.
\]
Choosing
$\boldsymbol\vartheta=\partial_\rho\boldsymbol X
\in\boldsymbol W_h$ gives
$\partial_\rho\boldsymbol X=\boldsymbol0$.
Continuity across the interfaces therefore implies that
$\boldsymbol X=\boldsymbol X^c$ for some
$\boldsymbol X^c\in\mathbb R^2$.
Substituting this into \eqref{3.5} gives
\[
(\boldsymbol X^c\cdot\boldsymbol n^{\star m},\eta)_h=0,
\qquad \forall\eta\in V_h.
\]
Since $\boldsymbol X^c\cdot\boldsymbol n^{\star m}\in V_h$,
we may take
$\eta=\boldsymbol X^c\cdot\boldsymbol n^{\star m}$ to obtain
\[
\|\boldsymbol X^c\cdot\boldsymbol n^{\star m}\|_h^2=0.
\]
Thus $\boldsymbol X^c$ is orthogonal to
$\boldsymbol n^{\star m}$ throughout each cell.
The span assumption implies $\boldsymbol X^c=\boldsymbol0$.

Hence the homogeneous system has only the trivial solution.
Since the discrete system is linear and has
$7\dim V_h$ unknown coefficients and the same number of
equations, existence and uniqueness follow.
\end{proof}

\subsection{Extension to anisotropic area-preserving curve-shortening flow}
\label{subsec:apcsf}

In this subsection, we extend the LDG method to the anisotropic AP-CSF---a second-order geometric flow with normal velocity
\[
    \partial_t\boldsymbol{X}\cdot \boldsymbol{n}=-\mu+\overline{\mu},\quad \overline{\mu}=\int_{\Gamma(t)}\mu ds/\int_{\Gamma(t)} ds.
\]
Analogously to \eqref{2.1}, we reformulate the anisotropic AP-CSF as a coupled system:
\begin{equation}
\label{3.12}
\left\{
\begin{aligned}
&\boldsymbol{n}^\star\cdot\partial_{t} \boldsymbol{X}=-Q(\mu-\overline{\mu}),\\
&\mu\boldsymbol{n}^\star=-\partial_{\rho}\left(G(\theta)\frac{\partial_\rho\boldsymbol{X}}{Q}\right).
\end{aligned}
\right.
\end{equation}
The semi-discrete LDG scheme is obtained from
\eqref{2.5}--\eqref{2.6} by replacing \eqref{2.5a} with
\begin{equation}\label{eq:apcsf-semidiscrete-first}
(\boldsymbol n^\star\cdot\partial_t\boldsymbol X,\eta)_h
+\bigl(Q(\mu-\overline\mu_h),\eta\bigr)_h=0,
\qquad \forall\eta\in V_h,
\end{equation}
whereas the fully discrete scheme replaces \eqref{3.1a} by
\begin{equation}\label{eq:apcsf-fully-discrete-first}
\frac{1}{\tau}
\bigl(
(\boldsymbol X^{m+1}-\boldsymbol X^m)
\cdot\boldsymbol n^{\star m},\eta
\bigr)_h
+
\bigl(\mu^{m+1}-\overline\mu_h^{m+1}, Q^m\eta
\bigr)_h
=0,
\quad \forall\eta\in V_h.
\end{equation}
Here the discrete means are defined by
\begin{equation}\label{mum}
\overline\mu_h
=
\frac{(Q\mu,1)_h}{(Q,1)_h},
\qquad
\overline\mu_h^{m+1}
=
\frac{(Q^m\mu^{m+1},1)_h}{(Q^m,1)_h}.
\end{equation}
Since $Q^m$ is given, $\overline\mu_h^{m+1}$ is linear
in $\mu^{m+1}$.


\begin{theorem}[AP-CSF]
\label{thm:apcsf-properties}
For a differentiable semi-discrete solution with $Q>0$,
\begin{equation}\label{eq:apcsf-semidiscrete-energy}
\frac{d}{dt}W_c^h(t)
=
-\bigl(
Q(\mu-\overline\mu_h),\mu-\overline\mu_h
\bigr)_h
\le0.
\end{equation}
For a fully discrete solution with $Q^m,Q^{m+1}>0$,
define
\[
\mathcal D_{\mathrm{AP}}^m
=\bigl(
Q^m(\mu^{m+1}-\overline\mu_h^{m+1}),
\mu^{m+1}-\overline\mu_h^{m+1}
\bigr)_h.
\]
Then
\begin{equation}\label{eq:apcsf-energy-identity}
W^{m+1}-W^m
+\tau\mathcal D_{\mathrm{AP}}^m
+\frac{\alpha}{2}
\sum_{e\in\mathcal E_h}
|[\delta\boldsymbol X^m]_e|^2+\mathcal K^m
=0,
\end{equation}
where $W^m$ and
$\mathcal K^m$ are defined in \eqref{wm} and \eqref{km}, respectively.
Consequently, if $\mathcal K^m\ge0$, then the energy $W^m$ is non-increasing.

Moreover, suppose that $Q^m>0$ on every closed cell,
$\boldsymbol n^{\star m}\in\boldsymbol W_h$ satisfies \eqref{span}.
Then the fully discrete scheme for anisotropic AP-CSF admits a unique solution for every $\tau>0$.
\end{theorem}

\begin{proof}
It can be observed that
\[
\left(\mu-\overline\mu_h, Q\right)_h=0,\quad
\left(Q(\mu-\overline\mu_h),\mu\right)_h=
\left(Q, |\mu-\overline\mu_h|^2\right)_h.
\]
Thus the energy estimates follow as in Theorems~\ref{semid} and
\ref{thm:fd-csf-energy}. For well-posedness, the only additional
step is
\[
\mu^{m+1}=\overline\mu^{m+1}_h=:c_\mu,
\qquad
\boldsymbol\xi^{m+1}=0,
\]
and hence
\[0=
c_\mu\int_{I_j}|\boldsymbol n^{\star m}|^2\,d\rho=c_\mu\int_{I_j}(Q^m)^2\,d\rho.
\]
The  assumption $Q^m>0$ gives $c_\mu=0$ immediately.
The remaining proof is identical to that of Theorem 3.2.
\end{proof}

\section{Numerical results \& discussions}
\label{sec:numerical}

In this section, we present a series of numerical experiments to evaluate the performance of the proposed method, focusing on spatial convergence rates, energy dissipation properties, area loss, and mesh-ratio evolution. For quantitative comparison between curves, we employ the manifold distance to measure numerical errors. Given two closed curves $\Gamma_1$ and $\Gamma_2$, the manifold distance is computed as defined in \cite{Zhao2021}
$$
M(\Gamma_1,\Gamma_2)=|(\Omega_1\setminus\Omega_2 )\cup(\Omega_2\setminus\Omega_1)|=|\Omega_1|+|\Omega_2|-2|\Omega_1\cap\Omega_2|,
$$
where $\Omega_i$ denotes the region enclosed by $\Gamma_i$ and $|\Omega|$ denotes the area of $\Omega$.

Let $\boldsymbol{X}^m = (x^m, y^m)^T$ be the numerical solution of the fully discrete scheme \eqref{3.1}. To obtain a geometric representation of $\boldsymbol{X}^m$ for visualization and distance computations, we sample each cell $I_j$($j=1,\ldots,N$) by 300 uniformly distributed points. At the cell interface $\rho_{j+\frac{1}{2}}$, we take the value as the midpoint of the left and right traces: $\frac{1}{2}\left( (\boldsymbol{X}^m)^- + (\boldsymbol{X}^m)^+ \right)|{\rho_{j+\frac{1}{2}}}$. Connecting the sampling points in sequence yields a polygon $\Gamma^m$, which we use both for visualization and for computing the manifold distance.

Denote by $A(t=t_m)$ the area enclosed by the curve related to the numerical solution $\boldsymbol{X}^m$, defined as
\begin{equation}
\label{4.1}
    A^h(t)=\sum_{j=1}^{N}\int_{I_j}x\partial_\rho yd\rho.
\end{equation}
For systematic assessment we define normalized energy, normalized area loss, and mesh ratio as follows:
\begin{equation}
\frac{W^h(t)}{W^h(0)}\biggr|_{t=t_m}=\frac{W^m}{W^0},\quad
    \Delta A^h(t)|_{t=t_m}=\frac{A^m-A^0}{A^0},\quad
    \Psi(t)|_{t=t_m}=\frac{\max\limits_{1\le j\le N}|\boldsymbol{h}_j^m|}{\min\limits_{1\le j\le N}|\boldsymbol{h}_j^m|},
\end{equation}
where $\boldsymbol{h}^m_j=(\boldsymbol{X}^m)^-|_{\rho_{j+\frac{1}{2}}}-(\boldsymbol{X}^m)^+|_{\rho_{j-\frac{1}{2}}}$, $W^m$ defined as \eqref{2.7} corresponds to the discrete energy. To examine how individual terms contribute to the total discrete energy dissipation, we furthermore split $W_c^h(t)=W^h_1(t)+W^h_2(t)$, where
$$W^h_1(t)=\int_{I}\gamma(\theta)Qd\rho,\quad
W^h_2(t)=\frac{\alpha}{2}\sum\limits_{j=1}^N|\boldsymbol{X}^+-\boldsymbol{X}^-|^2\biggr|_{\rho_{j+\frac{1}{2}}},
$$
and rewrite the normalized energy accordingly:
$$\frac{W^h_c(t)}{W^h_c(0)}\biggr|_{t=t_m}=\frac{W^h_1(t)}{W^h_c(0)}\biggr|_{t=t_m}+\frac{W^h_2(t)}{W^h_c(0)}\biggr|_{t=t_m}=\frac{W^m_1}{W^0_c}+\frac{W^m_2}{W^0_c}.$$

For the experiments below, we consider four initial curves:
\begin{itemize}
    \item (Curve 1): the unit circle;
    \item (Curve 2): an ellipse with semi-major axis 2 and semi-minor axis 1;
    \item (Curve 3): a `flower' curve parameterized by
\[
\left\{
\begin{aligned}
&x=(2+\cos(6\delta))\cos\delta,\\
&y=(2+\cos(6\delta))\sin\delta,
\end{aligned}
\right.\quad \delta=2\pi\rho,\quad \rho\in I;
\]
\item (Curve 4): the so-called `Mikula' curve \cite{Sevcovic} parameterized by
\[
\left\{
\begin{aligned}
&x=2\cos\delta,\\
&y=2\Big(0.7\sin\delta+\sin(\cos\delta)\\
&\qquad+\frac14\big(1-\cos(2\delta)+
\frac12\cos(4\delta)-\cos(6\delta)+\frac12\cos(8\delta)\big)\Big),
\end{aligned}
\right.
\]
where $\delta=2\pi\rho$, $\rho\in I$.
\end{itemize}

\subsection{Isotropic case}
\label{subsec:numerical-isotropic}

We first consider $\gamma(\theta)=1$, corresponding to $\beta=0$ in \eqref{eq:surface_energy}.

\subsubsection[Convergence tests]{Convergence tests}
\begin{example}[Convergence tests]
To validate the accuracy of the proposed LDG schemes, we systematically examine spatial convergence rates for both the CSF and AP-CSF using piecewise polynomial approximations of degree $k$ ($k=1,2,3,4$). The temporal step size is set as $\tau=Ch^{k+1}$, where $C$ is a positive constant.
\end{example}

For the isotropic CSF with Curve 1 as initial data, an exact solution is available:
\[
    \boldsymbol{X}_{\text{true}}(\rho,t)=\sqrt{1-2t}(\cos(2\pi\rho),\,\sin(2\pi\rho)),\quad \rho\in I,\quad
    t\in[0,0.5).
\]
For this case, we compute numerical errors by direct comparison with the true solution. In cases where the exact solution is unavailable, we use reference solutions obtained with the LDG method using $P^4$ element on a refined mesh ($N=2000$) and a tiny time step ($\tau=10^{-6}$).

\begin{table}[htbp!]
	\centering
	\caption{Convergence of the manifold distance for the CSF. Computations use $P^k$ finite elements in space ($k = 1, 2, 3,4$) where the initial curve is chosen as Curve 1. Time step and final time: $\tau=5h^{k+1}$, $T=0.25$.}
	\resizebox{\textwidth}{!}{%
		\begin{tabular}{llll @{\hspace{0.8em}} lll@{\hspace{0.8em}}lll@{\hspace{0.8em}}lll}
		\toprule
		   &$P^1$&&&$P^2$&&&$P^3$&&&$P^4$&&\\
		\cmidrule(r){2-4} \cmidrule(r){5-7} \cmidrule(r){8-10} \cmidrule(r){11-13}
		&$N$   & error& order&$N$   & error& order&$N$   & error& order&$N$   & error& order\\
		\bottomrule
		$\alpha=1/h$&5 & 1.73e-01  &-      &5 & 2.32e-02  &-    &5 & 2.30e-02  &-    & 5 & 5.17e-03  &-\\
		&10& 4.37e-02  & 1.98 & 10& 3.21e-03  &2.85 &10& 1.61e-03  &3.83 & 10& 1.62e-04  &4.99\\
		&20& 1.10e-02  &1.97  & 20& 4.06e-04  &2.98 &20& 9.94e-05  &4.02 & 15& 2.10e-05  &5.03\\
		&40& 2.78e-03 & 1.99  & 40& 5.08e-05  &2.99 &40& 5.74e-06  &4.11 & 20& 4.83e-06  &5.10\\
		\\
		$\alpha=0.1$&5 & 2.70e-01  &-      &5 & 1.06e-01  &-    &5 & 2.50e-02  &-    & 5 & 5.17e-03  &-\\
		&10& 1.18e-01  & 1.19 & 10& 1.53e-02  &2.79 &10& 1.61e-03  &3.95 & 10& 1.62e-04  &4.99\\
		&20& 3.41e-02  &1.79  & 20& 1.94e-03  &2.98 &20& 9.94e-05  &4.02 & 15& 2.10e-05  &5.03\\
		&40& 8.93e-03 & 1.93  & 40& 2.44e-04  &2.99 &40& 5.74e-06  &4.11 & 20& 4.83e-06  &5.10\\
		\\
		$\alpha=0$&5 & 2.66e-01  &-      &5 & 1.06e-01  &-    &5 & 2.50e-02  &-    & 5 & 5.17e-03  &-\\
		&10& 1.18e-01  & 1.17 & 10& 1.53e-02  &2.79 &10& 1.61e-03  &3.95 & 10& 1.62e-04  &4.99\\
		&20& 3.40e-02  &1.79  & 20& 1.94e-03  &2.98 &20& 9.94e-05  &4.02 & 15& 2.10e-05  &5.03\\
		&40& 8.90e-03 & 1.93  & 40& 2.44e-04  &2.99 &40& 5.74e-06  &4.11 & 20& 4.83e-06  &5.10\\
		\bottomrule
	\end{tabular}}
	\label{tab1}
\end{table}

\begin{table}[htbp!]
	\centering
	\caption{Convergence of the manifold distance for the AP-CSF, using $P^k$ elements in space, where the initial curve is chosen as Curve 2. Time step and final time: $\tau=5h^{k+1}$, $T=0.25$.}
    		\resizebox{\textwidth}{!}{%
            \begin{tabular}{llll @{\hspace{0.8em}} lll@{\hspace{0.8em}}lll@{\hspace{0.8em}}lll}
		\toprule
           &$P^1$&&&$P^2$&&&$P^3$&&&$P^4$&&\\
            \cmidrule(r){2-4} \cmidrule(r){5-7} \cmidrule(r){8-10}\cmidrule(r){11-13}
		&$N$   & error& order&$N$   & error& order&$N$   & error& order&$N$   & error& order\\
            \bottomrule
 		$\alpha=1/h$&5 & 7.00e-01  &-      &5 & 3.63e-02  &-    &5 & 7.71e-03  &-    & 5 & 3.00e-03  &-\\
		  &10& 1.63e-01  & 2.10 & 10& 4.55e-03  &2.99 &10& 4.45e-04  &4.11 & 10& 8.67e-05  &5.11\\
        &20& 4.00e-02  &2.02  & 20& 6.26e-04  &2.85 &20& 2.43e-05  &4.19 & 15& 1.06e-05  &5.17\\
		  &40& 9.98e-03 & 2.00  & 40& 7.41e-05  &3.07 &40& 1.46e-06  &4.05 & 20& 2.42e-06  &5.13\\
           \\
          $\alpha=0.1$&5 & 7.34e-01  &-      &5 & 4.79e-02  &-    &5 & 2.28e-02  &-    & 5 & 5.78e-03  &-\\
		  &10& 2.20e-01  & 1.73 & 10& 9.48e-03  &2.33 &10& 1.78e-03  &3.67 & 10& 7.23e-04  &3.00\\
        &20& 5.72e-02  &1.94  & 20& 1.29e-03  &2.86 &20& 1.93e-04  &3.20 & 15& 1.09e-04  &4.65\\
		  &40& 1.44e-02 & 1.98  & 40& 1.80e-04  &2.84 &40& 1.82e-05  &3.41 & 20& 2.75e-05  &4.79\\
          \\
          $\alpha=0$&5 & 7.38e-01  &-      &5 & 4.94e-02  &-    &5 & 2.41e-02  &-    & 5 &  6.74e-03 &-\\
		  &10& 2.24e-01  & 1.71 & 10& 1.03e-02  &2.25 &10&   2.32e-03&3.37 & 10&  7.88e-04 &3.09\\
        &20& 5.81e-02  &1.95  & 20& 1.53e-03  &2.75 &20&   2.43e-04&3.25 & 15& 9.53e-05  &5.20\\
		  &40& 1.53e-02 & 1.92  & 40& 2.39e-04  &2.67 &40&   1.80e-05&3.75 & 20& 1.68e-05  &6.01\\
		\bottomrule
	\end{tabular}}
    \label{tab2}
\end{table}

We test convergence for both flows using the initials Curves 1 and 2. Tables \ref{tab1}-\ref{tab2} report the manifold distance and estimated convergence rates at the final time $T = 0.25$ for the CSF and AP-CSF, respectively. When the penalty parameter is chosen as $\alpha=1/h$, the scheme attains the optimal convergence rate of order $k+1$ for $P^k$ elements. In contrast, setting $\alpha=0$ or $\alpha=0.1$ still yields a convergent method, but with notably reduced overall accuracy---particularly evident in the case of AP-CSF (cf. Table \ref{tab2}). These findings further underscore the critical importance of incorporating a properly scaled penalty term to preserve the optimal accuracy of the proposed LDG schemes.

\subsubsection[Curve evolution and geometric properties]{Curve evolution and geometric properties}

\begin{example}[Curve evolution driven by the CSF and AP-CSF]
We examine the evolution of the elliptical initial curve (Curve 2) under two dynamics: the standard CSF and AP-CSF. Computations use $N=80$ cells, time step $\tau=10^{-3}$, penalty coefficient $\alpha=1/h$ (unless stated otherwise), and polynomial degree $k=1$.
\end{example}

\begin{figure}[h!]
	\centering
    \subfigure[\empty]{ 
\includegraphics[trim=0 0 55 0,clip,width=0.30\textwidth]
{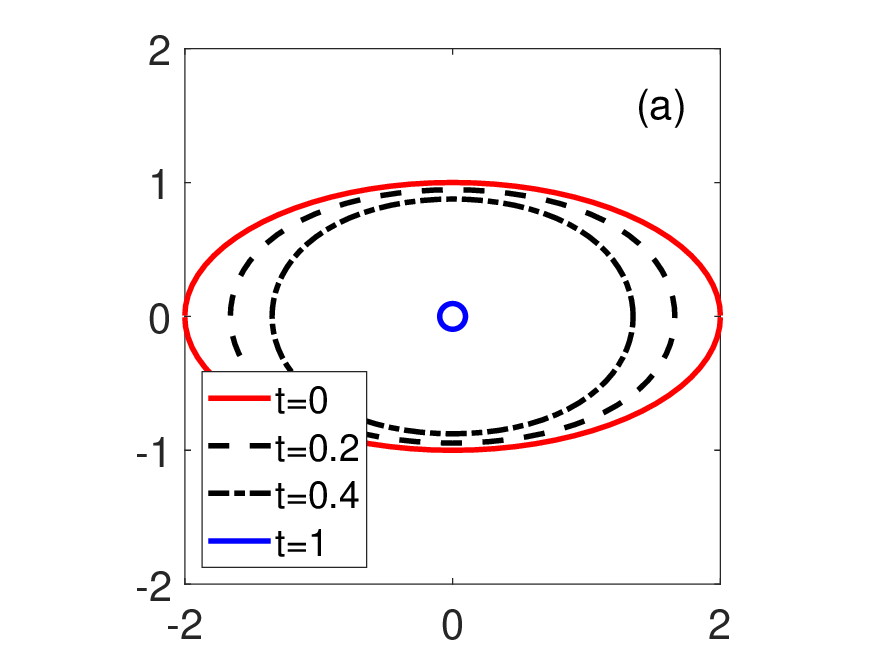}}
\subfigure[\empty]{ 
\includegraphics[trim=0 0 28 0,clip,width=0.30\textwidth,height=0.26\textwidth]
{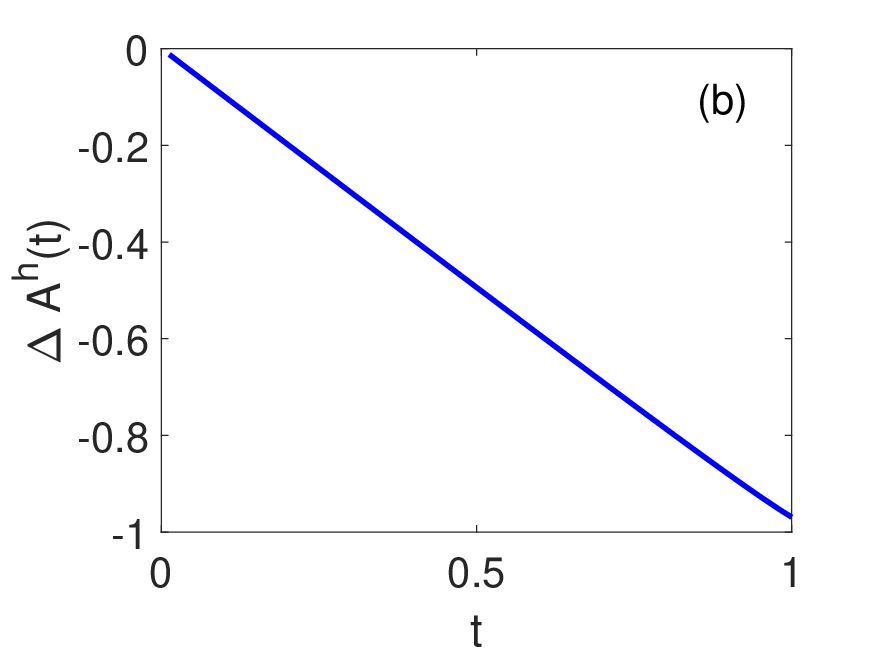}}
\subfigure[\empty]{ 
\includegraphics[trim=0 0 28 0,clip,width=0.30\textwidth,height=0.26\textwidth]
{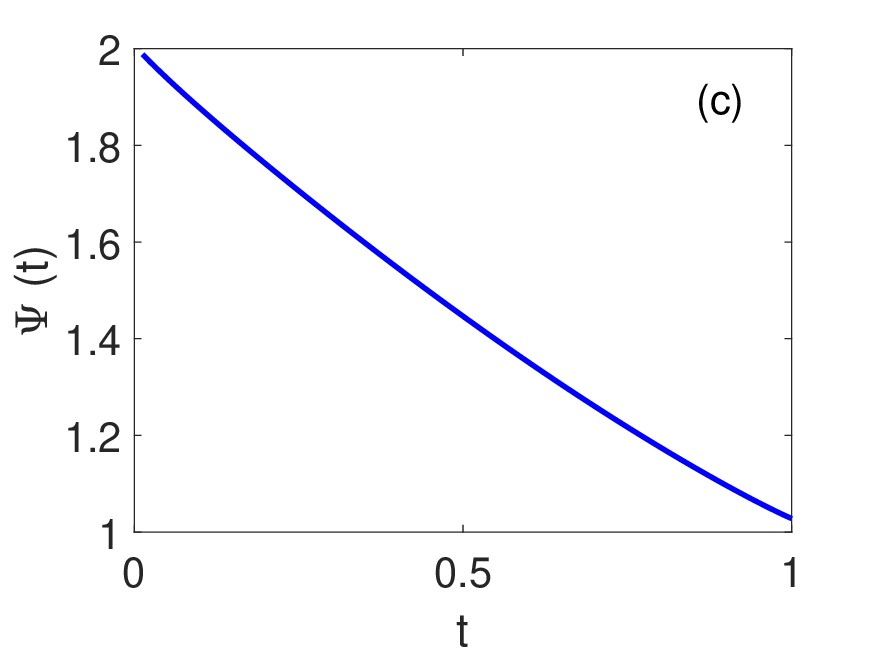}}\\\vspace{-3mm}
\subfigure[\empty]{ 
\includegraphics[trim=0 0 28 0,clip,width=0.30\textwidth,height=0.26\textwidth]
{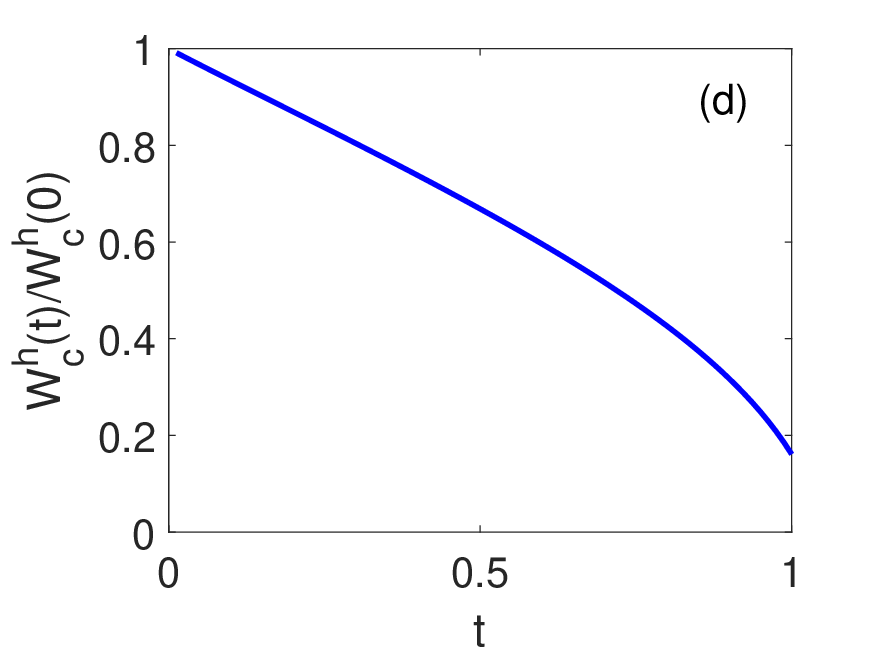}}
\subfigure[\empty]{ 
\includegraphics[trim=0 0 28 0,clip,width=0.30\textwidth,height=0.26\textwidth]
{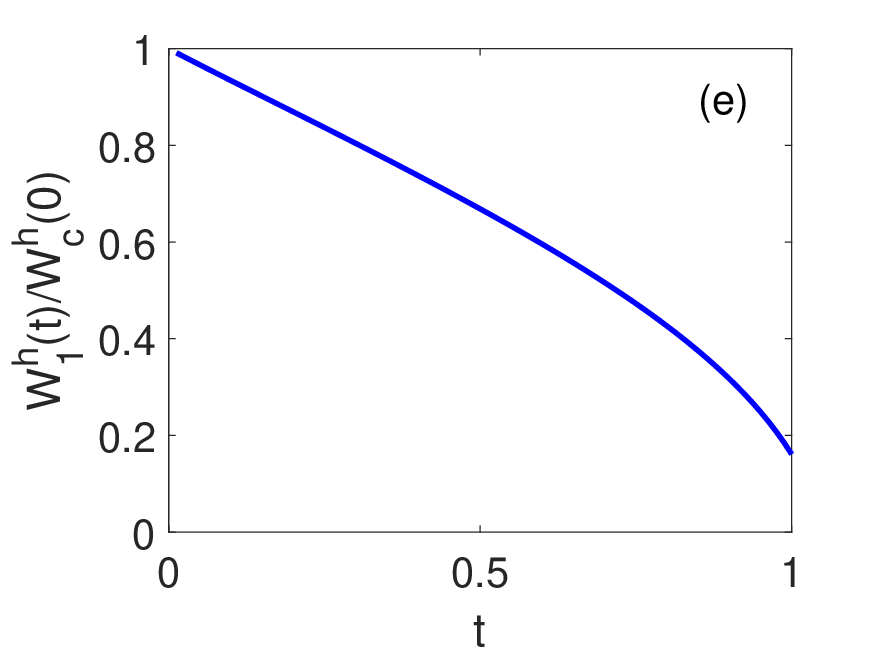}}
\subfigure[\empty]{ 
\includegraphics[trim=0 0 28 0,clip,width=0.30\textwidth,height=0.26\textwidth]
{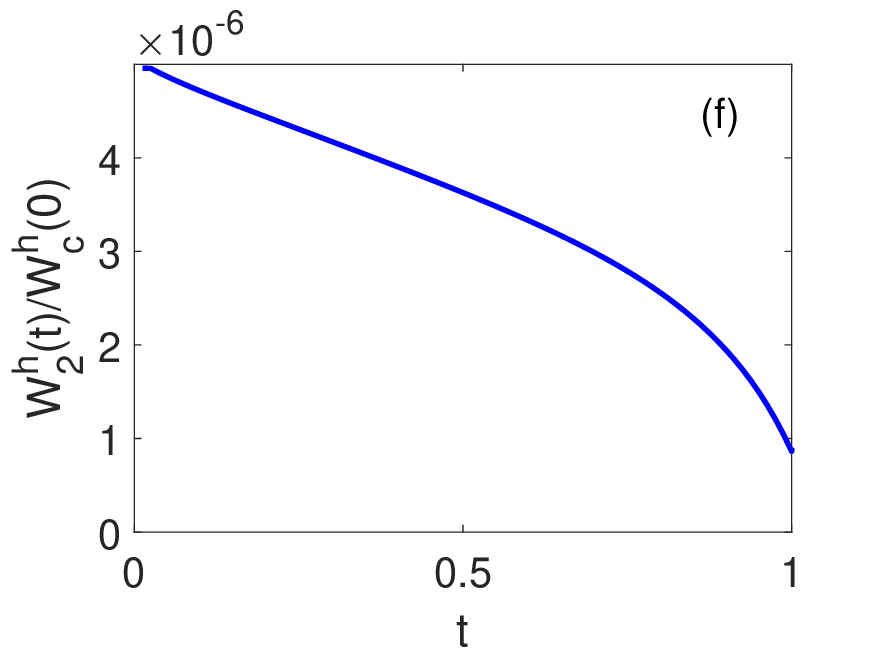}}\vspace{-3mm}
	\caption{Numerical simulation of the CSF starting from an ellipse (Curve 2). Quantitative diagnostics: (a) evolution; (b) relative area loss $\Delta A^h(t)$; (c) mesh ratio $\Psi(t)$; (d) normalized total energy $W^h_c(t)/W_c^h(0)$; (e) normalized energy component $W_1^h(t)/W_c^h(0)$; (f) normalized energy component $W_2^h(t)/W_c^h(0)$.}
\label{fig1}
\end{figure}

\begin{figure}[h!]
	\centering
    \subfigure[\empty]{ 
\includegraphics[trim=0 0 55 0,clip,width=0.30\textwidth,height=0.25\textwidth]
{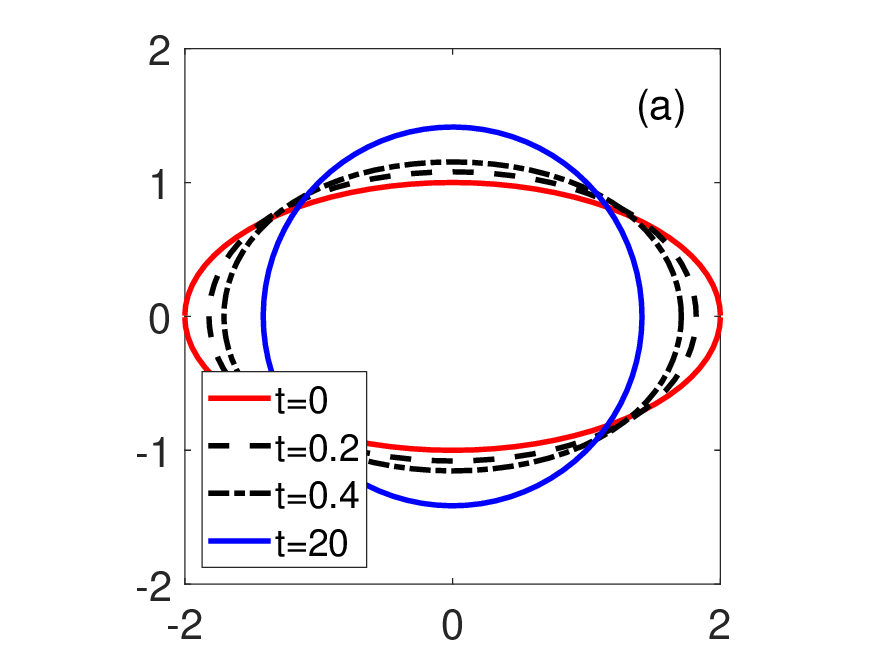}}
\subfigure[\empty]{ 
\includegraphics[trim=0 0 28 0,clip,width=0.30\textwidth,height=0.25\textwidth]
{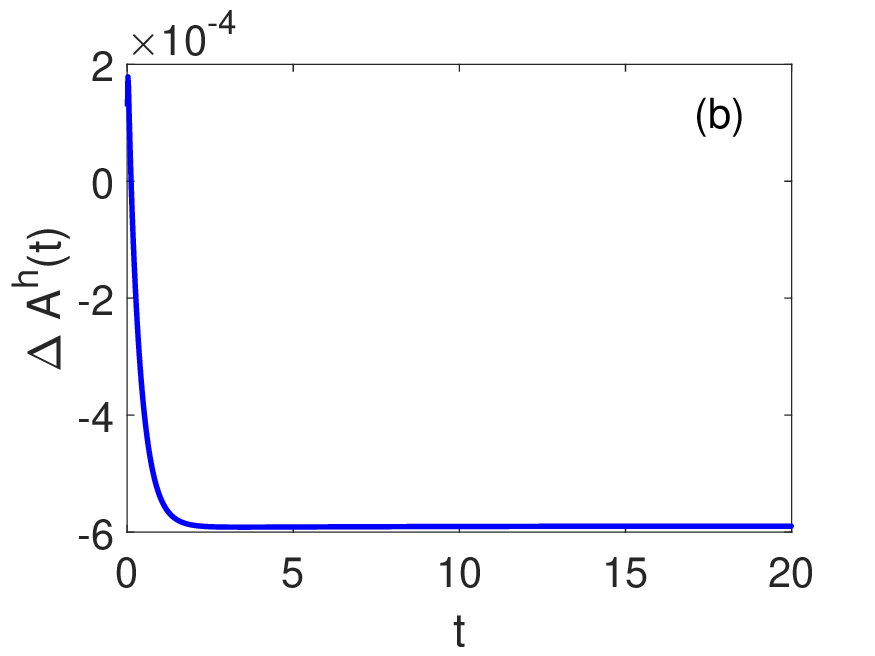}}
\subfigure[\empty]{ 
\includegraphics[trim=0 0 28 0,clip,width=0.30\textwidth,height=0.25\textwidth]
{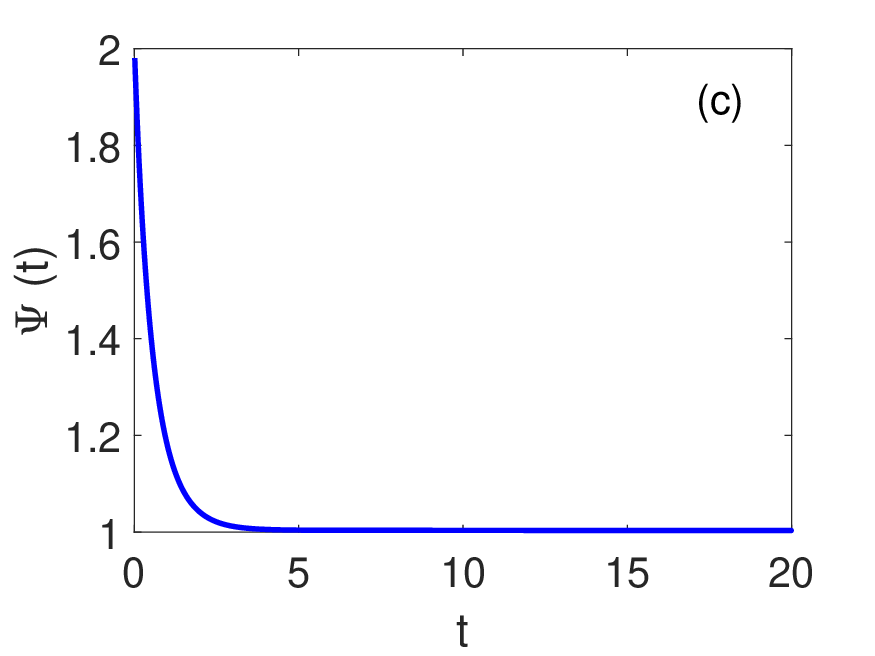}}\\\vspace{-3mm}
\subfigure[\empty]{ 
\includegraphics[trim=0 0 28 0,clip,width=0.30\textwidth,height=0.25\textwidth]
{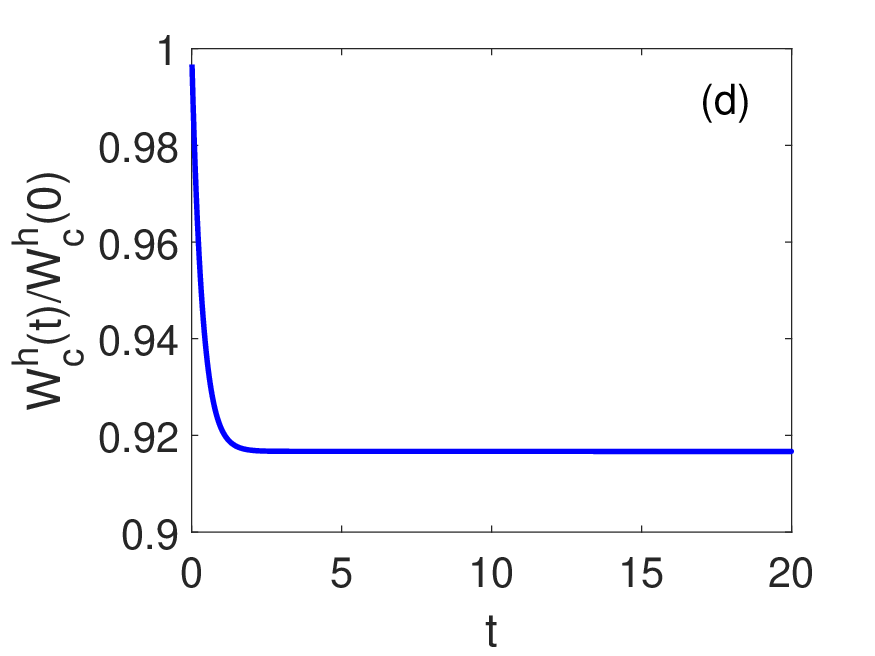}}
\subfigure[\empty]{ 
\includegraphics[trim=0 0 28 0,clip,width=0.30\textwidth,height=0.25\textwidth]
{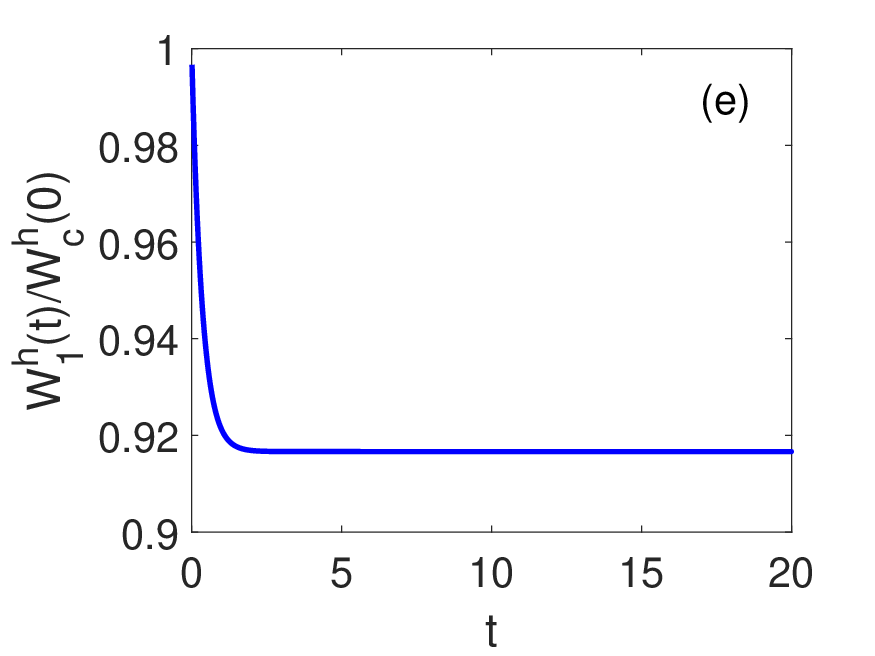}}
\subfigure[\empty]{ 
\includegraphics[trim=0 0 28 0,clip,width=0.30\textwidth,height=0.25\textwidth]
{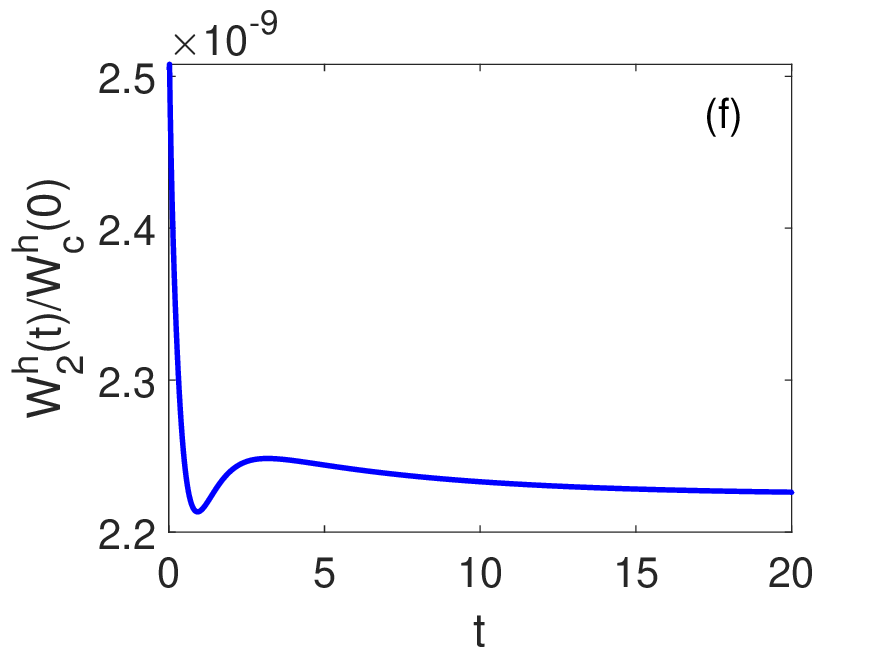}}\vspace{-3mm}
	\caption{Numerical simulation of the AP-CSF starting from an ellipse with penalty factor $\alpha=1/h$. Quantitative diagnostics: (a) evolution; (b) relative area loss $\Delta A^h(t)$; (c) mesh ratio $\Psi(t)$; (d) normalized total energy $W^h_c(t)/W_c^h(0)$; (e) normalized energy component $W_1^h(t)/W_c^h(0)$; (f) component $W_2^h(t)/W_c^h(0)$.}
\label{fig2}
\end{figure}

Figs. \ref{fig1}-\ref{fig2} display the evolution of the ellipse under the two flows along with time histories of three key geometric quantities: (i) normalized energy, (ii) relative area loss, and (iii) mesh ratio.
For the AP-CSF, the relative area loss stabilizes at approximately $6\times 10^{-4}$ (Fig. \ref{fig2}(b)), and the mesh ratio remains well-behaved throughout the simulation (Fig. \ref{fig2}(c)). Panels \ref{fig1}(d)-\ref{fig2}(d) confirm that the LDG discretization preserves the expected perimeter-decreasing property for both flows. The component energies $W_1^h(t)/W_c^h(0)$ and $W_2^h(t)/W_c^h(0)$ are plotted in Figs. \ref{fig1}(e)-\ref{fig2}(e) and Figs. \ref{fig1}(f)-\ref{fig2}(f), respectively. Notably, $W_2^h$ is negligible---at the order of $10^{-9}$---in Fig. \ref{fig2}(f), so $W_c^h$ and $W_1^h$ differ only slightly, in agreement with the energy decomposition presented in \cite{Li2021}.


We repeat the AP-CSF computation in Fig.~\ref{fig2} with the same numerical parameters but a highly nonuniform initial mesh. The initial parameter values are generated by the monotone map
\[
\rho\longmapsto \rho+\frac{0.95}{2\pi}\sin(2\pi\rho),
\]
whose derivative has a maximum-to-minimum ratio of \(39\). As shown in Fig.~\ref{fig2_nonuniform}, the initially clustered nodes redistribute during the evolution, with the mesh ratio decreasing from \(66.69425\) at \(t=0\) to \(1.000122\) at \(t=20\). The regular evolution and substantial improvement in mesh uniformity support the scheme's robustness to poor initial mesh quality in this test.

\begin{figure}[h!]
	\centering
\subfigure[\empty]{
\includegraphics[trim=0 0 0 0,clip,width=0.22\textwidth,height=0.21\textwidth]
{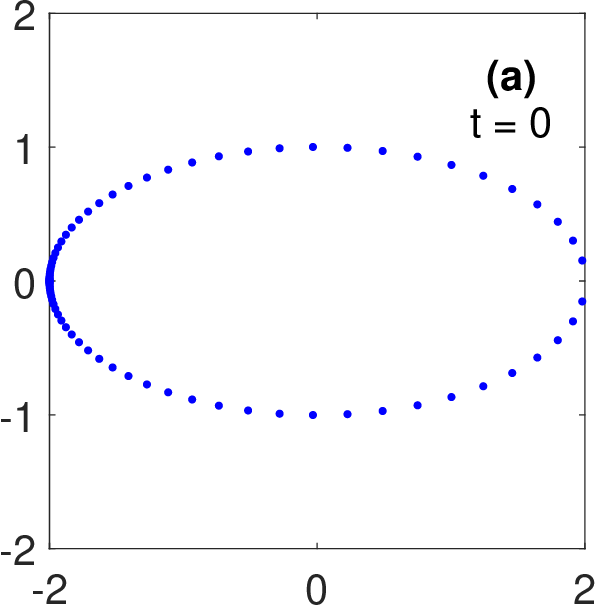}}
\subfigure[\empty]{
\includegraphics[trim=0 0 0 0,clip,width=0.22\textwidth,height=0.21\textwidth]
{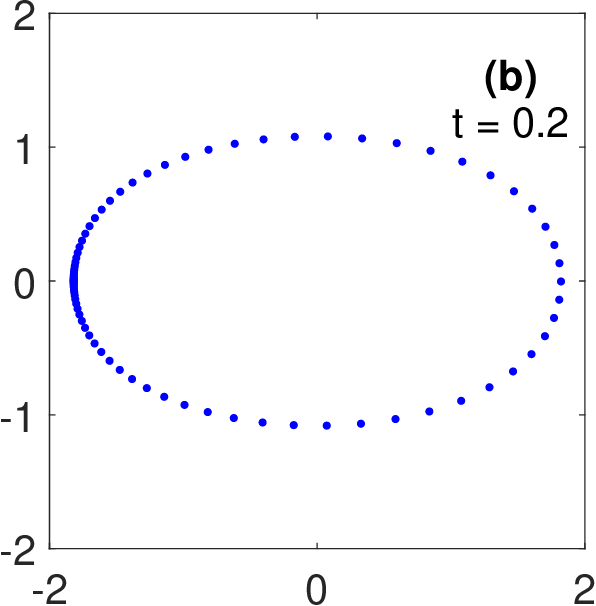}}
\subfigure[\empty]{
\includegraphics[trim=0 0 0 0,clip,width=0.22\textwidth,height=0.21\textwidth]
{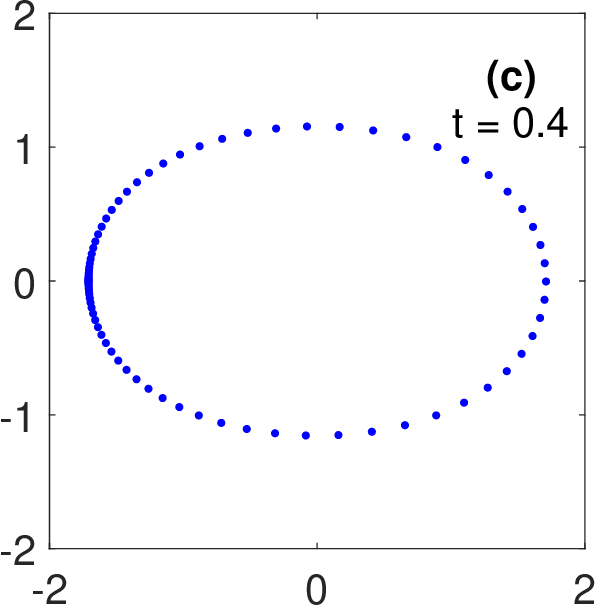}}
\subfigure[\empty]{
\includegraphics[trim=0 0 0 0,clip,width=0.22\textwidth,height=0.21\textwidth]
{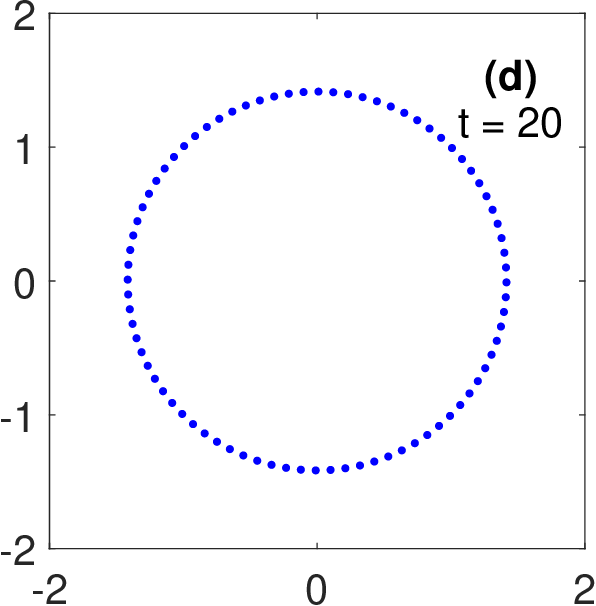}}\\[-1mm]
\subfigure[\empty]{
\includegraphics[trim=0 0 0 0,clip,width=0.48\textwidth,height=0.25\textwidth]
{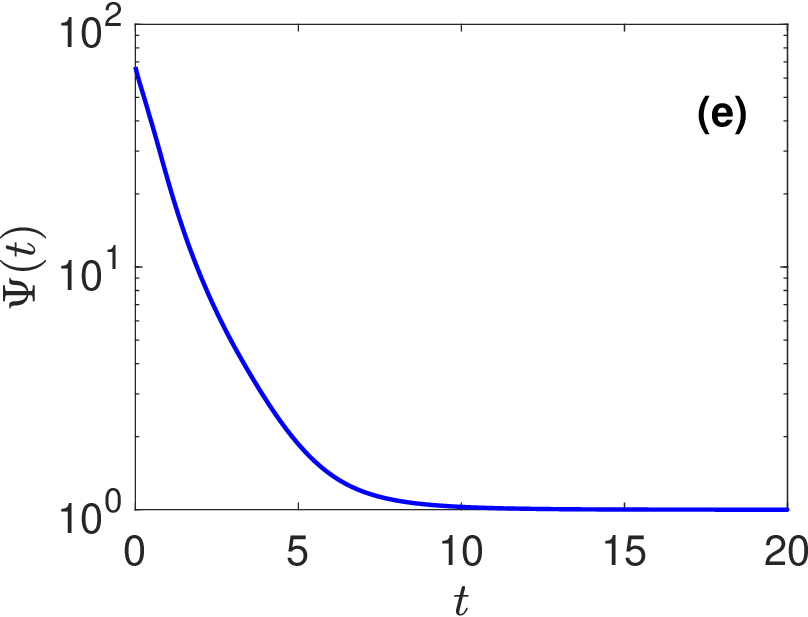}}
	\caption{ Simulation of the AP-CSF starting from an ellipse with a severely nonuniform distribution of mesh points, using the same parameters as in Fig.~\ref{fig2}. Panels (a)--(d) show the nodes at $t=0,0.2,0.4,20$, respectively, and panel (e) shows the mesh ratio $\Psi(t)$.}
\label{fig2_nonuniform}
\end{figure}

\begin{figure}[htbp!]\hspace{-8mm}
    \subfigure[\empty]{ 
\includegraphics[trim=0 0 50 0,clip,width=0.29\textwidth,height=0.24\textwidth]
{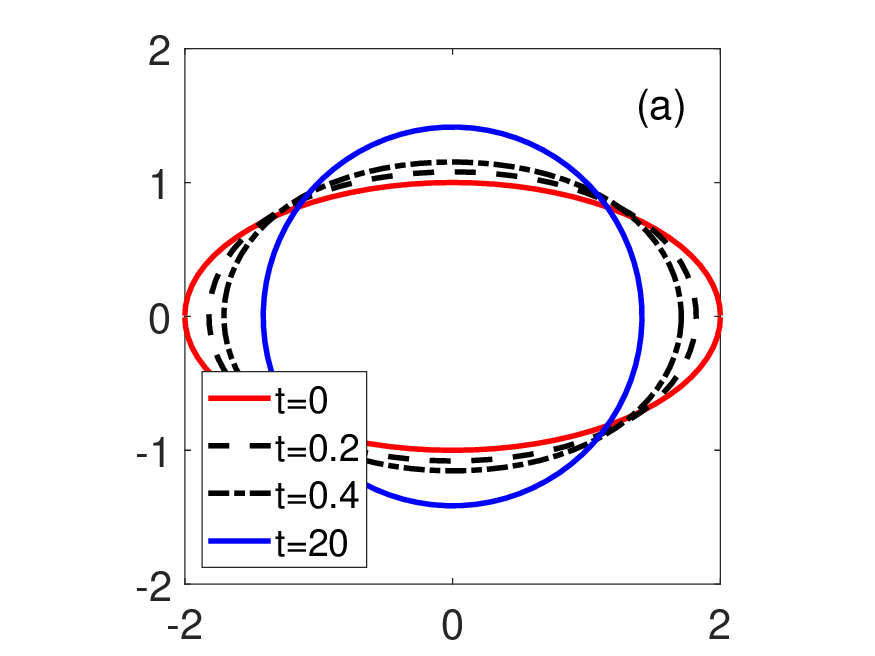}}\hspace{-4mm}
\subfigure[\empty]{ 
\includegraphics[trim=0 0 28 0,clip,width=0.25\textwidth,height=0.24\textwidth]
{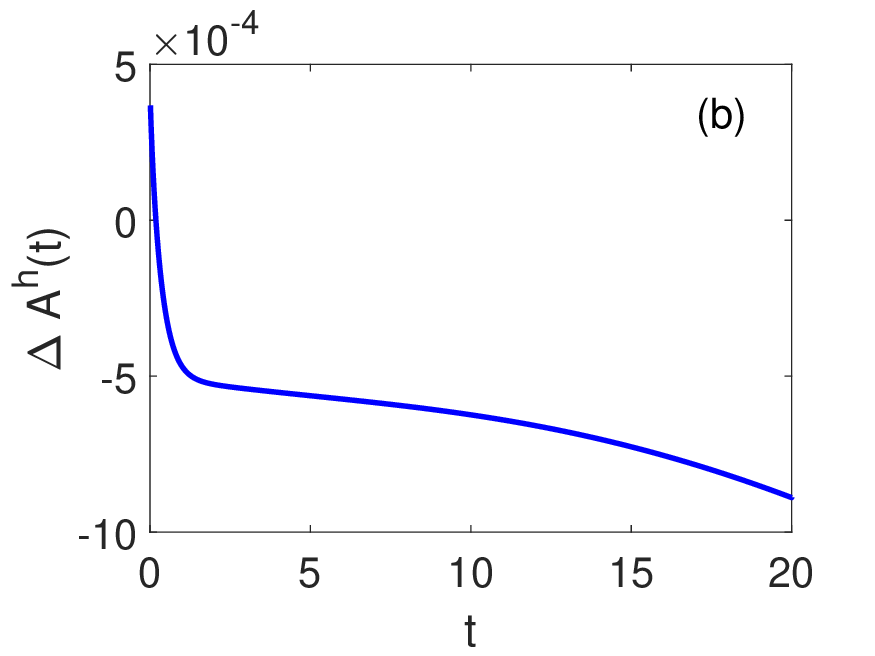}}\hspace{-2mm}
\subfigure[\empty]{ 
\includegraphics[trim=0 0 28 0,clip,width=0.25\textwidth,height=0.24\textwidth]
{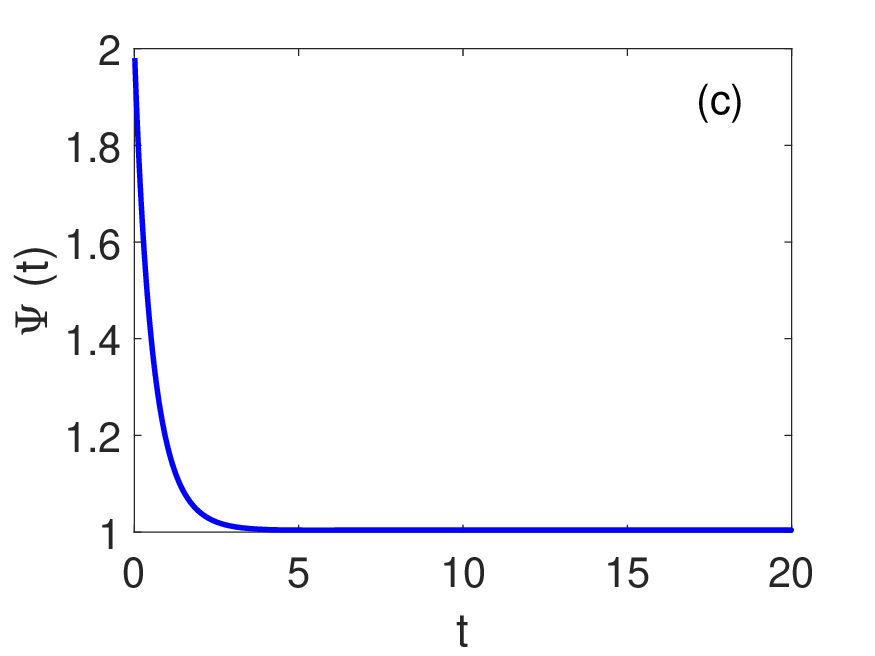}}\hspace{-2mm}
\subfigure[\empty]{ 
\includegraphics[trim=0 0 28 0,clip,width=0.25\textwidth,height=0.24\textwidth]
{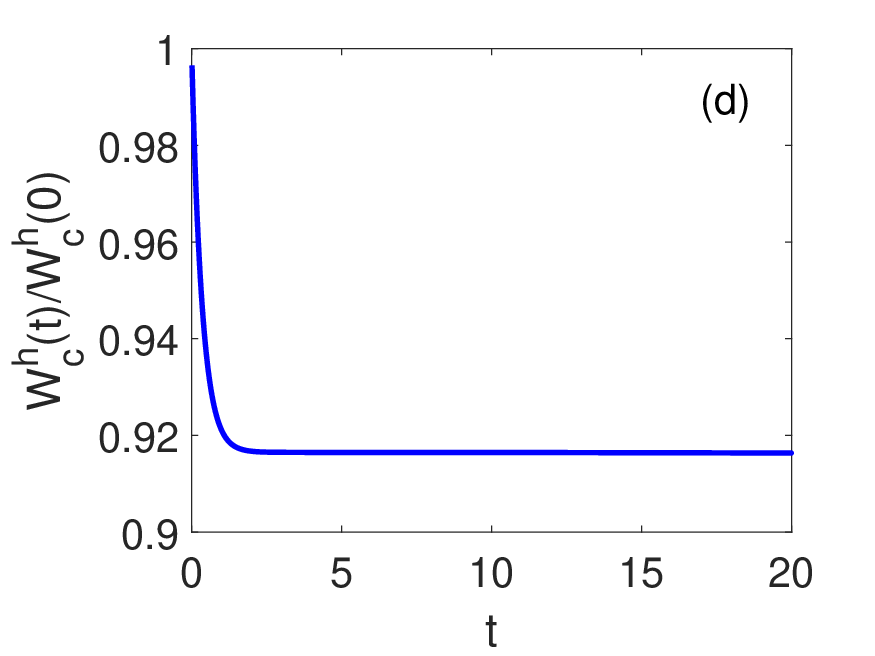}}\vspace{-3mm}
	\caption{Numerical simulation of the AP-CSF starting from an ellipse with $\alpha=0$: (a) geometric evolution; (b) normalized relative area loss $\Delta A^h(t)$; (c) mesh ratio $\Psi(t)$; (d) normalized total energy $W^h_c(t)/W_c^h(0)$.}
\label{fig3}
\end{figure}

\begin{figure}[htbp!]
	\centering
    \subfigure[\empty]{ 
\includegraphics[trim=0 0 50 0,clip,width=0.30\textwidth,height=0.25\textwidth]
{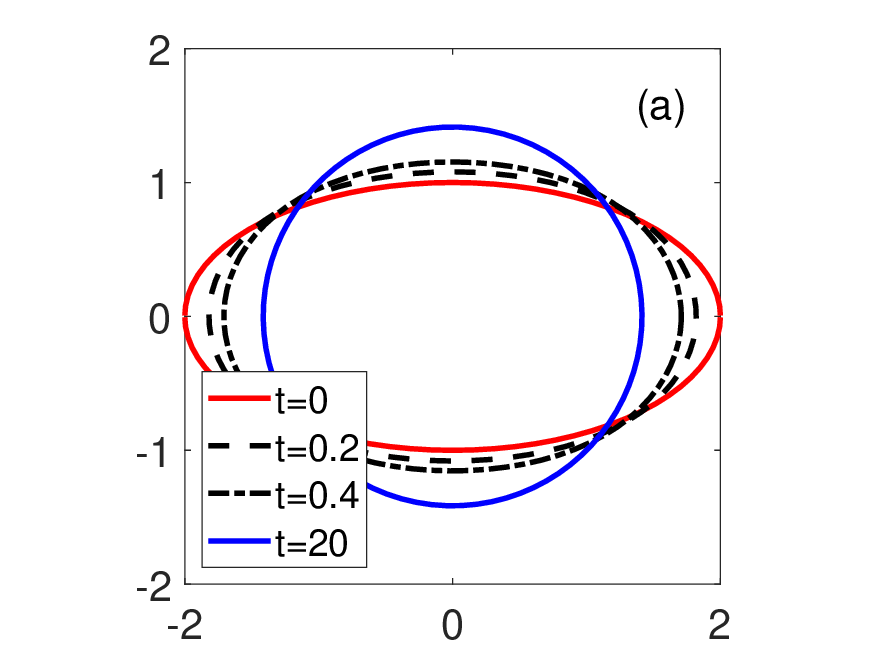}}
\subfigure[\empty]{ 
\includegraphics[trim=0 0 28 0,clip,width=0.30\textwidth,height=0.25\textwidth]
{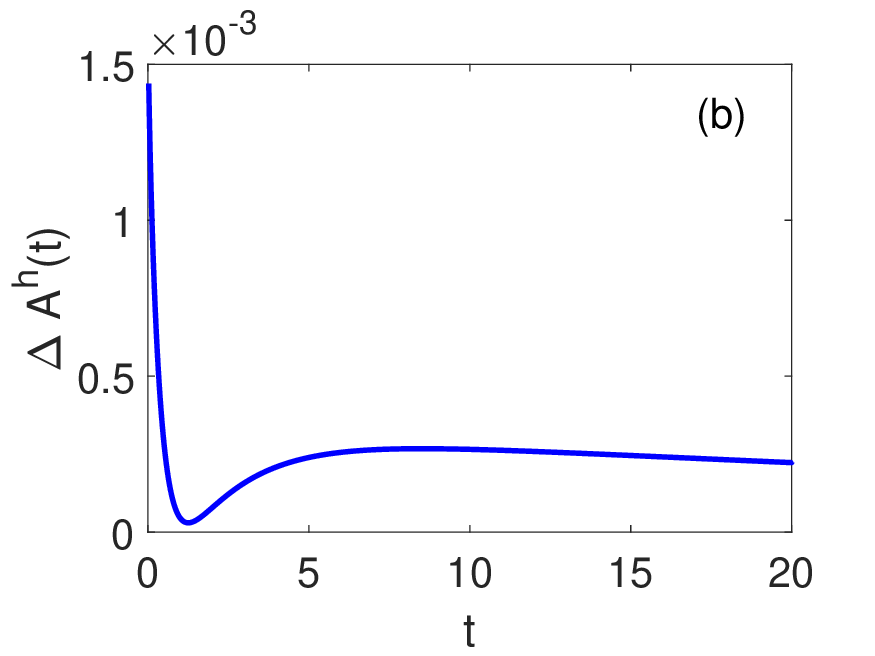}}
\subfigure[\empty]{ 
\includegraphics[trim=0 0 28 0,clip,width=0.30\textwidth,height=0.25\textwidth]
{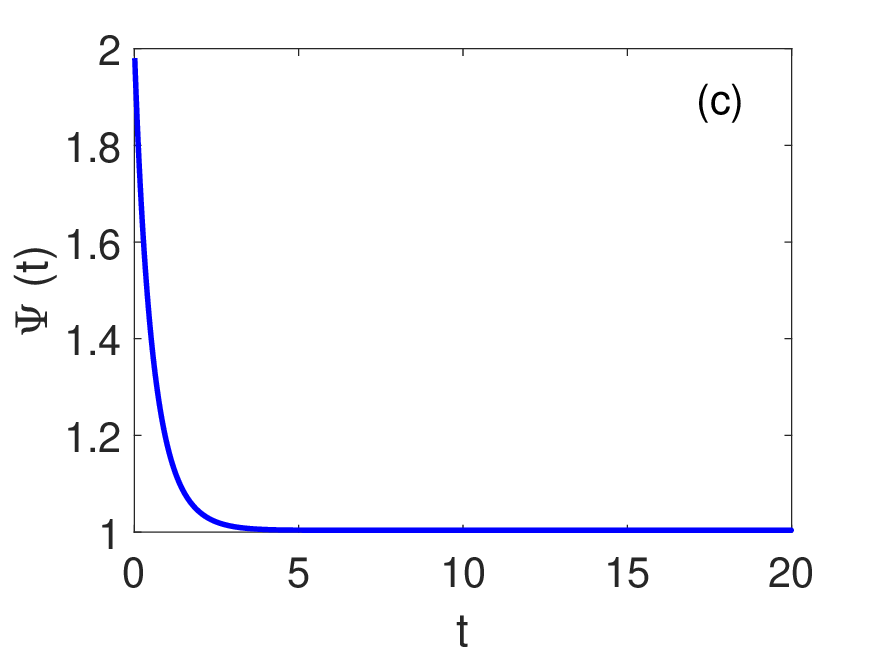}}\\\vspace{-3mm}
\subfigure[\empty]{ 
\includegraphics[trim=0 0 28 0,clip,width=0.30\textwidth,height=0.25\textwidth]
{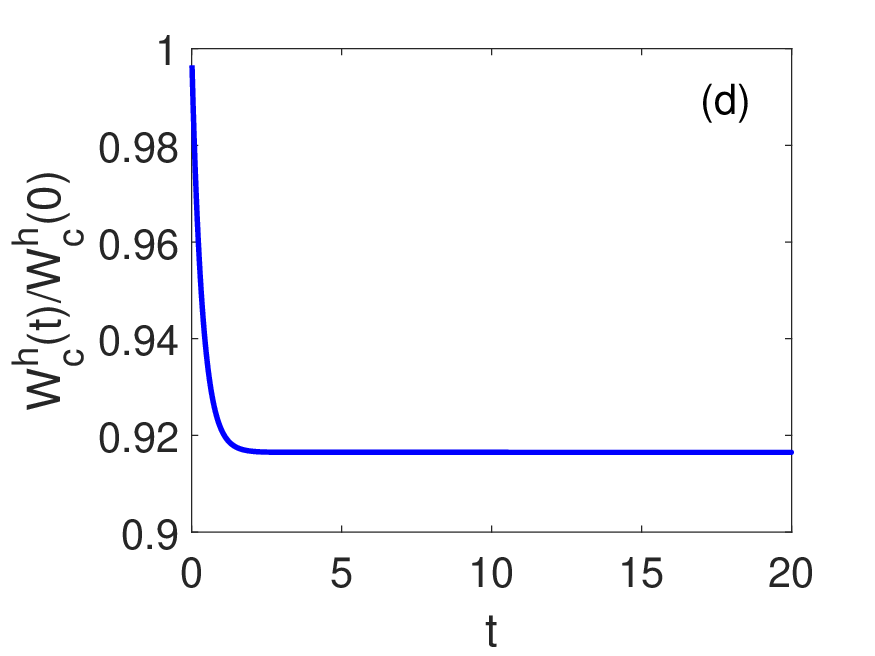}}
\subfigure[\empty]{ 
\includegraphics[trim=0 0 28 0,clip,width=0.30\textwidth,height=0.25\textwidth]
{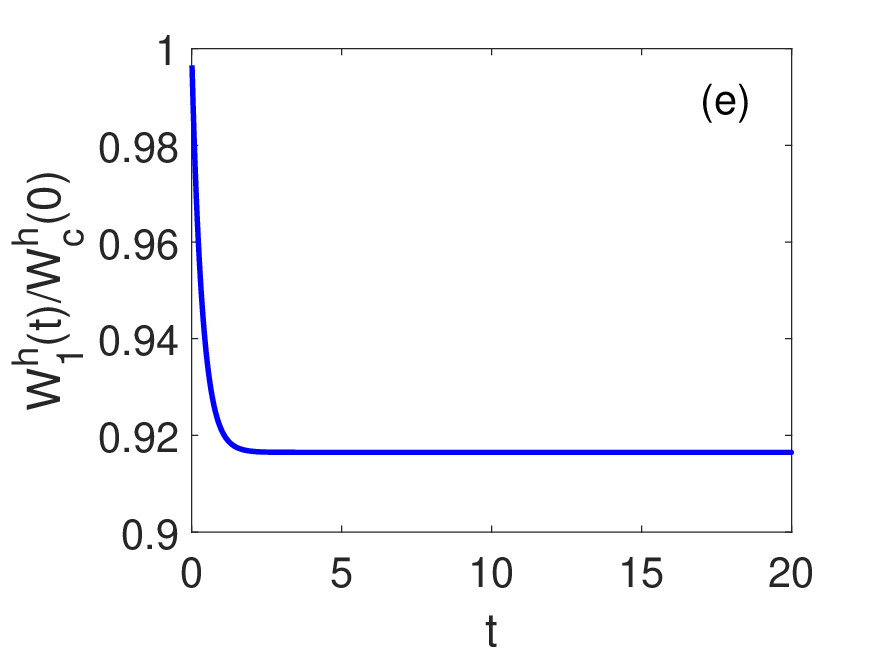}}
\subfigure[\empty]{ 
\includegraphics[trim=0 0 28 0,clip,width=0.30\textwidth,height=0.25\textwidth]
{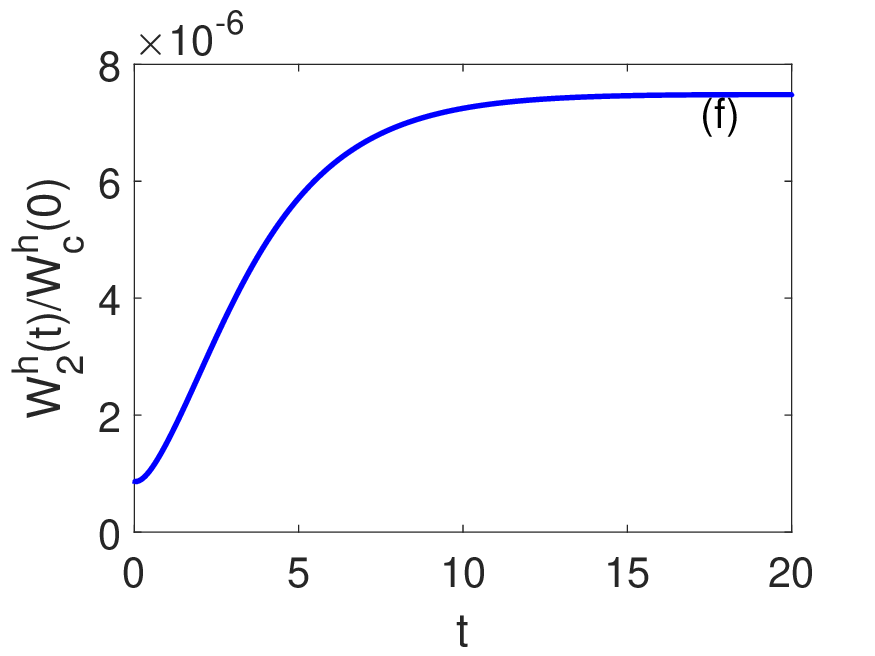}}\vspace{-3mm}
	\caption{Numerical simulation of the AP-CSF starting from an ellipse with $\alpha=0.1$: (a) evolution; (b) relative area loss $\Delta A^h(t)$; (c) mesh ratio $\Psi(t)$; (d) normalized total energy $W^h_c(t)/W_c^h(0)$; (e) normalized energy component $W_1^h(t)/W_c^h(0)$; (f) normalized energy component $W_2^h(t)/W_c^h(0)$.}
\label{fig4}
\end{figure}

\begin{figure}[htbp!]
	\centering
\includegraphics[trim=0 0 0 0,clip,width=0.9\textwidth]
{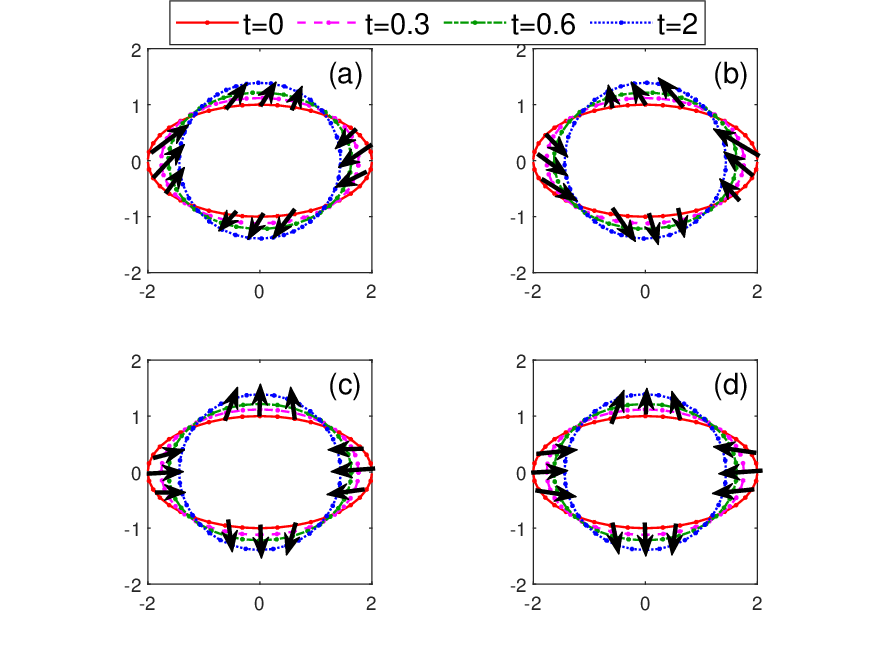}\vspace{-6mm}
	\caption{Tangential velocity on the initial elliptical curve under the AP-CSF for different numerical fluxes: (a) $\widehat{\bm{\xi}}=\bm{\xi}^+$; (b) $\widehat{\bm{\xi}}=\bm{\xi}^-$; (c) $\widehat{\bm{\xi}}=\bm{\xi}^++(\boldsymbol{X}^+-\boldsymbol{X}^-)/h$; (d) $\widehat{\bm{\xi}}=\bm{\xi}^-+(\boldsymbol{X}^+-\boldsymbol{X}^-)/h$.}
\label{fig5}
\end{figure}

To assess the effect of the penalty term, we compare \(\alpha=0\), \(0.1\), and \(1/h\). The relative area error accumulates over time for \(\alpha=0\) (Fig.~\ref{fig3}(b)). This accumulation is reduced but persists for \(\alpha=0.1\) (Fig.~\ref{fig4}(b)), whereas for \(\alpha=1/h\) the error rapidly stabilizes at a small level (Fig.~\ref{fig2}(b)).

 Fig.~\ref{fig5} highlights an interesting contrast in nodal motion among the four numerical fluxes, with arrows indicating the directions of motion. Although the curve shapes evolve similarly, the tangential motion differs markedly: in panels (a) and (b), the nodes appear to circulate along the curve in opposite directions, whereas no comparable rotational drift is evident in panels (c) and (d) over the displayed time interval.


\subsection{Anisotropic case}
\label{subsec:numerical-anisotropic}

\subsubsection[Convergence tests]{Convergence tests}

\begin{example}[Convergence tests]
 To test the method on an anisotropic problem with a nontrivial available exact solution, we consider the standard mobility-weighted anisotropic CSF~\cite{Barrett200828}:
\[ V=-\beta_{\mathrm{mob}}(\boldsymbol n)\mu,
\qquad
\beta_{\mathrm{mob}}(\boldsymbol n)=\gamma_\varepsilon(\boldsymbol n),\]
where \[\gamma_\varepsilon(\boldsymbol n)
=\sqrt{n_1^2+\varepsilon^2n_2^2}
\]
is the elliptic anisotropy. For a unit normal $\boldsymbol{n}=(-\sin\theta, \cos\theta)^T$, this homogeneous anisotropy can be written as
\[
\gamma_\varepsilon(\theta)
=\sqrt{\sin^2\theta+\varepsilon^2\cos^2\theta}.
\]
We use this angular representation to construct $G(\theta)$.
The flow admits the self-similarly shrinking ellipse
\[
\boldsymbol X_{\mathrm{ex}}(\rho,t)
=\sqrt{1-2t}
\begin{pmatrix}
\cos g(\rho)\\
\varepsilon\sin g(\rho)
\end{pmatrix},
\qquad
g(\rho)=2\pi\rho+0.1\sin(2\pi\rho),
\]
for $0\leq t<1/2$. The function $g$ produces a nonuniform initial distribution of the nodes.
\end{example}

In the computation, we retain the discrete relations \eqref{3.1b}--\eqref{3.1d} and the numerical fluxes \eqref{3.2}, replacing only the normal-velocity equation \eqref{3.1a} with
\[\int_{I_j}
\frac{\boldsymbol X^{m+1}-\boldsymbol X^m}{\tau}
\cdot\boldsymbol n^{\star m}\eta\,d\rho
+\int_{I_j}Q^m\gamma_\varepsilon(\boldsymbol n^m)
\mu^{m+1}\eta\,d\rho=0,
\qquad \forall\eta\in V_h.
\]
 In \eqref{3.1c}, the matrix $G(\theta^m)$ is constructed from $\gamma_\varepsilon$ according to \eqref{1.9}.

 Table~\ref{tab:elliptic-exact} reports the manifold distance at
$T=0.25$ for $\varepsilon=0.1$. The results show approximately $k+1$-th convergence for the $P^k$ approximations, respectively, with the refinement strategy $\tau=h^{k+1}$, when measured against this nontrivial anisotropic exact solution.

\begin{table}[htbp!]
    \centering
    \caption{Convergence of the manifold distance to the exact shrinking ellipse for the mobility-weighted anisotropic CSF with $\varepsilon=0.1$. The computations use $P^k$ elements, $\tau=h^{k+1}$, and $T=0.25$.}
    \resizebox{\textwidth}{!}{%
    \begin{tabular}{ccc@{\hspace{1.2em}}ccc@{\hspace{1.2em}}ccc@{\hspace{1.2em}}ccc}
        \toprule
        \multicolumn{3}{c}{$P^1$}
        & \multicolumn{3}{c}{$P^2$}
        & \multicolumn{3}{c}{$P^3$}
        & \multicolumn{3}{c}{$P^4$} \\
        \cmidrule(lr){1-3}\cmidrule(lr){4-6}\cmidrule(lr){7-9}\cmidrule(lr){10-12}
        $N$ & error & order
        & $N$ & error & order
        & $N$ & error & order
        & $N$ & error & order \\
        \midrule
         5 & 2.01e-02 & --
        & 5 & 2.50e-03 & --
        & 5 & 5.18e-04 & --
        & 5 & 1.04e-04 & -- \\
        10 & 4.95e-03 & 2.02
        & 10 & 3.24e-04 & 2.95
        & 10 & 3.26e-05 & 3.99
        & 10 & 3.18e-06 & 5.03 \\
        20 & 1.22e-03 & 2.02
        & 20 & 4.08e-05 & 2.99
        & 20 & 2.04e-06 & 4.00
        & 15 & 4.07e-07 & 5.07 \\
        40 & 3.05e-04 & 2.01
        & 40 & 5.11e-06 & 3.00
        & 40 & 1.38e-07 & 3.89
        & 20 & 1.02e-07 & 4.82 \\
        \bottomrule
    \end{tabular}}
    \label{tab:elliptic-exact}
\end{table}

\begin{example}[Convergence tests]
We evaluate the spatial convergence of the proposed LDG method for CSF and AP-CSF with the surface energy density $\gamma(\theta) = 1+\beta\cos(4\theta)$, using $\beta=0.05$ for weak anisotropy and $\beta=0.067$ and $\beta=0.07$ for strong anisotropy.
\end{example}

\begin{table}[h!]
	\centering
	\caption{Convergence of the manifold distance for the anisotropic CSF, using $P^k$ elements ($k = 1, 2, 3,4$), where the initial curve is chosen as Curve 1. Time step and final time: $\tau=5h^{k+1}$, $T=0.25$.}
\resizebox{\textwidth}{!}{%
        \begin{tabular}{llll @{\hspace{0.8em}} lll@{\hspace{0.8em}}lll@{\hspace{0.8em}}lll}
		\toprule
            &$P^1$&&&$P^2$&&&$P^3$&&&$P^4$&&\\
            \cmidrule(r){2-4} \cmidrule(r){5-7} \cmidrule(r){8-10}\cmidrule(r){11-13}
		&$N$   & error& order&$N$   & error& order&$N$   & error& order&$N$   & error& order\\
            \bottomrule
$\beta=0.05$&5 & 2.38e-01  &-      &5 & 1.01e-01  &-    &5 & 2.30e-02  &-    & 5 & 5.10e-03  &-\\
		  &10& 6.76e-02  & 1.81 & 10& 1.29e-02  &2.96 &10& 1.43e-03  &4.00 & 10& 1.64e-04  &4.95\\
            &20& 1.61e-02  &2.06  & 20& 1.45e-03  &3.15 &20& 8.39e-05  &4.09 & 15& 2.17e-05  &4.98\\
		  &40& 3.43e-03 & 2.23  & 40& 1.67e-04  &3.13 &40& 5.08e-06  &4.04 & 20& 5.16e-06  &4.99\\
	\\
$\beta=0.067$&5 & 2.39e-01  &-      &5 & 9.93e-02  &-    &5 & 2.15e-02  &-    & 5 & 9.96e-03  &-\\
		  &10& 6.24e-02  & 1.93 & 10& 1.36e-02  &3.06 &10& 1.46e-03  &3.87 & 10& 3.06e-04  &5.02\\
            &20& 1.38e-02  &2.17  & 20& 1.62e-03  &3.15 &20& 9.52e-05  &3.94 & 15& 3.94e-05  &5.06\\
		  &40& 3.45e-03 & 2.00  & 40& 1.93e-04  &3.16 &40& 5.93e-06  &4.00 & 20& 9.10e-06  &5.09\\
 \\
$\beta=0.07$&5 & 4.40e-01  &-      &5 & 8.32e-02  &-    &5 & 2.53e-02  &-    & 5 & 1.07e-02  &-\\
		  &10& 1.27e-01  & 1.78 & 10& 1.47e-02  &2.49 & 10 & 2.31e-03  &3.45 &10 & 6.79e-04  & 3.99\\
            &20& 2.35e-02  &2.44 &20& 1.86e-03  &2.98 &20& 3.99e-04  &2.53 &15& 3.26e-04  &1.80\\
		  &40&6.08e-03  & 1.94  &40& 1.81e-04 &3.35  &40&3.84e-04  & 0.05 &20& 3.19e-04 & 0.07 \\
		\bottomrule
	\end{tabular}}
    \label{tab3}
\end{table}

\begin{table}[h!]
	\centering
	\caption{Convergence of the manifold distance for the anisotropic AP-CSF, using $P^k$ elements in space, where the initial curve is chosen as Curve 2. Time step and final time: $\tau=5h^{k+1}$, $T=0.25$.}
\resizebox{\textwidth}{!}{%
        \begin{tabular}{llll @{\hspace{0.8em}} lll@{\hspace{0.8em}}lll@{\hspace{0.8em}}lll}
		\toprule
            &$P^1$&&&$P^2$&&&$P^3$&&&$P^4$&&\\
            \cmidrule(r){2-4} \cmidrule(r){5-7} \cmidrule(r){8-10}\cmidrule(r){11-13}
		&$N$   & error& order&$N$   & error& order&$N$   & error& order&$N$   & error& order\\
            \bottomrule
$\beta=0.05$&5 & 6.67e-01  &-      &5 & 7.75e-02  &-    &5 & 1.83e-02  &-    & 5 & 1.03e-02  &-\\
		  &10& 1.59e-01  & 2.06 & 10& 8.22e-03  &3.23 &10& 1.10e-03  &4.05 & 10& 3.03e-04  &5.08\\
            &20& 3.93e-02  &2.02  & 20& 9.26e-04  &3.15 &20& 6.86e-05  &4.00 & 15& 3.87e-05  &5.07\\
		  &40& 9.76e-03 & 2.01  & 40& 9.76e-05  &3.24 &40& 4.02e-06  &4.09 & 20& 9.11e-06  &5.02\\
	\\
$\beta=0.067$&5 & 7.29e-01  &-      &5 & 1.05e-01  &-    &5 & 9.57e-02  &-    & 5 & 6.56e-02  &-\\
		  &10& 1.59e-01  & 2.19 & 10& 1.18e-02  &3.16 &10& 5.59e-03  &4.09 & 10& 2.03e-03  &5.00\\
            &20& 4.02e-02  &1.98  & 20& 1.43e-03  &3.04 &20& 3.39e-04  &4.04 & 15& 2.67e-04  &5.01\\
		  &40& 9.47e-03 & 2.08  & 40& 1.80e-04  &2.99 &40& 2.06e-05  &4.04 & 20& 6.27e-05  &5.03\\
          \\
$\beta=0.07$&5 & 7.44e-01  &-      &5 & 8.42e-02  &-    &5 & 9.58e-02  &-    & 5 & 6.59e-02  &-\\
		  &10& 1.50e-01  & 2.30 & 10& 7.60e-03  &3.46 & 10 & 4.95e-03  &4.27 &10 & 2.21e-03  & 4.89\\
            &20& 4.02e-02  &1.89 &20& 4.22e-03  &0.84 &20& 4.16e-03  &0.24 &15& 2.08e-03  &0.15\\
		  &40&9.47e-03  & 2.08  &40& 2.25e-03 &0.91  &40&3.87e-03  & 0.10 &20& 2.05e-03 & 0.05 \\
		\bottomrule
	\end{tabular}}
    \label{tab4}
\end{table}

Tables \ref{tab3}-\ref{tab4} present the manifold distances and convergence rates at $T=0.25$. For $\beta=0.05$ and $0.067$, the observed rates are close to the expected order $k+1$. For $\beta=0.07$, however, the high-order rates deteriorate on finer meshes. The surface stiffness satisfies
\[
\gamma(\theta)+\gamma''(\theta)=1-15\beta\cos(4\theta),
\quad
\min_\theta(\gamma+\gamma'')=1-15\beta=-0.05
\quad \mathrm{for}\quad \beta=0.07.\]
The evolution is therefore not uniformly parabolic, and the regularity required for optimal higher-order convergence is not guaranteed.

Mesh refinement reveals increasingly localized regions of high curvature. As a representative diagnostic, we examine the $P^2$ AP-CSF solutions with $\beta=0.07$ and $\tau=5h^3$,
using the local curvature indicator $|\kappa_{h,i}|=2|\Delta\theta_i|/(\ell_{i-1}+\ell_i)$, where $\Delta\theta_i$ is the turning angle between consecutive segments of the reconstructed polygonal curve and $\ell_i$ is the length of segment $i$. After excluding a small number of samples adjacent to element interfaces, the upper-tail values of this indicator increase under refinement, while the mesh ratio remains moderate. These results provide no evidence of a mesh-independent curvature bound at $T=0.25$ and suggest that curvature concentration contributes to the loss of higher-order convergence.

As a complementary equilibrium benchmark, we compare the
$P^2$ AP-CSF solution at $T=2.5$ with the area-matched Wulff shape for $\beta=0.07$. Since $T=2.5$ is sufficiently large, the solution is already very close to its equilibrium state. Here, the area-matched Wulff shape is the Wulff shape (see the next subsection) enclosing the same area as the initial curve ; it is the equilibrium shape of the anisotropic AP-CSF \cite{Li2021,Bao202361}. Table~\ref{tab:beta007-diagnostic} shows that the discrepancy from this exact equilibrium also decreases at a reduced rate under mesh refinement. This comparison avoids errors from a numerically generated reference solution.

\begin{table}[h!]
    \centering
    \caption{Representative $P^2$ diagnostics for $\beta=0.07$. The first block uses AP-CSF at $T=0.25$ with $\tau=5h^3$; the curvature indicator is evaluated away from element interfaces. The second block uses AP-CSF at $T=2.5$ with the area-matched Wulff equilibrium as the exact target.}
    \setlength{\tabcolsep}{3pt}
    \begin{tabular*}{\textwidth}{@{\extracolsep{\fill}}llrrrr}
        \toprule
        \multirow{3}{*}{\parbox{3.0cm}{\centering\textit{$T=0.25$: local regularity diagnostic}}}
        & $N$ & 10 & 20 & 40 & 80\\[-1mm]
        \cmidrule(lr){2-6}
        & $\max |\kappa_h|_{\rm int}$ & 1.482 & 2.664 & 4.251 & 6.257\\
        & mesh ratio & 2.668 & 2.729 & 2.848 & 3.438\\
        \midrule
        \multirow{3}{*}{\parbox{3.0cm}{\centering\textit{$T=2.5$: Wulff-equilibrium error}}}
        & $N$ & 5 & 10 & 20 & 40\\[-1mm]
        \cmidrule(lr){2-6}
        & error & 5.31e-02 & 5.82e-03 & 3.71e-03 & 3.68e-03\\
        & order & -- & 3.19 & 0.65 & 0.01\\
        \bottomrule
    \end{tabular*}
    \label{tab:beta007-diagnostic}
\end{table}

\subsubsection[Curve evolution and geometric properties]{Curve evolution and geometric properties}
In this part, we investigate the evolution of Curves 2-4 under different anisotropy strengths. For the area-preserving flows considered below, the equilibrium shapes are given by the Wulff construction. For $\gamma(\theta)\in C^1([-\pi,\pi])$, the associated Wulff envelope is parametrized by~\cite{Bao2017,Peng1998}:
\[
    \left\{
    \begin{aligned}
        &x(\theta)=-\gamma(\theta)\sin\theta-\gamma'(\theta)\cos\theta,\nonumber\\
        &y(\theta)=\gamma(\theta)\cos\theta-\gamma'(\theta)\sin\theta,\nonumber\\
    \end{aligned}
    \quad \theta\in[-\pi,\pi].
    \right.
\]
Under strong anisotropy, this envelope may develop self-intersecting branches, commonly called ``ears". Removing these branches according to the Wulff construction yields the equilibrium Wulff shape. For comparison with the numerical solutions of AP-CSF, we rescale each Wulff shape to enclose the same area as the corresponding initial curve.

\begin{example}
We examine the evolution of several curves under CSF and AP-CSF for some anisotropic surface energy densities. Unless otherwise specified, we employ $N=80$ spatial cells, a time step  $\tau=10^{-3}$, a penalty parameter $\alpha=1/h$, and the piecewise linear LDG method ($k=1$).
\end{example}

\begin{figure}[h!]
	\centering
    \subfigure[\empty]{ 
\includegraphics[trim=0 0 55 0,clip,width=0.30\textwidth,height=0.22\textwidth]
{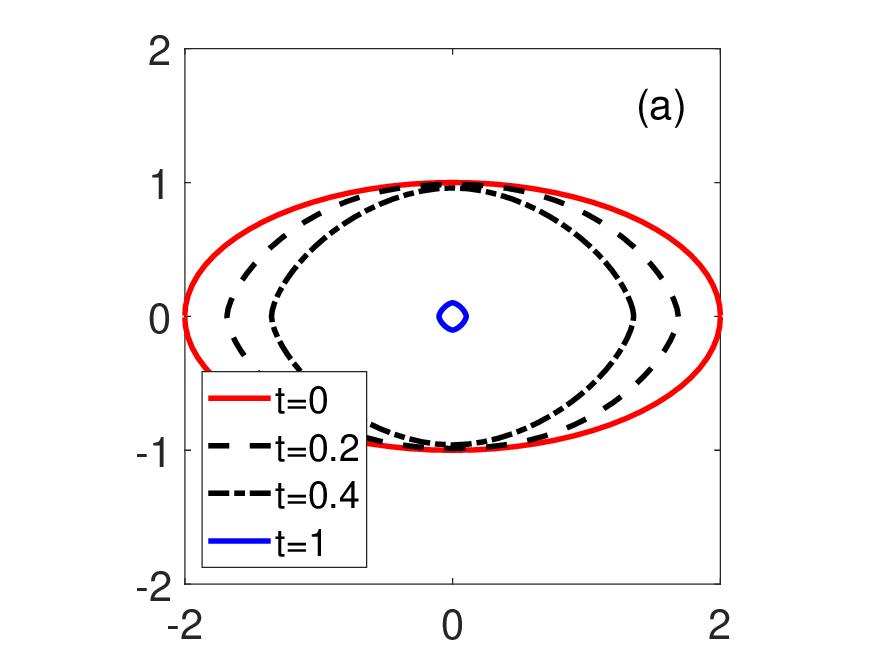}}
\subfigure[\empty]{ 
\includegraphics[trim=0 0 28 0,clip,width=0.30\textwidth,height=0.22\textwidth]
{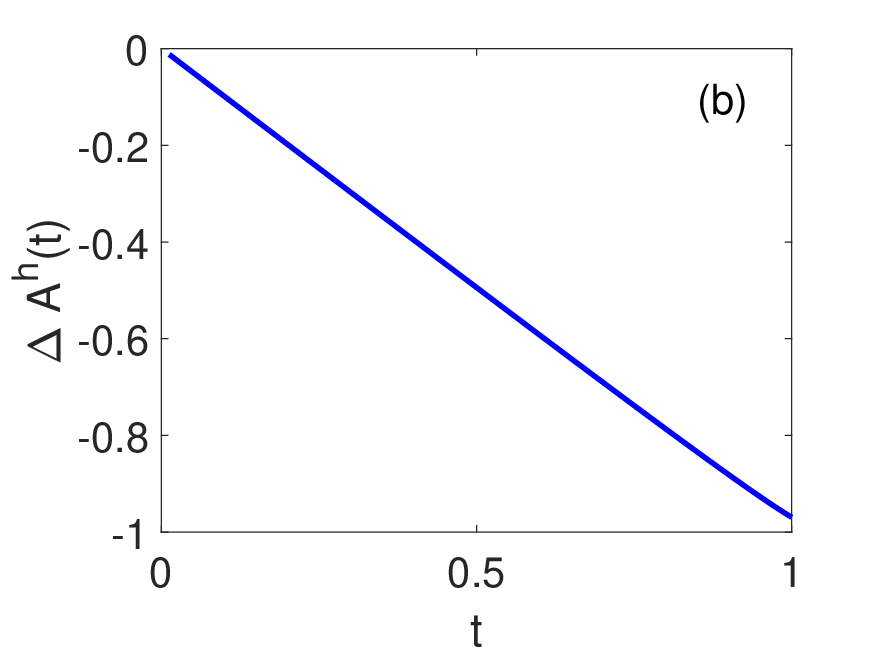}}
\subfigure[\empty]{ 
\includegraphics[trim=0 0 28 0,clip,width=0.30\textwidth,height=0.22\textwidth]
{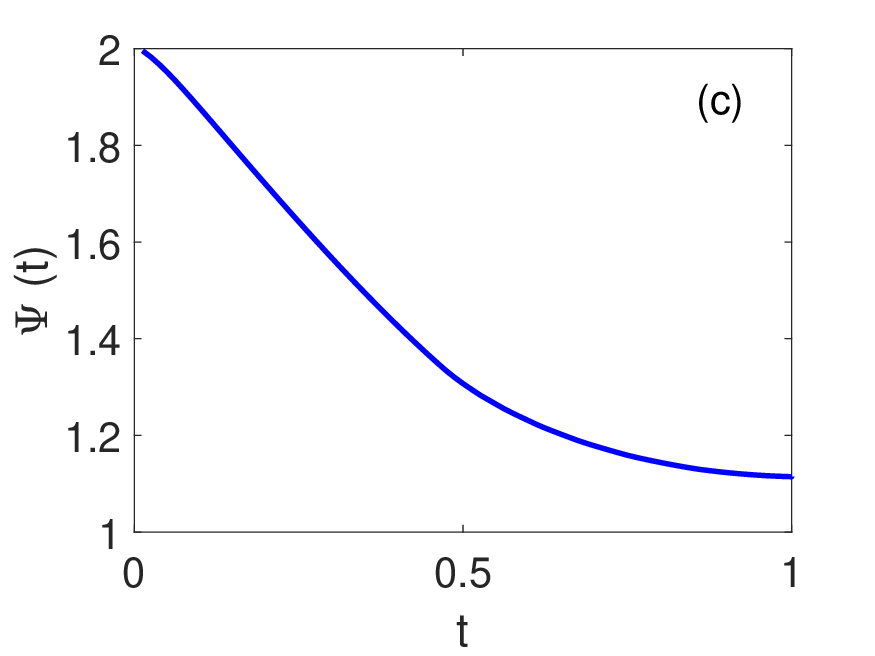}}\\\vspace{-4mm}
\subfigure[\empty]{ 
\includegraphics[trim=0 0 28 0,clip,width=0.30\textwidth,height=0.22\textwidth]
{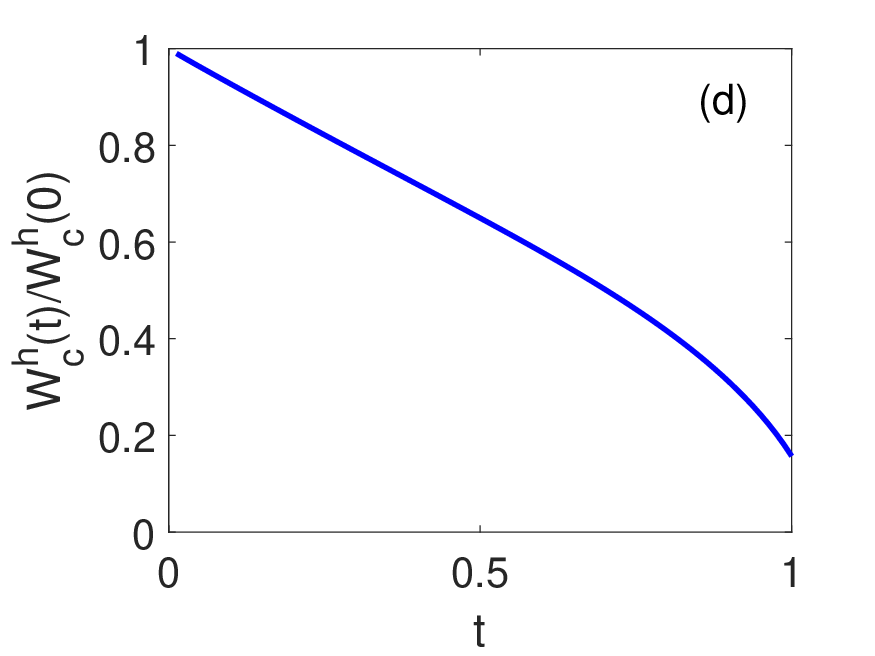}}
\subfigure[\empty]{ 
\includegraphics[trim=0 0 28 0,clip,width=0.30\textwidth,height=0.22\textwidth]
{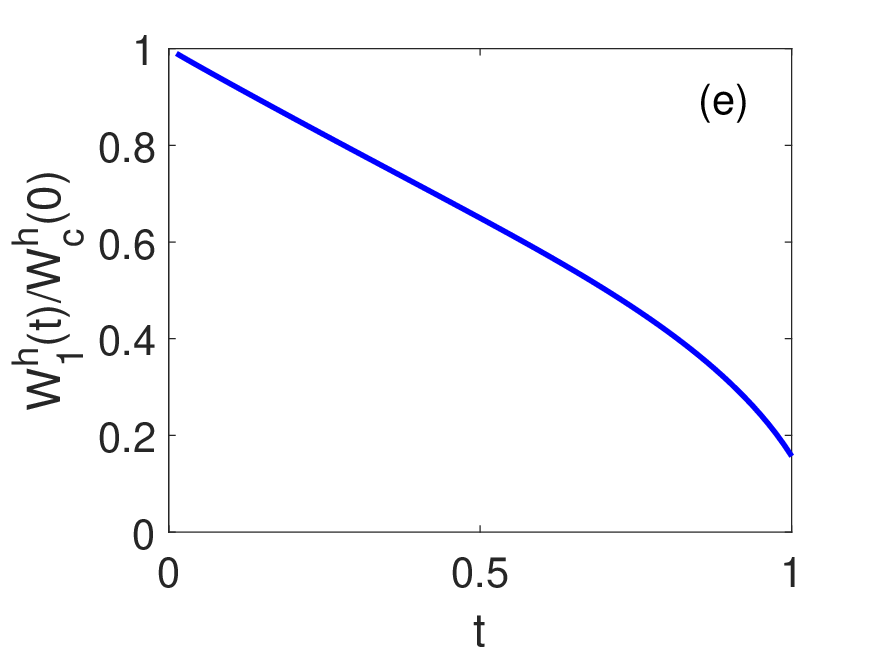}}
\subfigure[\empty]{ 
\includegraphics[trim=0 0 28 0,clip,width=0.30\textwidth,height=0.22\textwidth]
{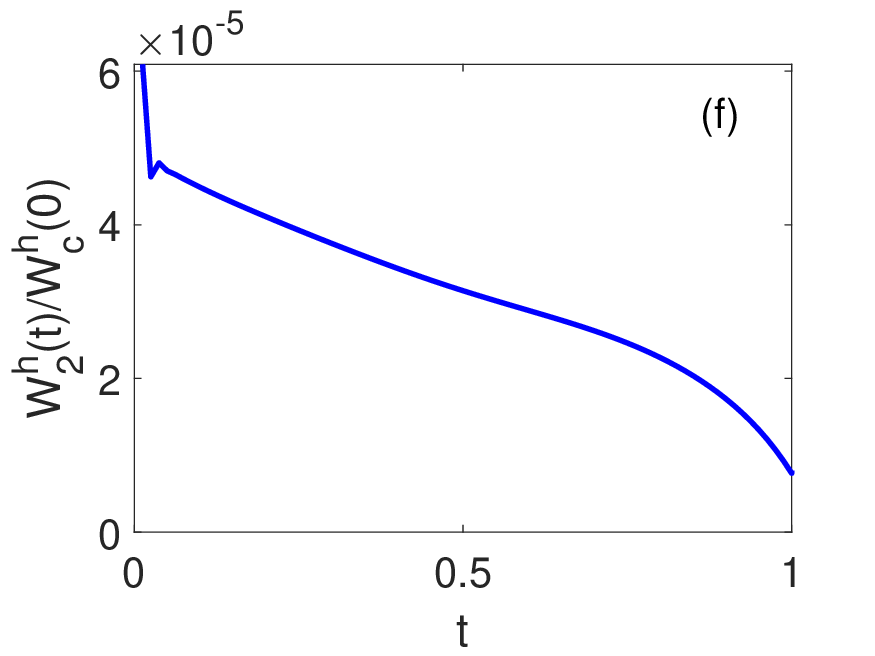}}
	\caption{Numerical simulation of the anisotropic CSF for an ellipse with $\gamma(\theta)=1+0.05\cos(4\theta)$. (a) Geometric evolution. Quantitative diagnostics: (b) relative area loss $\Delta A^h(t)$; (c) mesh ratio $\Psi(t)$; (d) normalized total energy $W^h_c(t)/W_c^h(0)$; (e) normalized energy component $W_1^h(t)/W_c^h(0)$; (f) component $W_2^h(t)/W_c^h(0)$.}
\label{fig6}
\end{figure}

\begin{figure}[h!]
	\centering
    \subfigure[\empty]{ 
\includegraphics[trim=10 0 55 0,clip,width=0.30\textwidth,height=0.22\textwidth]
{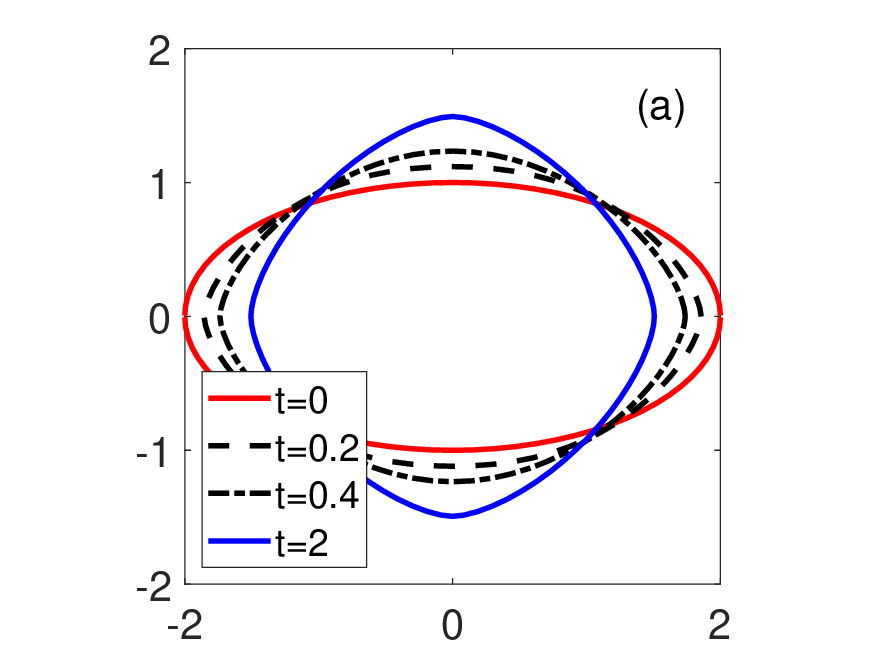}}
\subfigure[\empty]{ 
\includegraphics[trim=0 0 35 0,clip,width=0.30\textwidth,height=0.22\textwidth]
{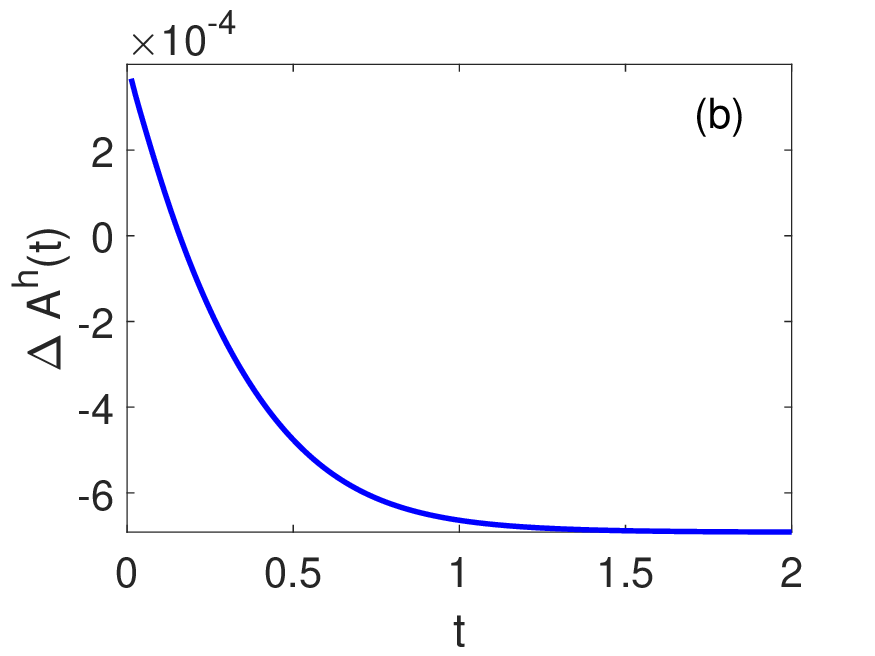}}
\subfigure[\empty]{ 
\includegraphics[trim=0 0 35 0,clip,width=0.30\textwidth,height=0.22\textwidth]
{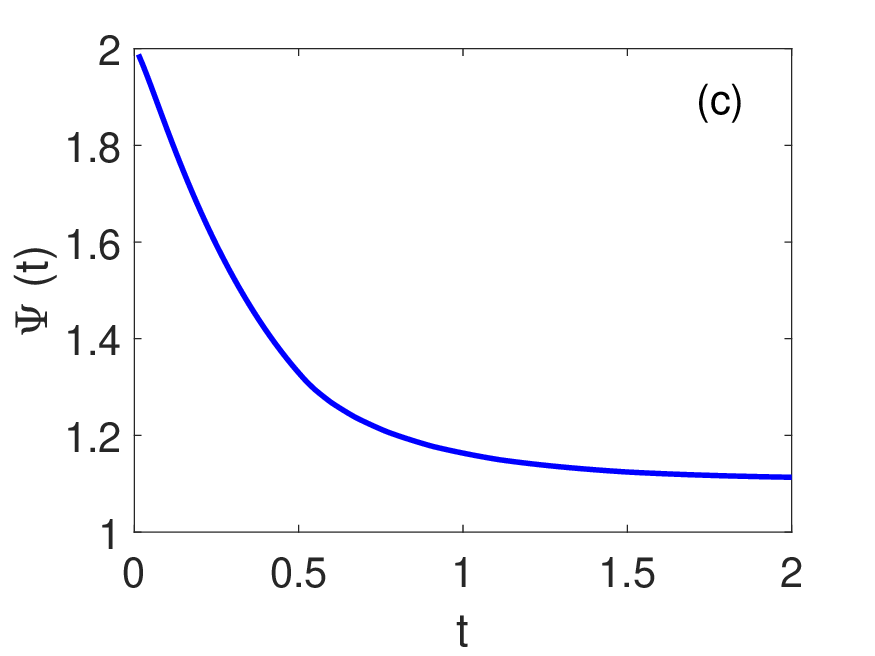}}\\
\vspace{-4mm}
\subfigure[\empty]{ 
\includegraphics[trim=0 0 35 0,clip,width=0.30\textwidth,height=0.22\textwidth]
{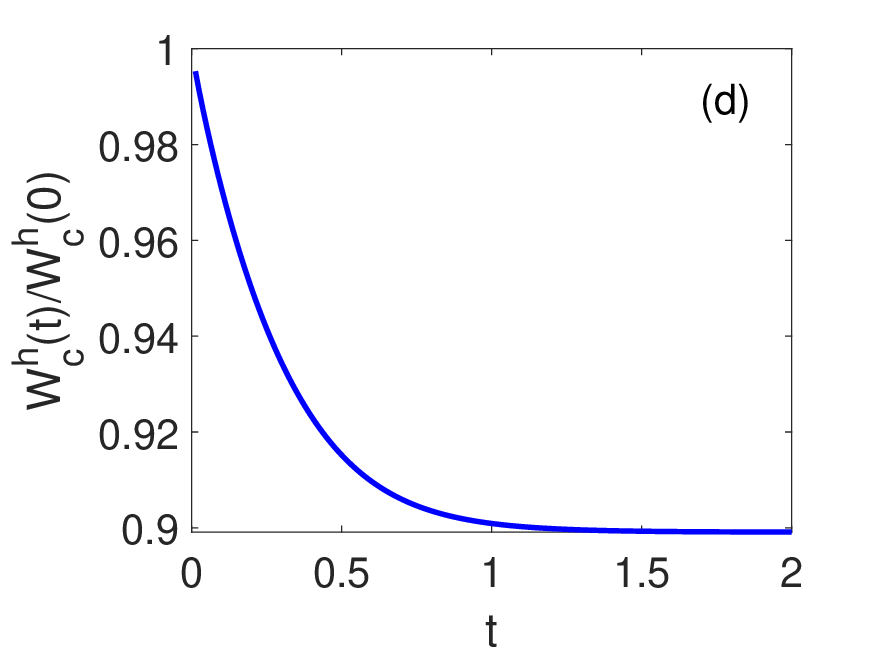}}
\subfigure[\empty]{ 
\includegraphics[trim=0 0 35 0,clip,width=0.30\textwidth,height=0.22\textwidth]
{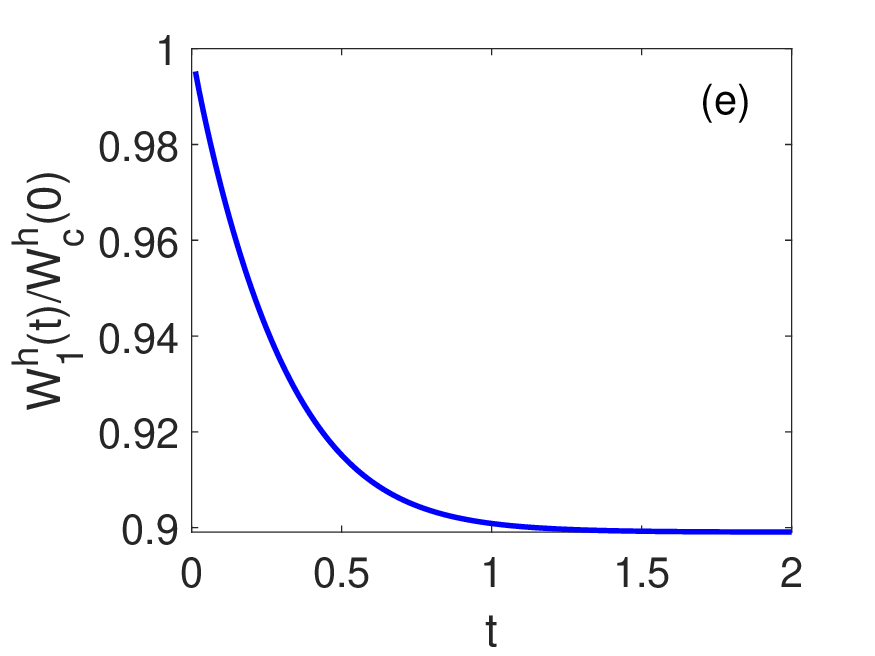}}
\subfigure[\empty]{ 
\includegraphics[trim=0 0 35 0,clip,width=0.30\textwidth,height=0.22\textwidth]
{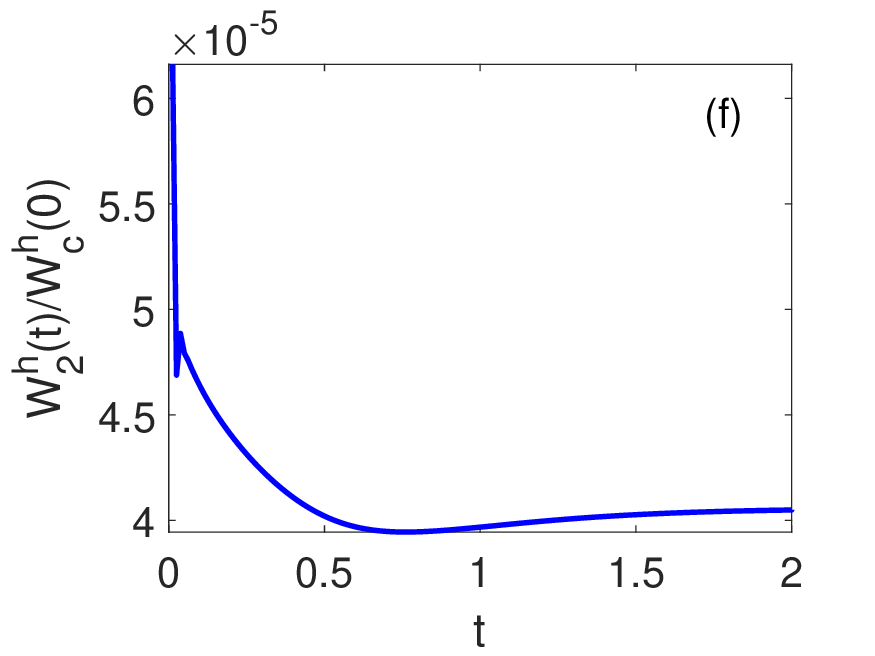}}
	\caption{Numerical simulation of the anisotropic AP-CSF for an ellipse with $\gamma(\theta)=1+0.05\cos(4\theta)$. (a) Geometric evolution. Quantitative diagnostics: (b) relative area loss $\Delta A^h(t)$; (c) mesh ratio $\Psi(t)$; (d) normalized total energy $W^h_c(t)/W_c^h(0)$; (e) normalized energy component $W_1^h(t)/W_c^h(0)$; (f) component $W_2^h(t)/W_c^h(0)$.}
\label{fig7}
\end{figure}

Figs. \ref{fig6}-\ref{fig7} present the CSF and AP-CSF evolutions of Curve 2, respectively, for the weak anisotropy $$\gamma(\theta)=1+0.05\cos(4\theta).$$
They also report the normalized energy, relative area loss, and mesh ratio over time. As in the isotropic case, the relative area loss for AP-CSF stabilizes near $7\times 10^{-4}$, and the mesh ratio remains below $1.2$ throughout the evolution. The discrete energy decreases in both flows. Figures~\ref{fig6}(f)--\ref{fig7}(f) further show that  $W_2^h/W_c^h(0)$ remains of order $10^{-5}$,  making a negligible contribution to the total energy. Consequently,
 $W_c^h$ and $W_1^h$ exhibit nearly identical behavior.

To examine the method beyond weak anisotropy, we next consider $$\gamma(\theta)=1+0.1\cos(4\theta)$$ with $\tau=10^{-4}$.
Figs.~\ref{fig010evolution}--\ref{fig010geometry} show the resulting anisotropic CSF and AP-CSF evolutions. The tangent-angle variation between consecutive time levels remains bounded, and the energy histories in Figs.~\ref{fig010geometry}(c1)--\ref{fig010geometry}(c2) decrease over time.

\begin{figure}[h!]
	\centering
\begin{tabular}{@{}c@{\hspace{0.018\textwidth}}c@{\hspace{0.018\textwidth}}c@{\hspace{0.018\textwidth}}c@{}}
\includegraphics[trim=50 0 75 0,clip,width=0.23\textwidth,height=0.23\textwidth]{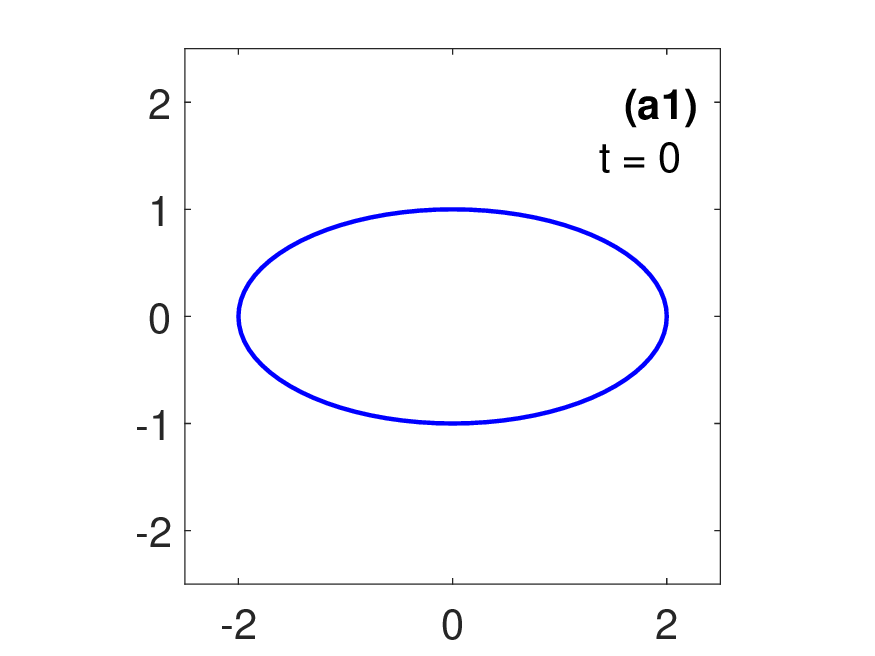}&
\includegraphics[trim=50 0 70 0,clip,width=0.23\textwidth,height=0.23\textwidth]{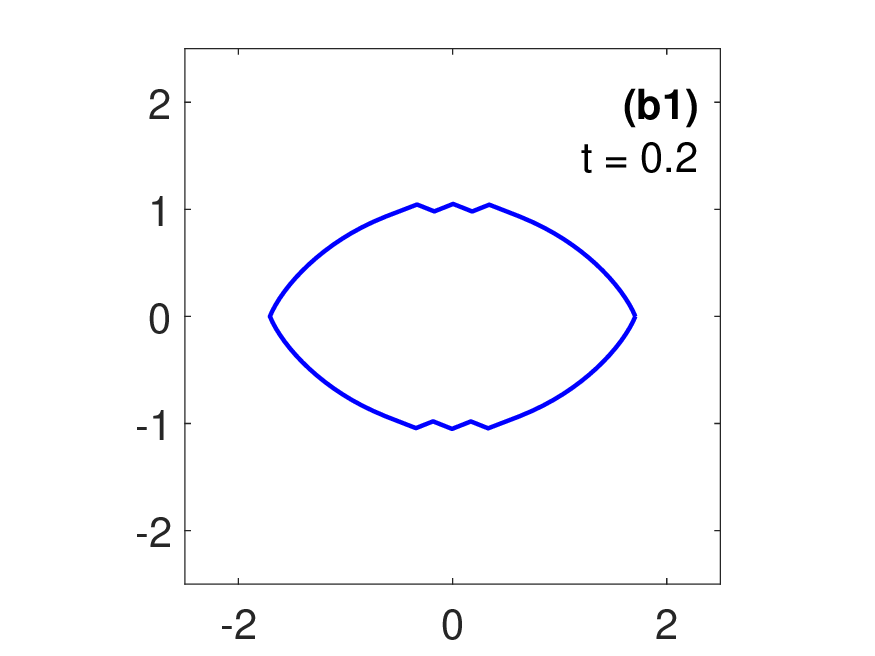}&
\includegraphics[trim=50 0 70 0,clip,width=0.23\textwidth,height=0.23\textwidth]{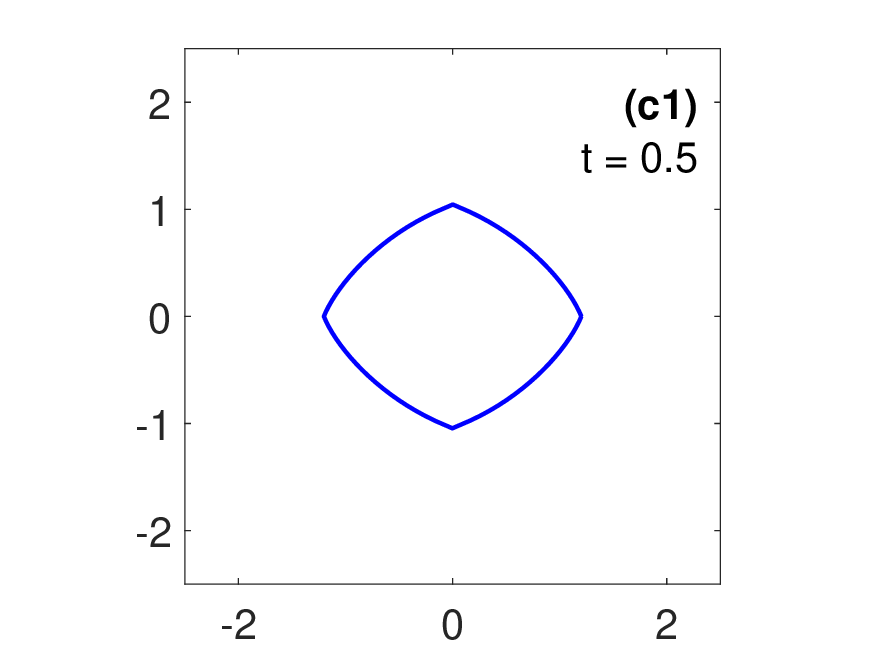}&
\includegraphics[trim=50 0 70 0,clip,width=0.23\textwidth,height=0.23\textwidth]{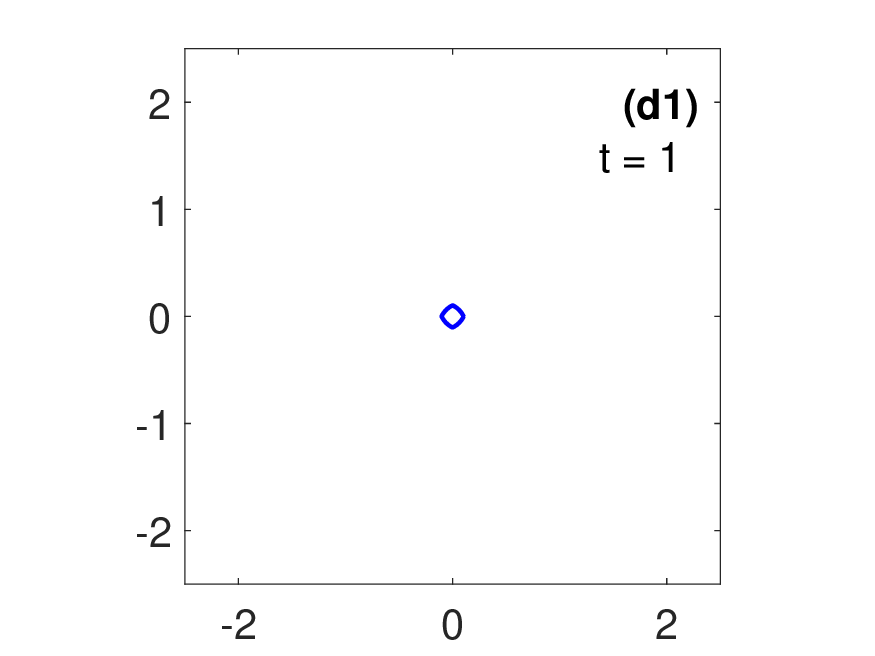}\\[-1mm]
\includegraphics[trim=50 0 70 0,clip,width=0.23\textwidth,height=0.23\textwidth]{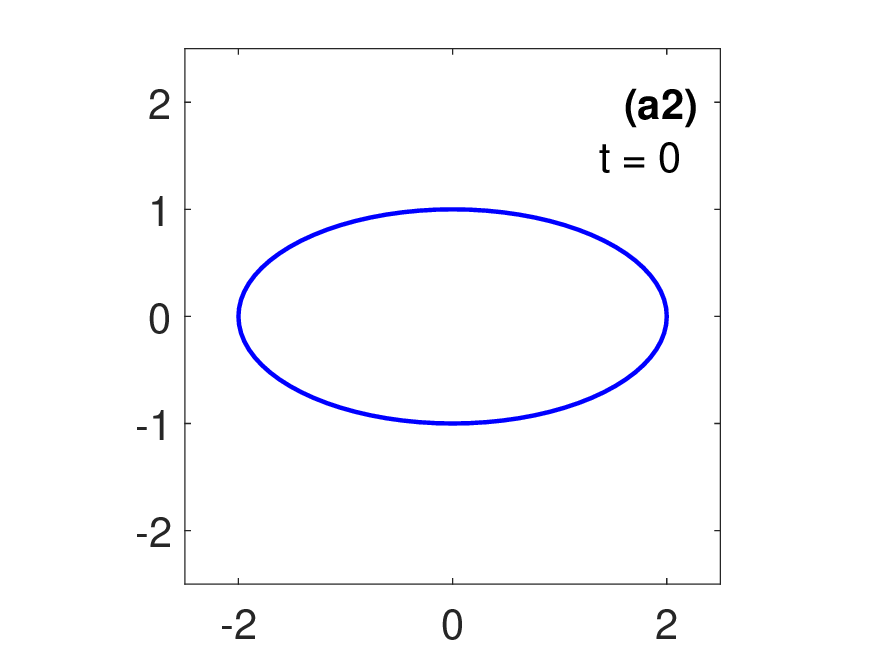}&
\includegraphics[trim=50 0 70 0,clip,width=0.23\textwidth,height=0.23\textwidth]{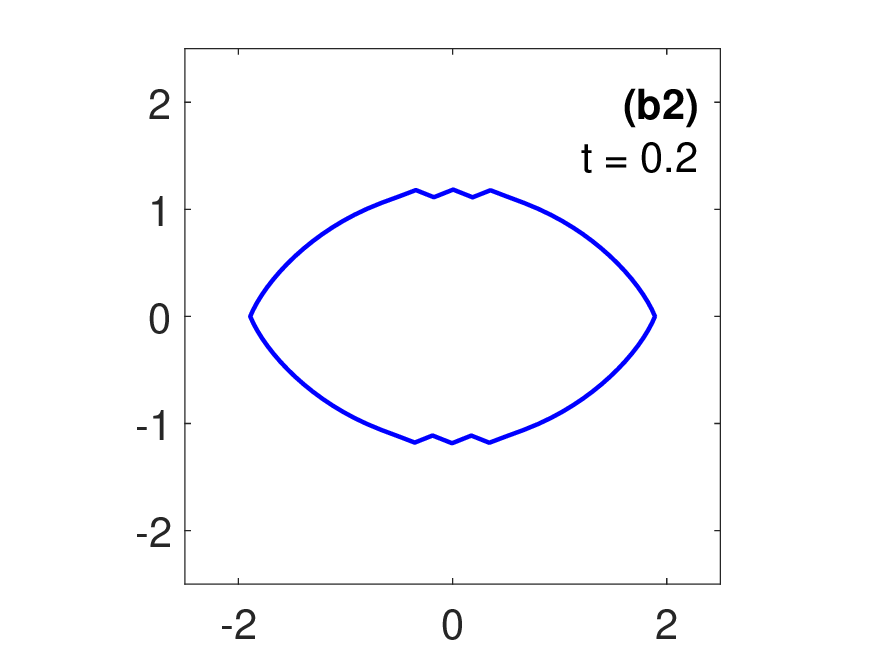}&
\includegraphics[trim=50 0 70 0,clip,width=0.23\textwidth,height=0.23\textwidth]{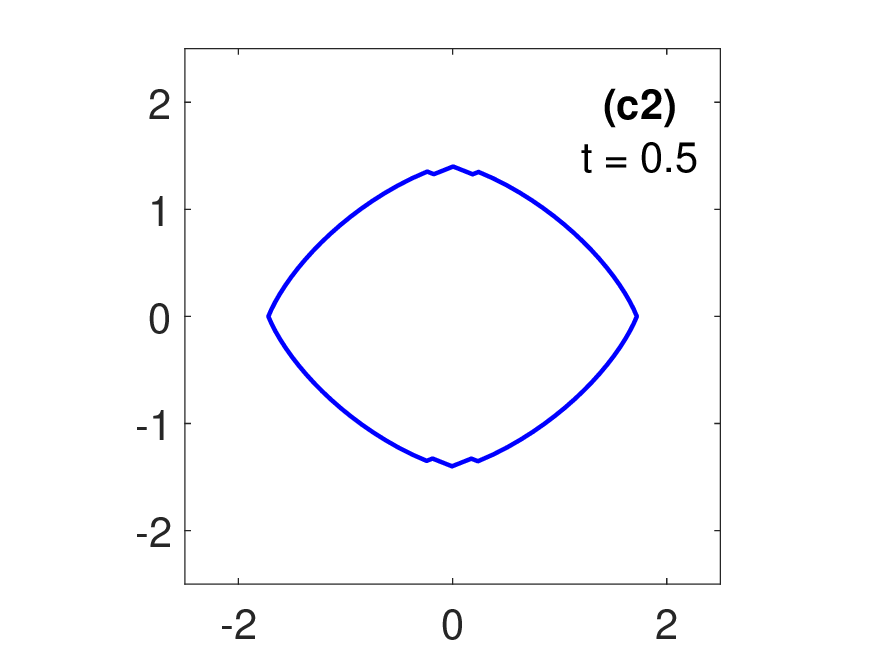}&
\includegraphics[trim=50 0 70 0,clip,width=0.23\textwidth,height=0.23\textwidth]{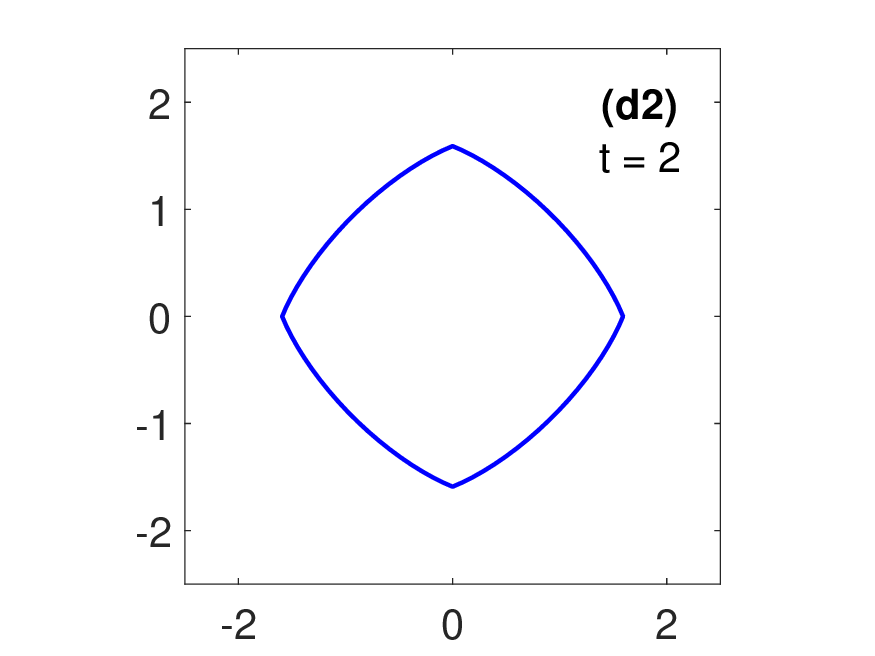}
\end{tabular}
\caption{Evolution under the anisotropic CSF (top row) and AP-CSF (bottom row) for an ellipse with $\gamma(\theta)=1+0.1\cos(4\theta)$.}
\label{fig010evolution}
\end{figure}

\begin{figure}[htbp!]
	\centering
\begin{tabular}{@{}c@{\hspace{0.025\textwidth}}c@{\hspace{0.025\textwidth}}c@{}}
\includegraphics[trim=5 0 35 0,clip,width=0.30\textwidth,height=0.25\textwidth]{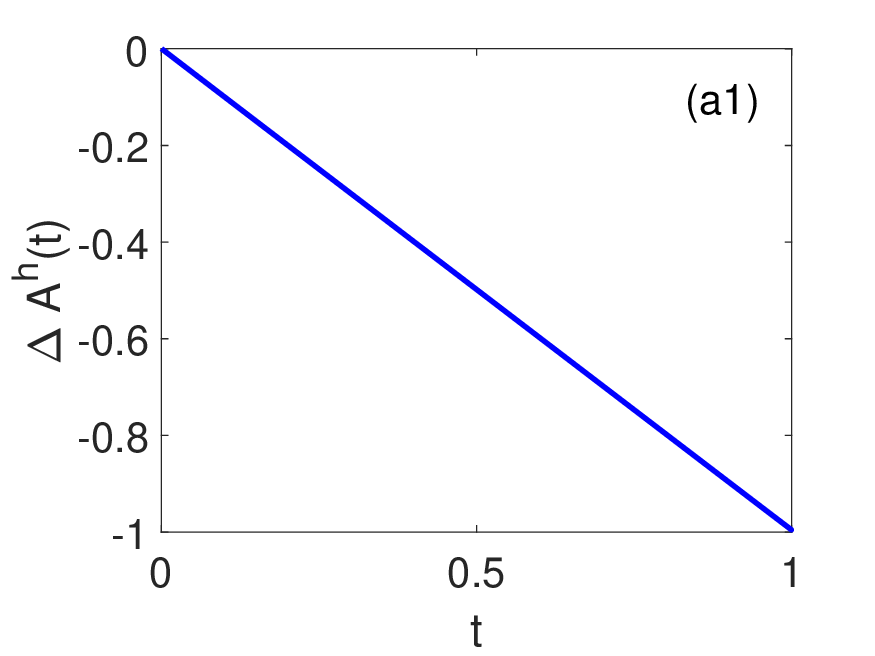}&
\includegraphics[trim=5 0 35 0,clip,width=0.30\textwidth,height=0.25\textwidth]{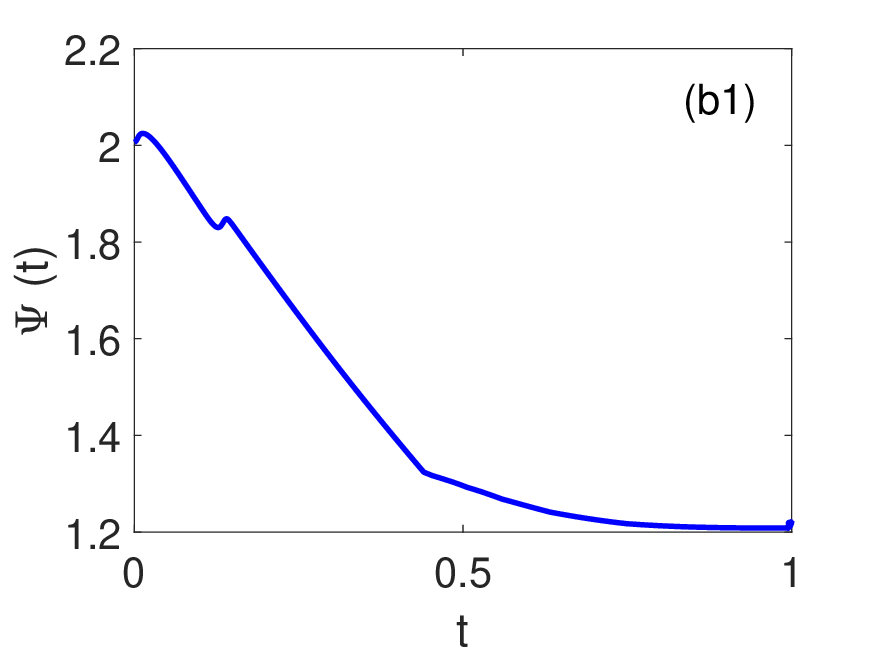}&
\includegraphics[trim=5 0 35 0,clip,width=0.30\textwidth,height=0.25\textwidth]{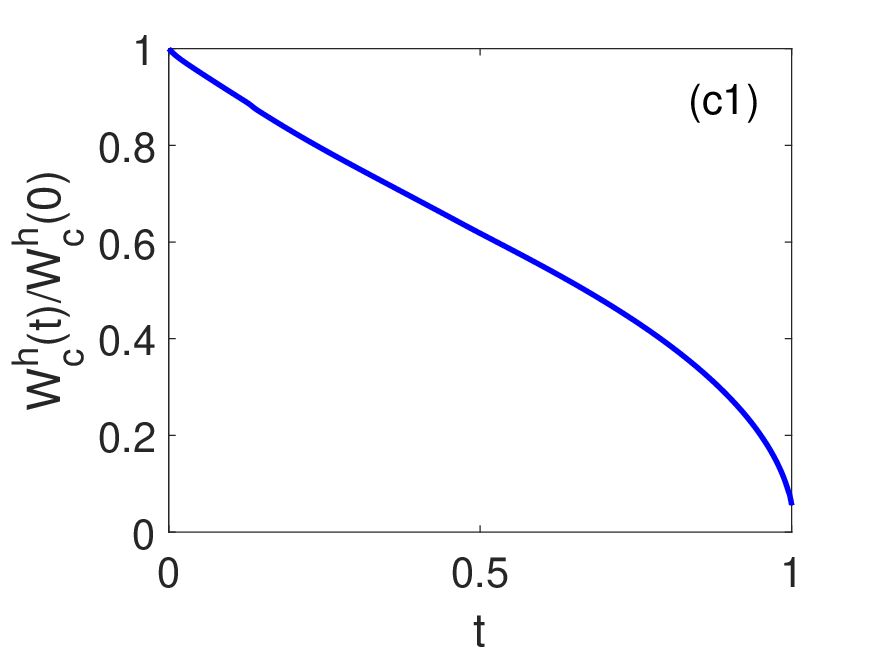}\\[-1mm]
\includegraphics[trim=5 0 35 0,clip,width=0.30\textwidth,height=0.25\textwidth]{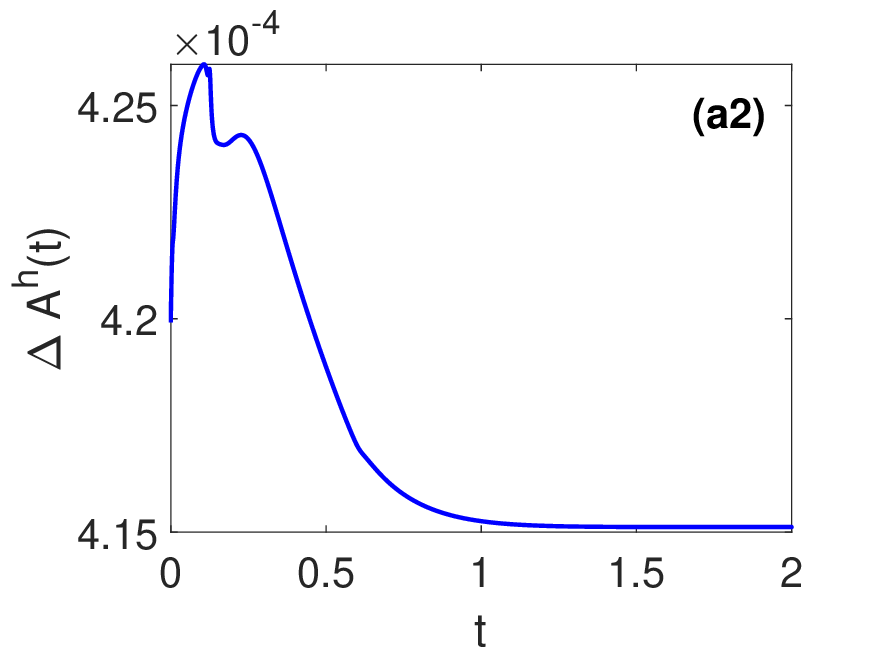}&
\includegraphics[trim=5 0 35 0,clip,width=0.30\textwidth,height=0.25\textwidth]{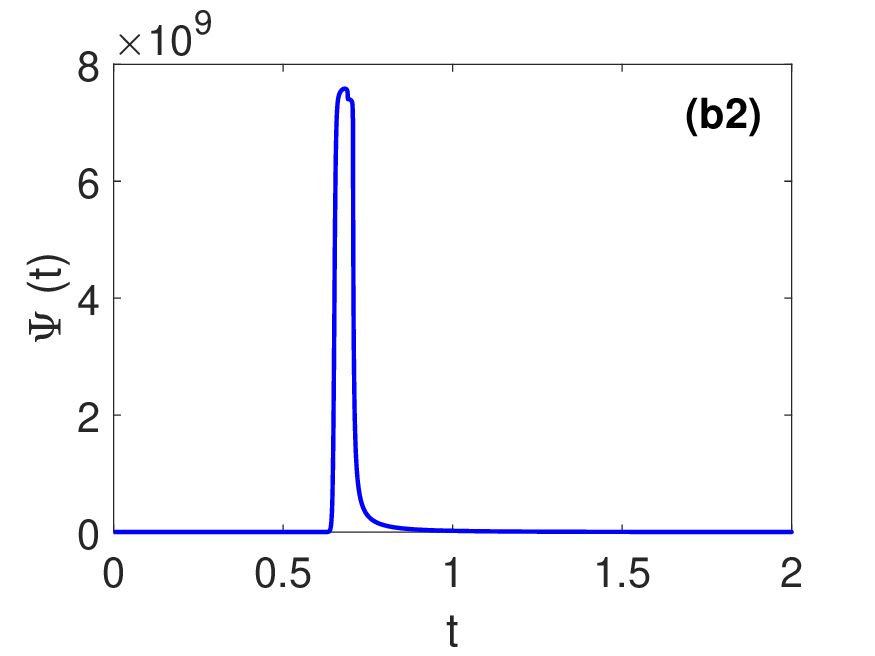}&
\includegraphics[trim=5 0 35 0,clip,width=0.30\textwidth,height=0.25\textwidth]{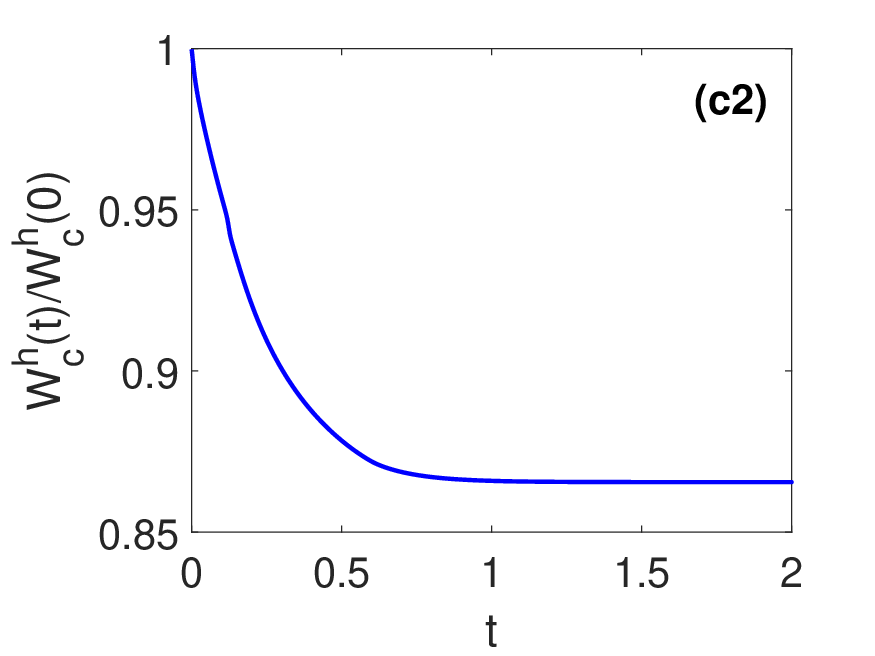}
\end{tabular}
\caption{Evolution of geometric quantities under the anisotropic CSF (top row) and AP-CSF (bottom row) for the ellipse with $\gamma(\theta)=1+0.1\cos(4\theta)$, showing from left to right: relative area loss $\Delta A^h(t)$, mesh ratio $\Psi(t)$, and normalized total energy $W^h_c(t)/W_c^h(0)$.}
\label{fig010geometry}
\end{figure}

\begin{figure}[htbp!]
	\centering
    \subfigure[\empty]{ 
\includegraphics[trim=55 0 73 0,clip,width=0.23\textwidth,height=0.22\textwidth]
{030insh1t00.eps}}
\subfigure[\empty]{ 
\includegraphics[trim=55 0 73 0,clip,width=0.23\textwidth,height=0.22\textwidth]
{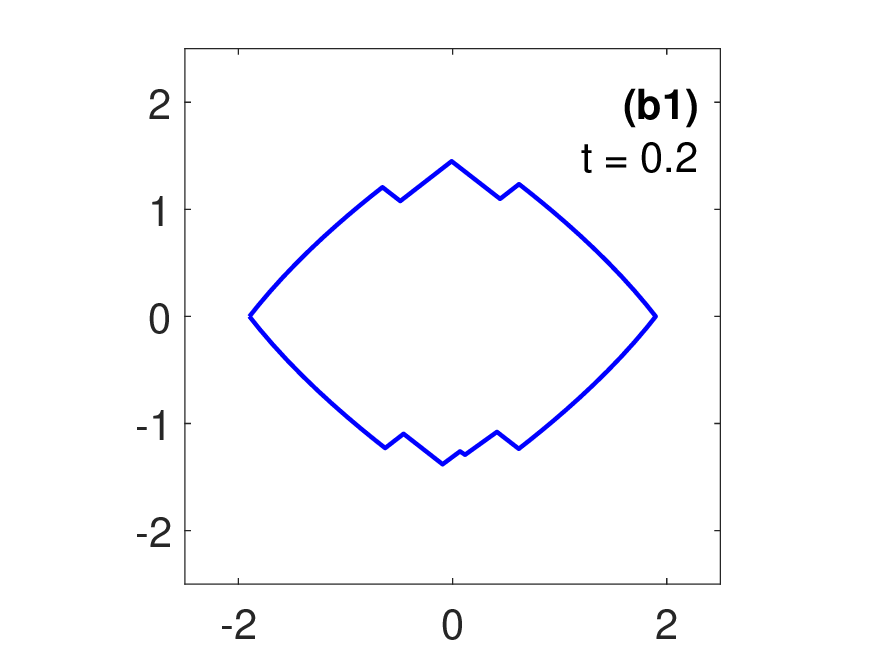}}
\subfigure[\empty]{ 
\includegraphics[trim=55 0 73 0,clip,width=0.23\textwidth,height=0.22\textwidth]
{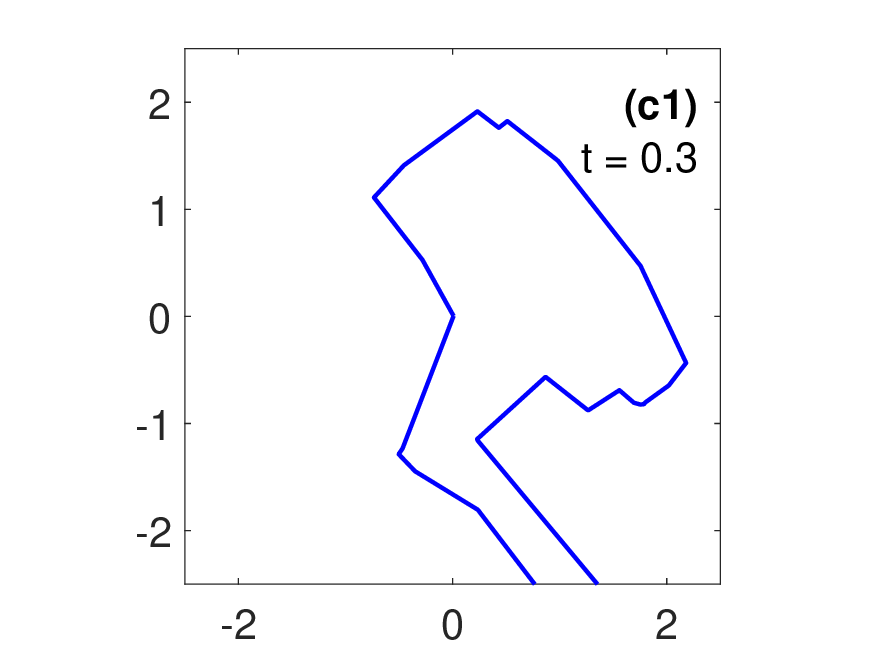}}
\subfigure[\empty]{ 
\includegraphics[trim=55 0 73 0,clip,width=0.23\textwidth,height=0.22\textwidth]
{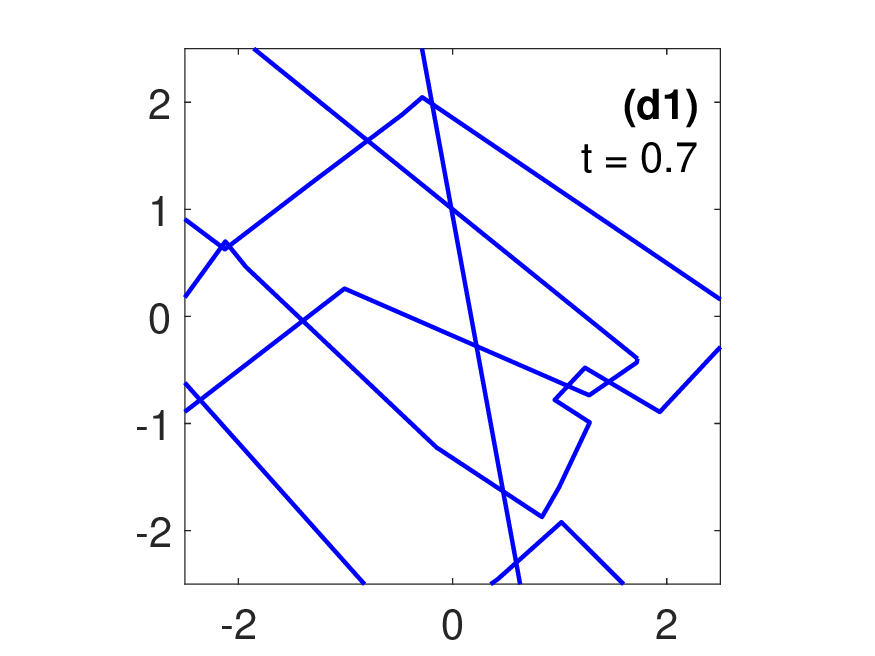}}\\\vspace{-3mm}
    \subfigure[\empty]{ 
\includegraphics[trim=55 0 73 0,clip,width=0.23\textwidth,height=0.22\textwidth]
{030insh1t00APCSF.eps}}
\subfigure[\empty]{ 
\includegraphics[trim=55 0 73 0,clip,width=0.23\textwidth,height=0.22\textwidth]
{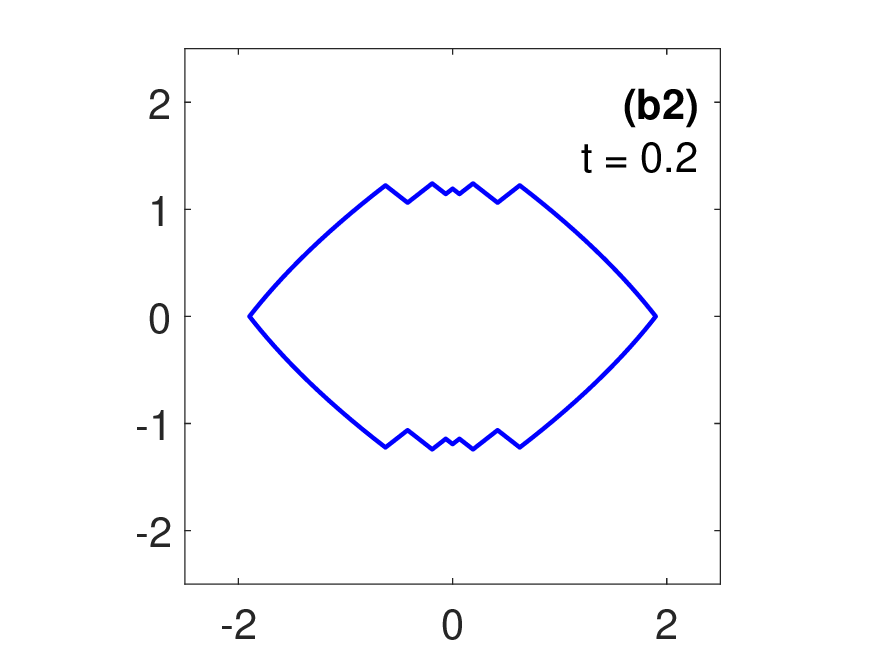}}
\subfigure[\empty]{ 
\includegraphics[trim=55 0 73 0,clip,width=0.23\textwidth,height=0.22\textwidth]
{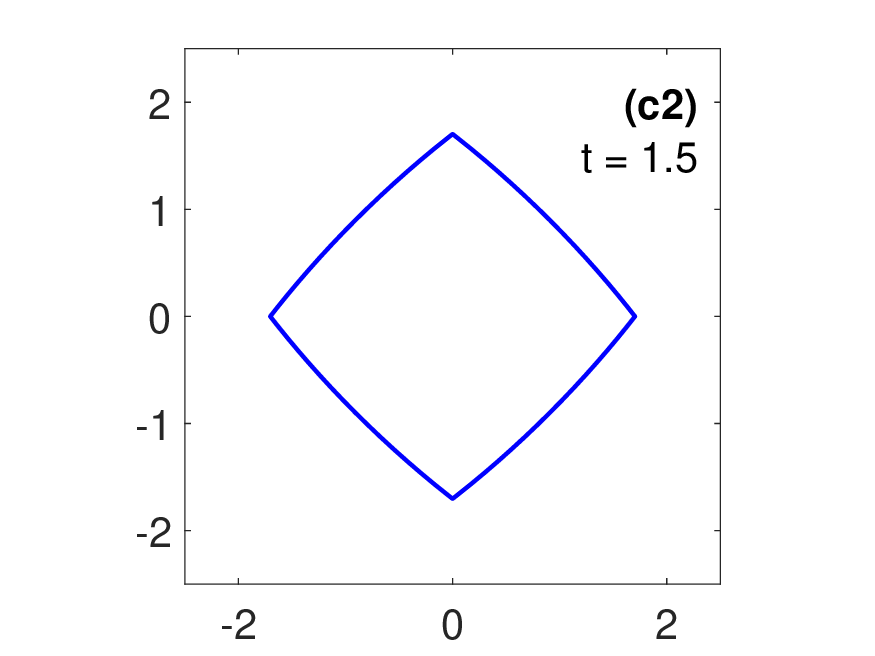}}
\subfigure[\empty]{ 
\includegraphics[trim=55 0 73 0,clip,width=0.23\textwidth,height=0.22\textwidth]
{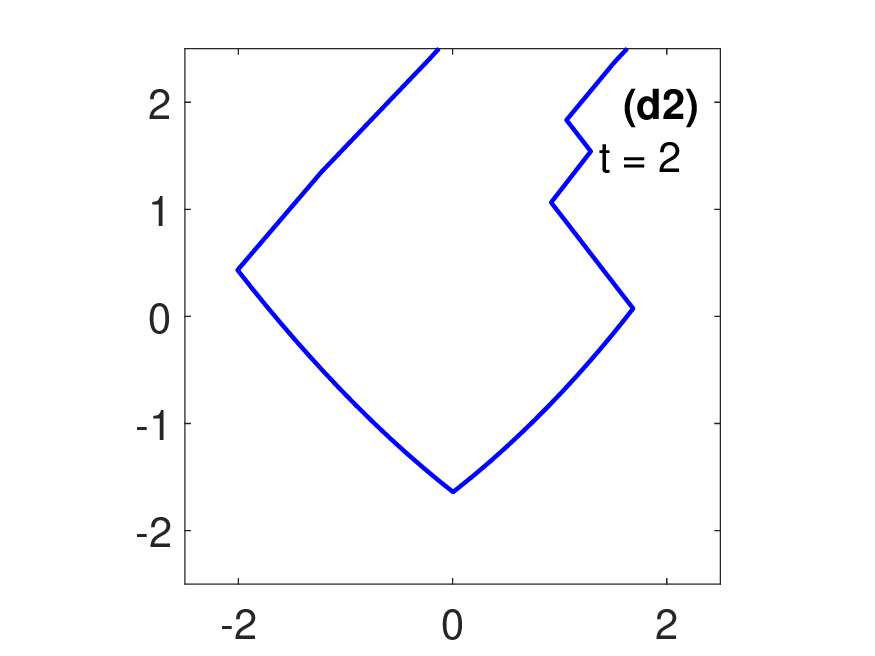}}\\\vspace{-3mm}
    \subfigure[\empty]{ 
\includegraphics[trim=55 0 73 0,clip,width=0.23\textwidth,height=0.22\textwidth]
{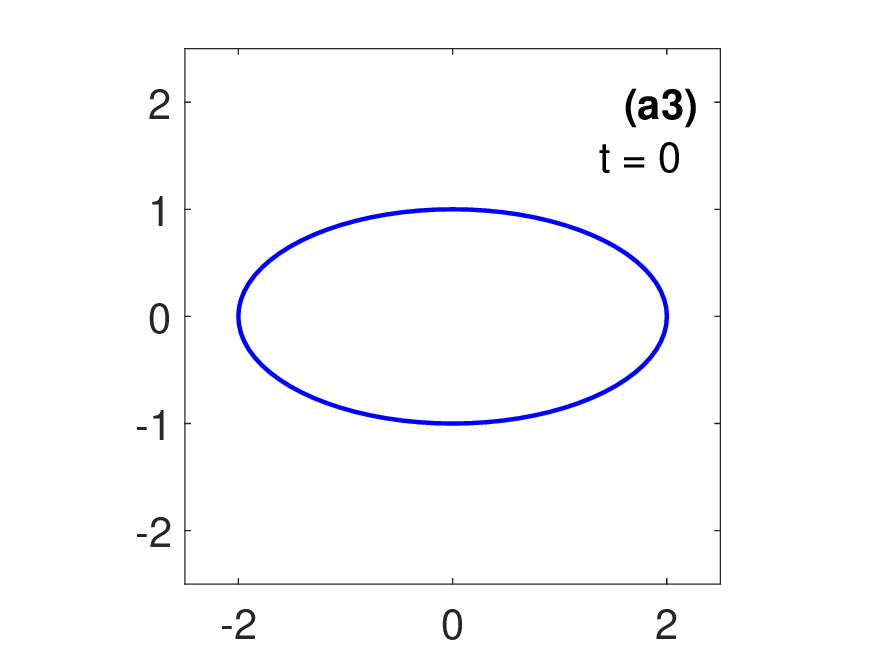}}
\subfigure[\empty]{ 
\includegraphics[trim=55 0 73 0,clip,width=0.23\textwidth,height=0.22\textwidth]
{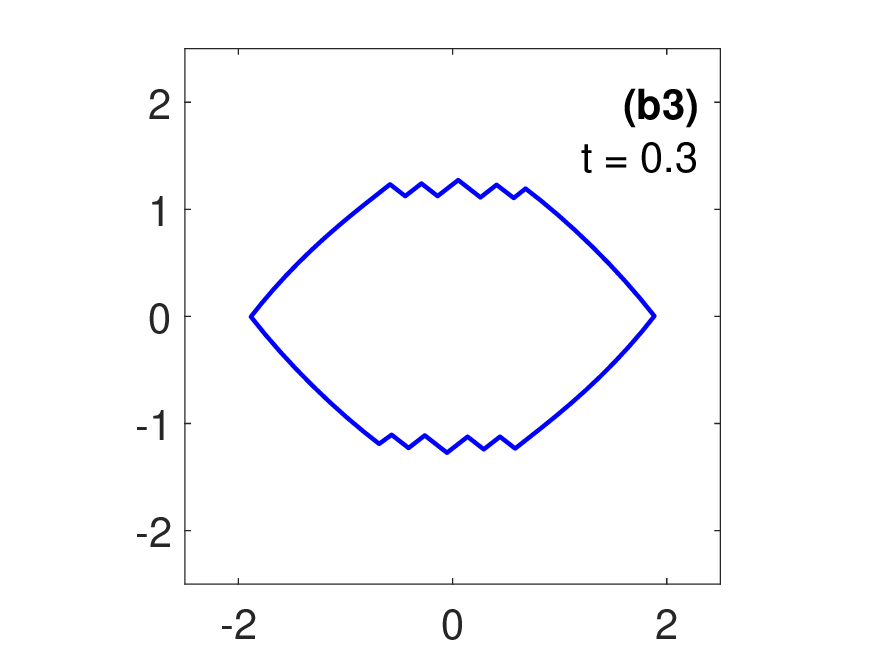}}
\subfigure[\empty]{ 
\includegraphics[trim=55 0 73 0,clip,width=0.23\textwidth,height=0.22\textwidth]
{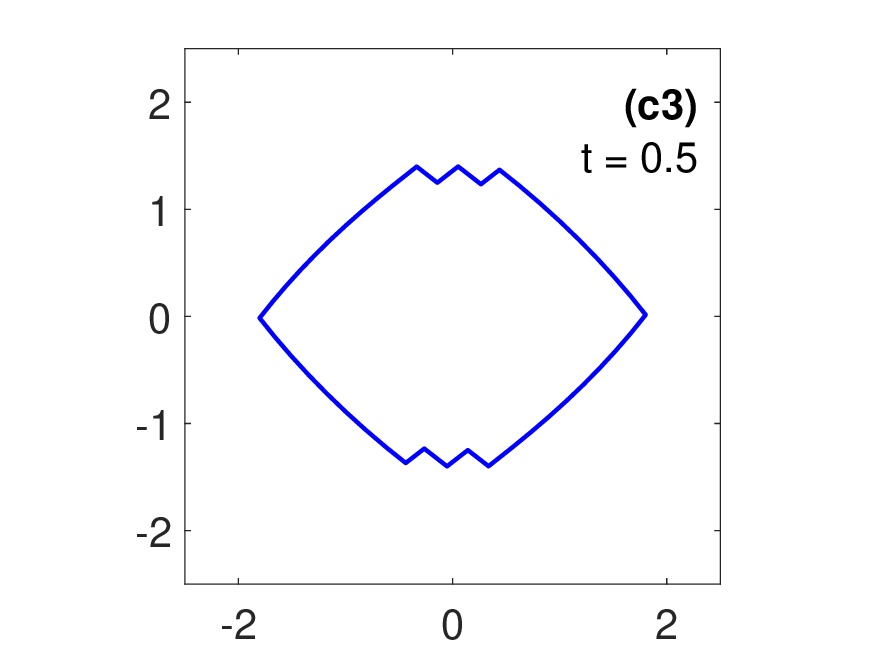}}
\subfigure[\empty]{ 
\includegraphics[trim=55 0 73 0,clip,width=0.23\textwidth,height=0.22\textwidth]
{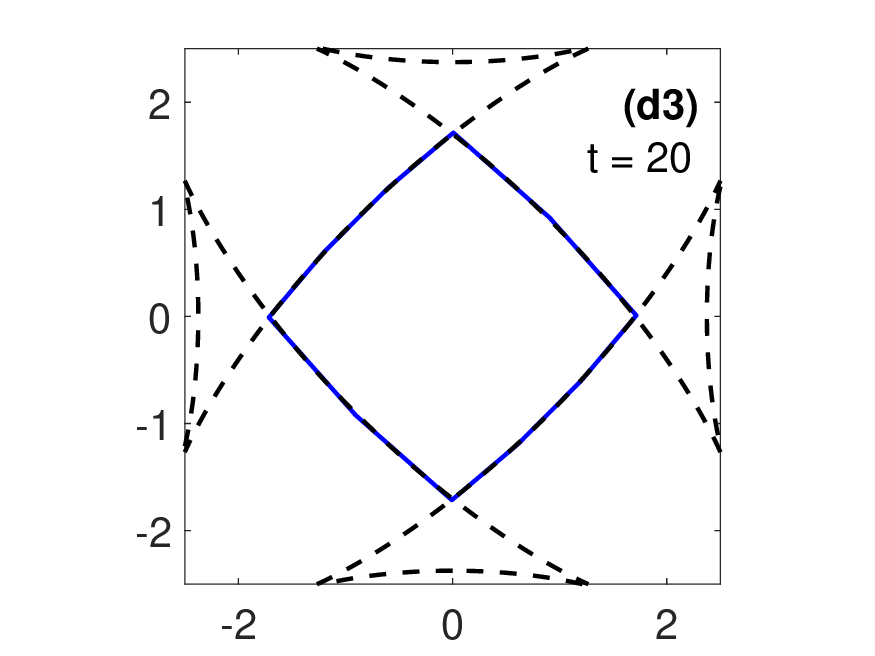}}
	\caption{Numerical simulation of the anisotropic AP-CSF for an ellipse with $\gamma(\theta)=1+0.3\cos(4\theta)$: (a1)-(d1) the ES-PFEM \cite{Li2021}; (a2)-(d2) the symmetrized PFEM \cite{Bao202361}; (a3)-(d3) the proposed LDG method.}
\label{fig8}
\end{figure}

\begin{figure}[htbp!]
	\centering
    \subfigure[\empty]{ 
\includegraphics[trim=0 0 10 5,clip,width=0.32\textwidth,height=0.3\textwidth]
{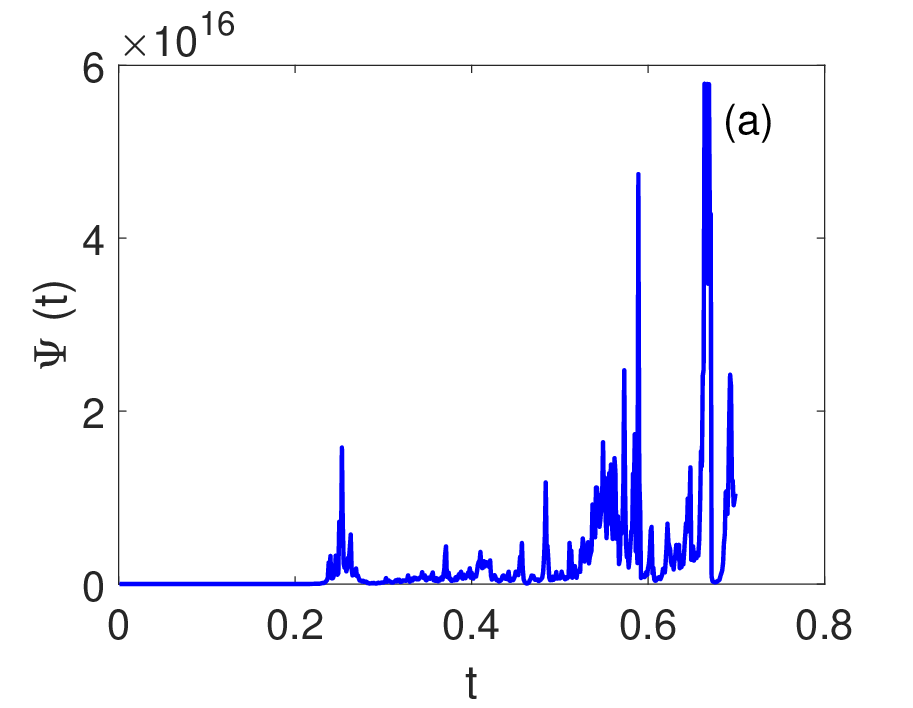}}
\subfigure[\empty]{ 
\includegraphics[trim=0 0 10 0,clip,width=0.32\textwidth,height=0.3\textwidth]
{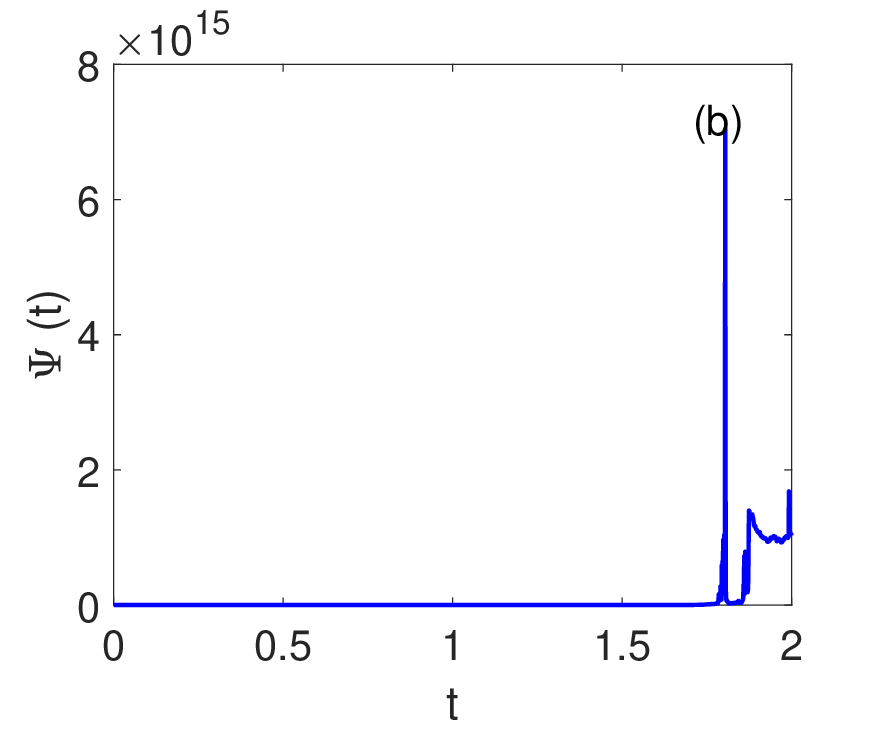}}
\subfigure[\empty]{ 
\includegraphics[trim=0 0 10 0,clip,width=0.32\textwidth,height=0.3\textwidth]
{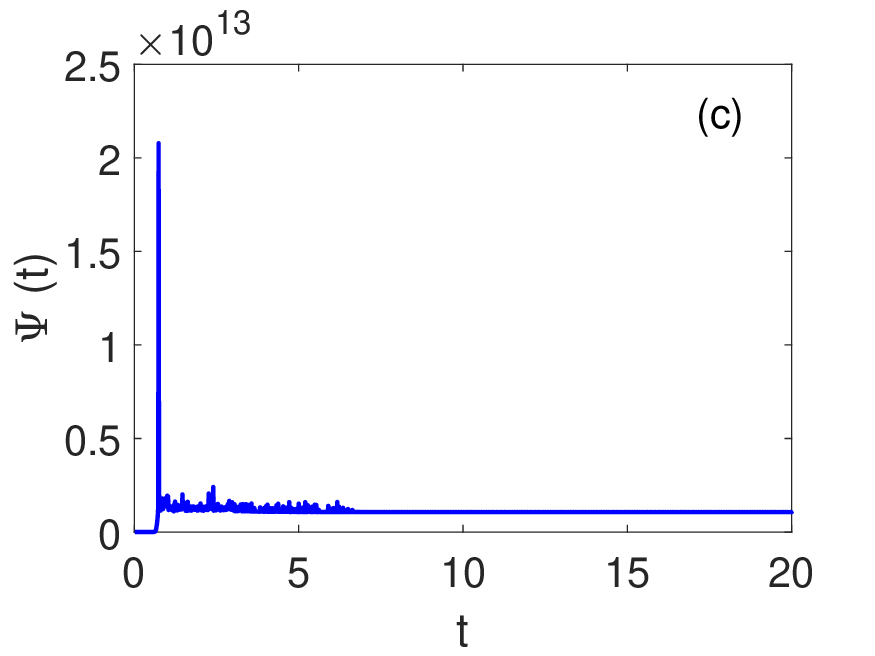}}\vspace{-3mm}
	\caption{Mesh ratio: (a) the ES-PFEM \cite{Li2021}; (b) the symmetrized PFEM \cite{Bao202361}; (c) the proposed LDG method.}
\label{fig9}
\end{figure}

For the stronger anisotropy
$$\gamma(\theta)=1+0.3\cos(4\theta),$$
Figs. \ref{fig8}-\ref{fig9} compare ES-PFEM~\cite{Li2021}, the symmetrized PFEM~\cite{Bao202361}, and the proposed LDG method with $k=1$. The numerical results indicate that ES-PFEM fails to maintain a stable evolution under such strong anisotropy (see Figs. \ref{fig8}(a1)-(d1)). The symmetrized PFEM initially approaches an apparent equilibrium, but the computed curve subsequently deteriorates as the mesh ratio increases sharply; see Figs. \ref{fig8}(a2)-(d2). The LDG solution retains geometric symmetry throughout the entire evolution, as shown in Figs.~\ref{fig8}(a3)--(d3), and its numerical equilibrium closely matches the area-matched Wulff shape, plotted as a black dashed line in Fig.~\ref{fig8}(d3).  Figure~\ref{fig9}(c) shows that the LDG computation continues even when the mesh ratio reaches approximately  $2\times 10^{13}$. This observation demonstrates that the computation can proceed despite severe mesh deterioration in this test.

\begin{figure}[h!]
	\centering
\subfigure[\empty]{ 
\includegraphics[trim=10 90 0 80,clip,width=0.99\textwidth]
{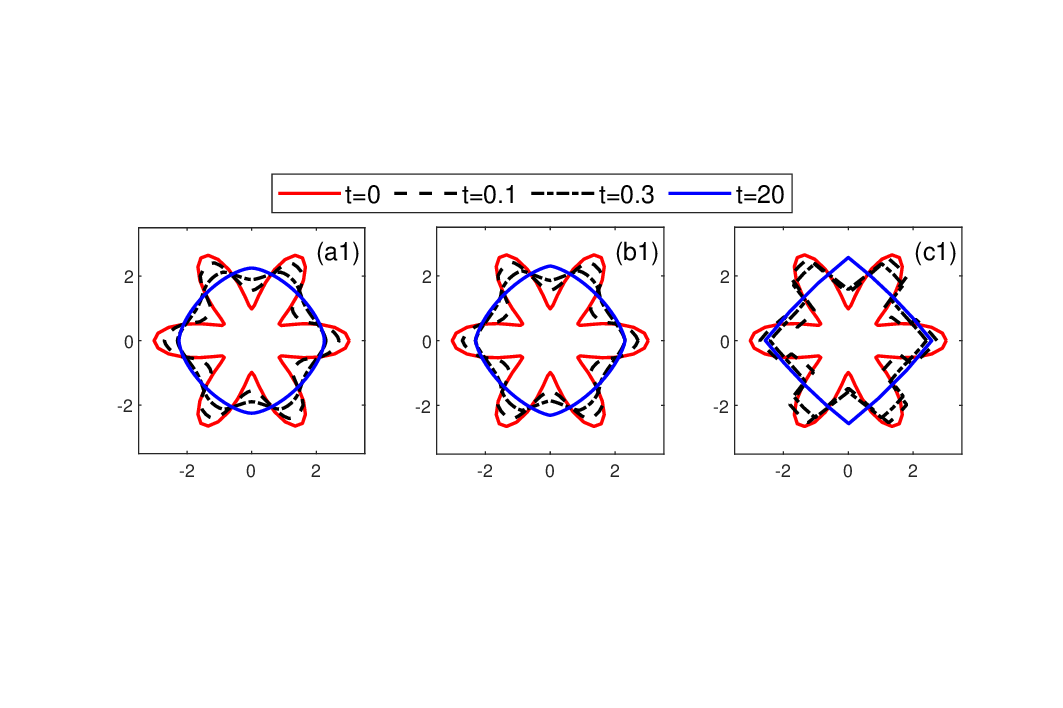}}\\\vspace{-6mm}
\subfigure[\empty]{ 
\includegraphics[trim=10 90 0 80,clip,width=0.99\textwidth]
{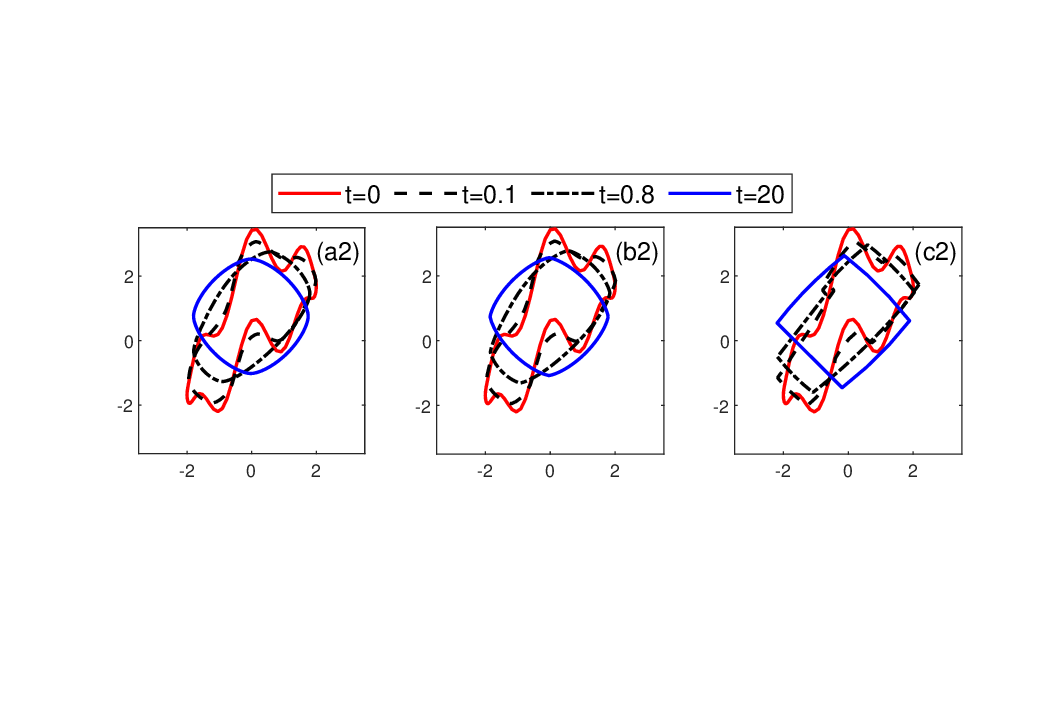}}
	\caption{Evolution of a flower (top) and Mikula (bottom) shapes under the anisotropic AP-CSF. Simulations correspond to  the energy density $\gamma(\theta)=1+\beta\cos(4\theta)$ with (a1)-(a2) $\beta=0.05$, (b1)-(b2) $\beta=0.067$, and (c1)-(c2) $\beta=0.3$. }
\label{fig12}
\end{figure}

\begin{figure}[h!]
    \centering
    \resizebox{0.98\textwidth}{!}{%
    \begin{tikzpicture}
        \node (a) at (-2.5,0) {
            \subfigure[\empty]{
                \includegraphics[trim=60 0 60 5,clip,width=0.29\textwidth]
                {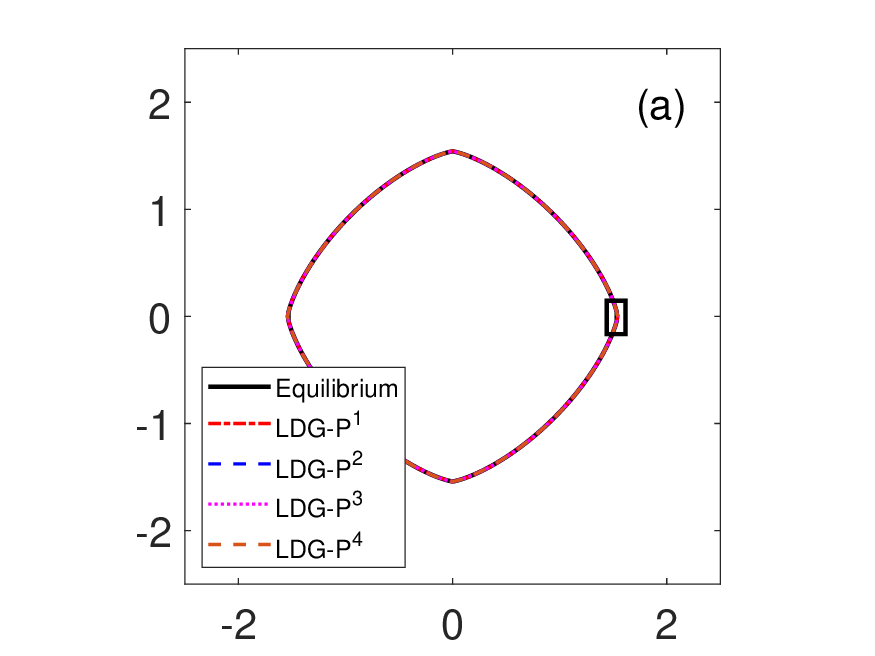}}
        };
        \node (b) at (1.7,0) {
            \subfigure[\empty]{
                \includegraphics[trim=60 0 60 5,clip,width=0.29\textwidth]
                {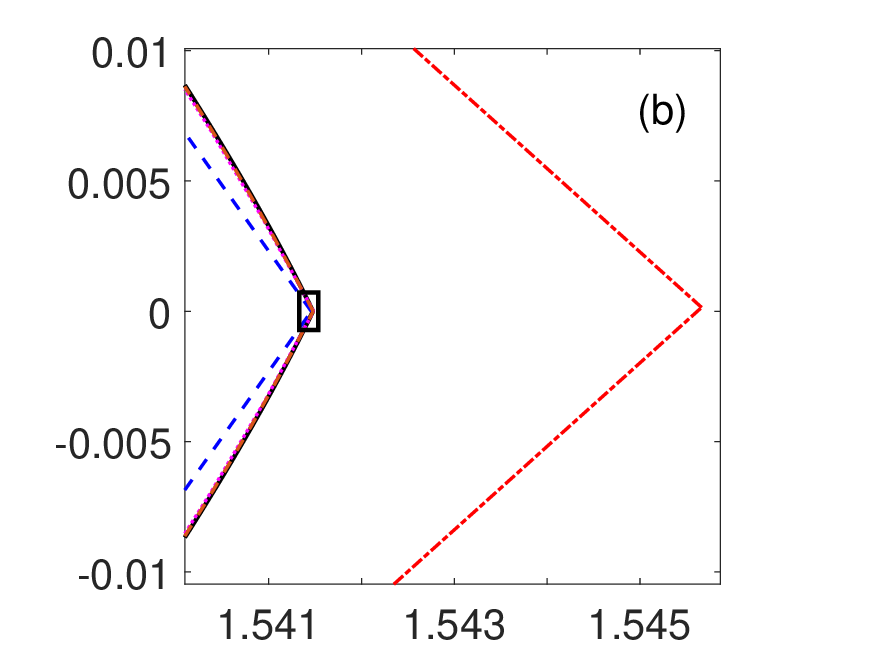}}
        };
        \node (c) at (5.8,0) {
            \subfigure[\empty]{
                \includegraphics[trim=60 0 60 5,clip,width=0.29\textwidth]
                {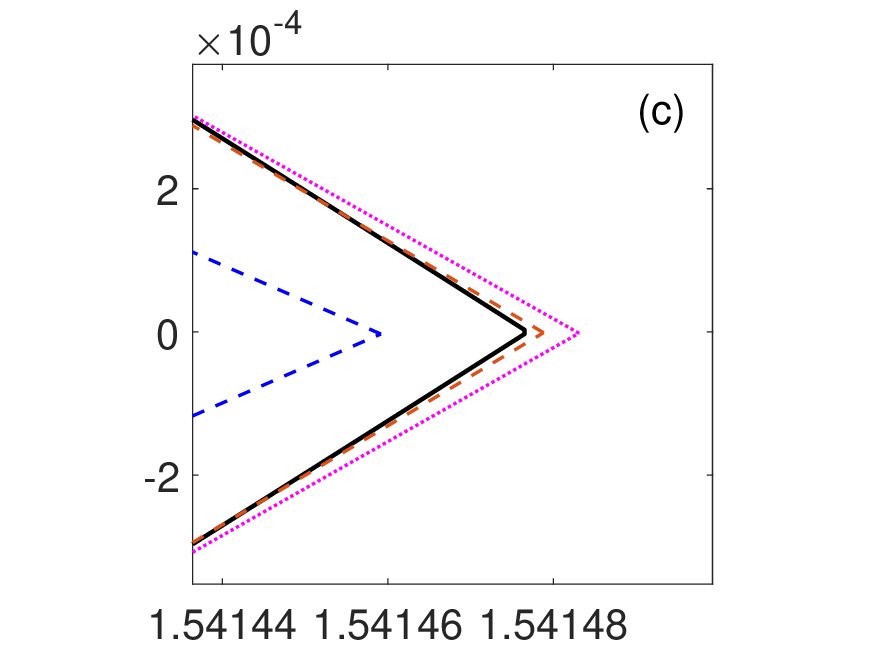}}
        };

        \node (arrow1) at (1.3,0) {
            \begin{tikzpicture}[baseline=-0.7ex]
                \coordinate (start) at ($(a.west)+(0,0.4)$);
                \coordinate (end) at ($(b.east)+(-6.5,0.4)$);
                \draw[->, >=stealth', line width=1.2pt, black] (start) -- (end);
                \node[above, font=\footnotesize] at (0.5,0.1) {};
            \end{tikzpicture}
        };

        \node (arrow2) at (4.3,0) {
            \begin{tikzpicture}[baseline=-0.7ex]
                \coordinate (start) at ($(a.west)+(0,0.4)$);
                \coordinate (end) at ($(b.east)+(-5.8,0.4)$);
                \draw[->, >=stealth', line width=1.2pt, black] (start) -- (end);
                \node[above, font=\footnotesize] at (0.5,0.1) {};
            \end{tikzpicture}
        };
    \end{tikzpicture}
    }
    \caption{Comparison between the numerically computed equilibrium shapes obtained using $P^k$ LDG elements and the theoretical equilibrium (i.e., the Wulff shape, shown as a solid black line). Panels (b)-(c) present zoomed-in views, with the strength of anisotropy set to $\beta=0.07$.}
    \label{fig13}
\end{figure}

Figure~\ref{fig12} examines the evolution of the flower and Mikula curves (Curves 3 and 4) for different anisotropy strengths. Increasing the anisotropy changes the final shapes substantially, producing more pronounced corners and flatter sides.
For $\gamma(\theta)=1+0.3\cos(4\theta)$, the equilibria exhibit distinct rhombic features. These shapes are qualitatively consistent with the anisotropic equilibrium shapes reported in~\cite{Bao202361,Bao202345}.

Finally, Fig.~\ref{fig13} compares the numerical equilibria with the corresponding area-matched Wulff shape. The numerical shapes agree qualitatively with the theoretical equilibrium, and the higher-order $P^k$ approximations show visibly closer agreement.

\section{Conclusions}
\label{sec:conclusion}

We have developed a family of high-order parametric local discontinuous Galerkin (LDG) methods for anisotropic curve-shortening flow and its area-preserving counterpart. Carefully designed consistent numerical fluxes yield unconditional energy dissipation for both semidiscrete schemes. Under mild assumptions, we also establish conditional energy dissipation and well-posedness for the fully discrete schemes. Numerical experiments demonstrate optimal spatial convergence of order
$O(h^{k+1})$ with $P^k$  approximations in isotropic and weakly anisotropic cases. They further show that the methods preserve  energy dissipation and remain robust under strong anisotropy, even with highly nonuniform meshes.

Future work will focus on higher-order time discretizations and extensions to other geometric flows.  Hyperbolic curvature flows, recently studied numerically by Deckelnick and N\"urnberg \cite{DeckelnickNuernberg2023Hyperbolic,DeckelnickNuernberg2025Surfaces}, offer another direction, requiring a separate formulation and analysis. We also plan to investigate computational efficiency and extend the framework to mean curvature flow and surface diffusion in $\mathbb{R}^3$ .

\begin{acknowledgements}
The work of Wei Jiang was partially supported by the National Natural Science Foundation of China (Nos. 12671530 and 12271414), and the work of Chunmei Su was supported by the National Key R\&D Program of China (2023YFA1008902),
and the National Natural Science Foundation of China (No. 12522118).
\end{acknowledgements}

\noindent\textbf{Data Availability} Data sharing not applicable to this article as no datasets were generated or analyzed during the current study.

\section*{Declarations}
	
\textbf{Conflict of interest} The author declares no conflict of interest.



\end{document}